\documentclass[12pt,twoside,reqno]{amsart}
\pdfoutput=1

\usepackage[main=english]{babel}
\usepackage[autostyle,english=american]{csquotes}
\MakeOuterQuote{"}

\usepackage{amssymb,mathtools}
\usepackage{bm}
\usepackage{xcolor}
\usepackage{dirtytalk}

\usepackage[left=2cm,right=2cm]{geometry}

\usepackage{amsthm,thmtools,amsmath}

\usepackage[style=numeric-comp,url=false,isbn=false,maxnames=6, backend=biber]{biblatex}
\AtEveryBibitem{\clearlist{language}}

\usepackage[final]{graphicx}
\graphicspath{{/img}}
\usepackage{subcaption}
\usepackage{tikz}
\usetikzlibrary{babel,cd,shapes,3d,arrows,decorations.markings}
\tikzcdset{arrow style=math font}
\usetikzlibrary{decorations.pathmorphing,shapes}
\usetikzlibrary{decorations.pathreplacing}

\tikzset{
    cross/.style=
    {cross out, draw, solid, thin, 
    minimum size=1.4*(#1-\pgflinewidth), 
    inner sep=0pt, outer sep=0pt, rotate=20},
    cross/.default={3}
} 

\tikzset{
    circlecross/.style= 
    {circle, minimum width=8pt, draw, inner sep=0pt, path picture={\draw (path picture bounding box.south east) -- (path picture bounding box.north west) (path picture bounding box.south west) -- (path picture bounding box.north east);}},
    circlecross/.default={3}
} 

\tikzset{
    mid arrow/.style=
    {postaction={decorate,decoration={markings,mark=at position .5 with {\arrow[xshift=2pt,#1]{stealth}}}}},
} 

\tikzset{
    crosshor/.style=
    {cross out, draw, solid, thin, 
    minimum size=1.4*(#1-\pgflinewidth), 
    inner sep=0pt, outer sep=0pt, rotate=0},
    crosshor/.default={3}
}

\newcommand{\nodalconic}[1]{
    \begin{scope}[shift={#1}]
            \draw[dashed] 
                (-0.5,-0.7) arc [x radius=0.5,y radius=0.11,start angle=180,end angle=0];
            \draw 
                (-0.5,0.7) arc [x radius=0.5,y radius=0.1,start angle=180,end angle=0]
                (-0.5,-0.7) arc [x radius=0.5,y radius=0.1,start angle=180,end angle=360]
                (-0.5,0.7) arc [x radius=0.5,y radius=0.1,start angle=180,end angle=360];
            \fill[lightgray, opacity=0.25] 
                (-0.5,0.7) arc[x radius=0.5,y radius=0.1,start angle=180,end angle=0] -- (0,0) -- (-0.5,0.7);
            \fill[lightgray, opacity=0.25] 
                (-0.5,0.7) arc[x radius=0.5,y radius=0.1,start angle=180,end angle=360] -- (0,0) -- (-0.5,0.7);
            \fill[lightgray, opacity=0.25] 
                (-0.5,-0.7) arc[x radius=0.5,y radius=0.1,start angle=180,end angle=0] -- (0,0) -- (-0.5,-0.7);
            \fill[lightgray, opacity=0.25] 
                (-0.5,-0.7) arc[x radius=0.5,y radius=0.1,start angle=180,end angle=360] -- (0,0) -- (-0.5,-0.7);
            \draw 
                (-0.5,-0.7) -- (0.5,0.7)
                (-0.5,0.7) -- (0.5,-0.7);
    \end{scope}
}

\newcommand{\conic}[1]{
    \begin{scope}[shift={#1}]
            \draw[dashed] 
                (-0.5,-0.7) arc [x radius=0.5,y radius=0.1,start angle=180,end angle=0];
            \draw
                (-0.5,0.7) arc [x radius=0.5,y radius=0.1,start angle=180,end angle=0]
                (-0.5,0.7) arc [x radius=0.5,y radius=0.1,start angle=180,end angle=360]
                (-0.5,-0.7) arc [x radius=0.5,y radius=0.1,start angle=180,end angle=360];
            \fill[lightgray, opacity=0.25] 
                (-0.5,-0.7) arc [x radius=0.5,y radius=0.1,start angle=180,end angle=360] -- (0.5,-0.7) to[out=125,in=270] (0.25,0) to[out=90, in=-125] (0.5,0.7)  -- cycle;
            \fill[lightgray, opacity=0.25] 
                (-0.5,-0.7) to[out=55,in=270] (-0.25,0) to[out=90,in=305] (-0.5,0.7) arc[x radius=0.5,y radius=0.1,start angle=180,end angle=360] -- cycle;
            \fill[lightgray, opacity=0.25] 
                (-0.5,-0.7) to[out=55,in=270] (-0.25,0) to[out=90,in=305] (-0.5,0.7) arc[x radius=0.5,y radius=0.1,start angle=180,end angle=0] -- cycle;
            \fill[lightgray, opacity=0.25] 
                (-0.5,-0.7) arc [x radius=0.5,y radius=0.1,start angle=180,end angle=0] -- (0.5,-0.7) to[out=125,in=270] (0.25,0) to[out=90, in=-125] (0.5,0.7)  -- cycle;
            \draw 
                (0.5,-0.7) to[out=125,in=270] (0.25,0) to[out=90, in=-125] (0.5,0.7)
                (-0.5,-0.7) to[out=55,in=270] (-0.25,0) to[out=90,in=305] (-0.5,0.7);
    \end{scope}
}

\newcommand{\reddiscfromcentral}[2]{
    \pgfmathsetmacro{\ang}{atan2(#2,#1)}
    \pgfmathsetmacro{\half}{0.5*sqrt((#1)*(#1)+(#2)*(#2))}
    \begin{scope}[shift={(4,3)}, rotate=\ang]
        \fill[red, opacity=0.3]
            (0,0.05) to[out=8,in=180] (\half,0.1) arc[x radius=\half,y radius=0.1, start angle=90, end angle=-90] to[out=180,in=-8] (0,-0.05) -- cycle;
        \draw[red]
            (\half,0.1) arc[x radius=0.05,y radius=0.1, start angle=90,end angle=-90];
        \draw[red,densely dashed]
            (\half,-0.1) arc[x radius=0.05,y radius=0.1, start angle=270,end angle=90];
        \draw[red]
            (0,0.05) to[out=8,in=180] (\half,0.1) arc[x radius=\half,y radius=0.1, start angle=90,end angle=-90] (\half,-0.1) to[out=180,in=-8] (0,-0.05);
    \end{scope}
}

\usepackage[shortlabels]{enumitem}

\usepackage[final]{hyperref}
\usepackage[noabbrev,capitalize]{cleveref}
\creflabelformat{equation}{#2\textup{#1}#3}
\hypersetup{
  pdfcreator = {},
  pdfproducer = {}
}

\usepackage{wrapfig}

\declaretheorem[numberwithin=section]{theorem}
\declaretheorem[sibling=theorem]{lemma, proposition, corollary}
\declaretheorem[sibling=theorem,style=definition]{definition}
\declaretheorem[sibling=theorem,style=remark]{remark, example}

\newtheorem*{question}{Question}

\newtheorem{thmINTRO}{Theorem}

\newtheorem{appINTRO}{Application}

\newtheorem*{mnTRHM}{Main Theorem}

\newcommand\pmat[1]{\begin{pmatrix}#1\end{pmatrix}}

\def\C{\mathbb{C}}

\def\F{\mathbb{F}}

\def\Q{\mathbb{Q}}
\def\R{\mathbb{R}}

\def\Z{\mathbb{Z}}

\def\cb{{\mathcal B}}
\def\cc{{\mathcal C}}
\def\cd{{\mathcal D}}

\def\cj{{\mathcal J}}

\def\co{{\mathcal O}}

\def\ct{{\mathcal T}}

\newcommand{\mC}{\mathcal{C}}

\newcommand{\mS}{\mathcal{S}}
\newcommand{\fx}{\operatorname{Fix}}
\newcommand{\pr}{\operatorname{Pr}}
\newcommand{\symp}{\operatorname{Symp}}
\newcommand{\aut}{\operatorname{Aut}}

\newcommand{\vis}{\mathrm{vis}}
\newcommand{\std}{\mathrm{st}}
\newcommand{\orb}{\mathrm{orb}}
\newcommand{\str}{\mathrm{star}}
\newcommand{\TSR}{T^*S^1\times \mathbb{R}^2}
\newcommand{\ATF}{\mathfrak{A}}
\newcommand{\ha}{\widehat}
\newcommand{\til}{\widetilde}
\newcommand{\Lvis}{L_{p,q}^{\normalfont{\text{vis}}}}
\newcommand{\Lstd}{L_{p,q}^{\normalfont{\text{st}}}}

\begin{document}

\author[]{Nikolas Adaloglou}
\address{Nikolas Adaloglou, 
    imj-prg, 
    Sorbonne Université et Université Paris Cité, CNRS}
\email{adaloglou@imj-prg.fr} 

\author[]{Gerard Bargalló i Gómez}
\address{Gerard Bargalló i Gómez,
    Institut für Mathematik
Humboldt-Universität zu Berlin}
\email{gerard.bargallo.i.gomez@hu-berlin.de} 

\author[]{Johannes Hauber}
\address{Johannes Hauber,
    Institut de Math\'ematiques,
    Universit\'e de Neuch\^atel}
\email{johannes.hauber@unine.ch} 

\date{\today}

\title{The Nearby Lagrangian Conjecture for pinwheels}

\begin{abstract}
    The Lagrangian skeleton of the rational homology ball $B_{p,q}$, for $0<q<p$ coprime~integers, is an immersed but not embedded Lagrangian, called a $(p,q)$-pinwheel.
    We show that any two embeddings of Lagrangian $(p,q)$-pinwheels in $B_{p,q}$ are related by a compactly supported Hamiltonian isotopy, establishing Arnold's \textit{nearby Lagrangian conjecture}\/ for this wide class of singular Lagrangians.
    Our proof has two largely independent parts: the first uses \textit{neck-stretching}\/ and the \textit{symplectic rational blow-up}\/ to understand embeddings of pinwheels up to symplectomorphism;
    the second computes that $\symp_c(B_{p,q})$ is generated by a twist about the pinwheel, which we call the \textit{pintwist} $\tau_{p,q}$.
    We provide three applications of our methods: Gromov non-squeezing for pin-balls; a new proof of the local Lagrangian unknotting theorem of Eliashberg--Polterovich; and that the only Lagrangian $(n,m)$-pinwheel in $B_{p,q}$ is of type $(p,q)$. 
\end{abstract}

\maketitle

\section{Introduction}
\subsection{Short summary}

The goal of this paper is to establish the nearby Lagrangian conjecture (NLC), for a family of singular Lagrangians called \textit{Lagrangian $(p,q)$-pinwheels}, for any coprime integers $0<q<p$.
Our point of departure is that a Lagrangian $(p,q)$-pinwheel naturally appears as the Lagrangian skeleton, denoted by $L^{\std}_{p,q}$, of the Liouville manifold $B_{p,q}$.
By a fundamental result of Khodorovskiy, who initiated the study of Lagrangian pinwheels in \cite{Kho13}, every Lagrangian $(p,q)$-pinwheel in a symplectic $4$-manifold has a Weinstein-type neighbourhood which is symplectomorphic to a neighbourhood of the skeleton of $B_{p,q}$.
Therefore, one can view $B_{p,q}$ as the "cotangent bundle" of $L^{\std}_{p,q}$.
With this perspective, we will show that:

\begin{mnTRHM}
    Any Lagrangian $(p,q)$-pinwheel in $B_{p,q}$ is Hamiltonian isotopic to $L^{\std}_{p,q}$.
\end{mnTRHM}
\vspace{4pt}
The technical core of the paper consists of two largely independent parts:
\begin{enumerate}
    \item
        We determine how the \textit{symplectic rational blow-up} along a $(p,q)$-pinwheel $L_{p,q}\subset B_{p,q}$ affects the symplectic topology of $B_{p,q}$.
        In particular, we show that there is a compactly supported symplectomorphism $\psi:B_{p,q}\rightarrow B_{p,q}$ which maps $L_{p,q}$ to $L^{\std}_{p,q}$.
        A key point is using the almost toric fibration on $B_{p,q}$ to construct a convenient compactification $X_{p,q}$ which makes the analysis of pseudoholomorphic curves manageable.
    
    \item
        We compute the weak homotopy type of $\symp_c(B_{p,q})$.
        By understanding how $\symp_c(B_{p,q})$ acts on a distinguished symplectic cylinder $K_{p,q}$, we are able to show that $\symp_c(B_{p,q})$ is weakly homotopy equivalent (w.h.e.) to $\Z$, generated by a very natural compactly supported symplectomorphism $\tau_{p,q}$, which we call the \textit{pintwist}.
\end{enumerate}

In Section \ref{section: background}, we collect the relevant background material concerning the symplectic and algebraic geometry of $B_{p,q}$ and pinwheels.
Section \ref{sec: Uniq up to symp} concerns the first technical part of the paper and Section \ref{sec: symp to ham} contains the second technical part.
In \cref{sec:application}, we provide three applications of the results of the previous sections
to related quantitative and Lagrangian embedding problems of the rational homology balls $B_{p,q}$.
Finally, in the Appendix we clarify an obtuse point in the definition of pinwheels and show that all pinwheels are locally visible, in the sense that the Khodorovskiy neighbourhood admits an almost toric fibration compatible with the pinwheel.

\vspace{4pt}
In the rest of the introduction we will first provide more context around the nearby Lagrangian conjecture and Lagrangian pinwheels.
Then we will give a thorough discussion of our main results and the techniques we will develop to obtain them.

\subsection{Context and motivation}
\subsubsection{The nearby Lagrangian conjecture}

The \textit{nearby Lagrangian conjecture}, attributed to V. I. Arnol'd, asserts that \textit{given a smooth closed manifold $Q$, any exact Lagrangian in $T^*Q$ is Hamiltonian isotopic to the zero section}.
It is one of the foundational open problems in symplectic geometry, as it implies, among other things, that the smooth topology of $Q$ and the symplectic topology of $T^*Q$ are strongly related.

The NLC can be broken into two different subconjectures:
\begin{enumerate}[label={\small\textbf{NLC}$_\arabic*$}, widest={\small\textbf{NLC}$_2$}, leftmargin=4em]
   \item
        Any exact Lagrangian embedding $Q\hookrightarrow T^*Q$ is Hamiltonian isotopic (possibly after reparametrization) to the zero section.
    
    \item
        Any exact Lagrangian $L\subset T^*Q$ is diffeomorphic to $Q$.
\end{enumerate}
The two subconjectures have a wildly different character depending on whether $T^*Q$ has dimension $4$ or $\geq 6$.
When $\dim T^*Q=4$, NLC$_2$ becomes manageable by the classification of compact surfaces and therefore research has focused on NLC$_1$, where the effectiveness of the theory of pseudoholomorphic curves has borne results.
Using positivity of intersections of pseudoholomorphic curves, which is specific to dimension $4$, the full NLC has been established for some simple smooth $2$-dimensional manifolds, namely $T^*S^2$ \cite{Hi04,LiWu12}, $T^*T^2$ \cite{DRGI16}, $T^*\mathbb{R}P^2$ \cite{HPW,Ad25} and in the non-compact case $T^*\mathbb{R}^2$ \cite{EliPol96}.
In addition, a version of the NLC has also been established in \cite{Dim17} for the Lagrangian \textit{Whitney sphere}, i.e.\ a Lagrangian sphere with a single positive self-intersection.

\vspace{4pt}
Considerable effort was needed to establish each of these results, and there does not seem to be a unified approach.
However, all of these ultimately leverage the fact that there exists some compactification of the respective cotangent bundle to a simple rational symplectic $4$-manifold, namely $\mathbb{C}P^2$ or $S^2\times S^2$, where rational pseudoholomorphic curves are very well understood.
The reader is pointed to the comprehensive survey \cite{PolSch24} for more information and anecdotes around these low-dimensional results and their history.

\vspace{4pt}
The situation is starkly different in higher dimensions.
There, already NLC$_2$ becomes extremely intricate; for example, it is not at all clear to what extent the smooth structure of $Q$ dictates the symplectic structure of $T^*Q$.
The most general step towards proving NLC$_2$ culminated in \cite{AK18}, building upon \cite{FuSeSm08, Ab12, Kra13} amongst others, where the authors show that any exact Lagrangian is simple homotopy equivalent to the zero section.
See also \cite{AbCoGuKr25, AbAGCoKr} for more recent results.
However, in high dimensions there has not been any progress towards NLC$_1$, because of the lack of control over the intersections of pseudoholomorphic curves.

\subsubsection{Pinwheels}
A $p$-pinwheel is the topological space obtained by attaching a $2$-cell to a $1$-cell via the $p$-to-$1$ covering map.
A Lagrangian $(p,q)$-pinwheel is a $p$-pinwheel whose $2$-cell is Lagrangian.
In addition, we require that the $2$-cell meets the $1$-cell in a specified way, which depends on $q$.
Since the definitions are somewhat technical, we refer to the Background Section \ref{section: pinwheels} for the details.
We note that $(1,1)$-pinwheels are just Lagrangian discs and $(2,1)$-pinwheels are Lagrangian projective planes.

\vspace{4pt}
Pinwheels were first defined by Khodorovskiy in \cite{Kho13}, in order to define the \textit{symplectic rational blow-up}, which is an operation that replaces a standard neighbourhood of a pinwheel with a chain of symplectic spheres, called the \textit{Wahl chain} $\mC_{p,q}$, where the self-intersection numbers of the exceptional spheres in the Wahl chain depend on the Hirzebruch--Jung continued fraction associated to $\frac{p^2}{pq-1}$.
The symplectic rational blow-up is the opposite operation of the \textit{symplectic rational blow-down}, defined by Symington in \cite{Sym98:RatBlo}.

\vspace{4pt}
Lagrangian pinwheels lie at the interface between symplectic geometry and algebraic geometry, specifically surface singularities.
In particular, one can view the $(p,q)$-pinwheel as the \textit{vanishing cycle} of the $\frac{1}{p^2}(1,pq-1)$ cyclic quotient singularity.
In this setting, the rational blow-up operation can be viewed as exchanging the smoothing of the singularity with the chain of exceptional spheres $\mC_{p,q}$ that appear in the minimal resolution of the $\frac{1}{p^2}(1,pq-1)$ singularity.\footnote{In addition, the symplectomorphism $\tau_{p,q}$ can be viewed as the symplectic monodromy associated to the smoothing of the singularity.}
This point of view was first adopted by Evans and Smith in their influential papers \cite{EvSm18,EvSm20} where they show that Lagrangian pinwheels are the right objects whose symplectic rigidity reflects the algebraic geometry of quotient singularities and their smoothings.

\vspace{4pt}
The rational homology ball $B_{p,q}$ is the smoothing, or Milnor fibre, of the $\frac{1}{p^2}(1,pq-1)$ cyclic quotient singularity.
As such, it admits a universal $p$-fold covering by the Milnor fibre of the $A_{p-1}$-singularity, from which one can infer most basic symplectic properties of $B_{p,q}$; see \cref{sec: intro bpq}.
By Khodorovskiy's neighbourhood theorem \cite[Lemma 3.3/3.4]{Kho13}, every $(p,q)$-pinwheel has a neighbourhood symplectomorphic to a neighbourhood of $L^{\std}_{p,q}$.
Hence, $L^{\std}_{p,q}$ can be viewed as the zero section of $B_{p,q}$.
Our main theorem is: 
\begin{thmINTRO}[\normalfont{NLC$_1$ for pinwheels; Theorems \ref{th:symp} and \ref{thrm: symp B_pq}}]
\label{th: nearby lagrangian for pinwheels}
    Let $L$ be a $(p,q)$-pinwheel in $B_{p,q}$. 
    Then, there exists a compactly supported Hamiltonian isotopy $\psi$ such that $\psi(L)=L^{\normalfont{\std}}_{p,q}$.
\end{thmINTRO}
By now, there is an extensive list of results studying qualitative and quantitative embeddings of the rational homology balls $B_{p,q}$ and pinwheels, see \cite{Kho13Bounds, EvSm18,EvSm20, ShSmi20,BrSch24,Ad24,AdHa25,ABEHS25,Bu25}.
These results show that Lagrangian pinwheels and their neighbourhoods naturally extend several classical theorems of symplectic geometry, e.g.\ Gromov's non-squeezing \cite{AdHa25}, the staircases of McDuff--Schlenk \cite{ABEHS25}, and Biran's Lagrangian barrier phenomena \cite{BrSch24}.

\subsection{Main results and outline of the paper}

In \cref{section: background} we set up notation and collect the relevant facts we will use in the rest of the paper.
To keep the paper at a reasonable length we point the reader to precise references for the various facts we use.
In particular, we will assume working familiarity with notions of almost toric fibrations and toric diagrams, at the level of Evans' invaluable book \cite{Ev23:Book}.

Theorem \ref{th: nearby lagrangian for pinwheels} is the combination of two independent results which form the technical body of the paper.
To show the nearby Lagrangian conjecture for a pinwheel $L \subset B_{p,q}$ we first show it to be true \textit{up to symplectomorphism}:
\begin{thmINTRO}\label{thrm_intro: B}
    Let $L$ be a Lagrangian $(p,q)$-pinwheel in $B_{p,q}$.
    Then, there exists $\psi\in \symp_c(B_{p,q})$ such that $\psi(L)=L^{\normalfont{\std}}_{p,q}$.
\end{thmINTRO}
Section \ref{sec: Uniq up to symp} concerns the proof of Theorem \ref{thrm_intro: B}.
First, we set up a convenient compactification ${X}_{p,q}$ of $B_{p,q}$, which allows us to pass to the more standard setting of closed pseudoholomorphic curves.
Most of the geometry of ${X}_{p,q}$ is encoded in the compactifying divisor $\cd_{p,q}$.
Then, we show that rationally blowing up a $(p,q)$-pinwheel produces a symplectic manifold $\til{X}_{p,q}$ which is \textit{regulated}, in the sense of \cite{ABEHS25}.
Roughly, this means that $\til{X}_{p,q}$ admits a foliation by rational curves, all but finitely many of which are smooth. 
The main technical part of this section is to show that given \textit{any} Lagrangian $(p,q)$-pinwheel, the resulting blown-up manifolds are symplectomorphic via a symplectomorphism that preserves certain aspects of the corresponding regulations.
In particular, \cref{thm:regulation_of_tilX} shows that the regulation consists of exactly one broken fibre, and its shape (i.e.\ number of spheres and their self-intersections) depends only on $p$ and $q$.
This implies that the complements of the Khodorovskiy neighbourhoods of any two Lagrangian pinwheels are symplectomorphic.
The last part of this section then shows that such a symplectomorphism can be extended to the interior of the neighbourhoods, thus providing the desired symplectomorphism.

To upgrade the symplectomorphism of Theorem \ref{thrm_intro: B} to a Hamiltonian diffeomorphism, we compute the weak homotopy type of $\symp_c(B_{p,q})$ to obtain control over the symplectomorphism group.
We show that a certain natural symplectomorphism $\tau_{p,q}\in \symp_c(B_{p,q})$, first studied in Buck's PhD thesis \cite{Bu25}, generates the symplectic mapping class group.
Buck showed, using notions from Lagrangian Floer theory, that $\tau_{p,q}$ has infinite order in $\symp_c(B_{p,q})$.
Using completely different techniques, we show that, in fact, it generates the symplectic mapping class group.
The pintwist $\tau_{p,q}$ is very reminiscent of the usual Dehn--Seidel twist about a Lagrangian sphere or $\R P^2$.

\begin{thmINTRO}\label{thrm_intro: C}
    The group $\symp_c(B_{p,q})$ is weakly homotopy equivalent to $\Z$, with a generator being the \textit{pintwist} $\tau_{p,q}$.
    Therefore, for any $\psi\in \symp_c(B_{p,q})$ there exists $\phi\in \text{Ham}_c(B_{p,q})$ such that $\psi=\tau_{p,q}^n\circ\phi$.
\end{thmINTRO}

\cref{sec: symp to ham} contains the proof of Theorem \ref{thrm_intro: C}.
The main point is to understand $\symp_c(B_{p,q})$ via its action on the space of symplectic cylinders $\mS(K_{p,q})$ which agree outside a compact set with the anticanonical divisor $K_{p,q}$, which we call \textit{the central cylinder} of $B_{p,q}$; this is detailed in \cref{sec: intro bpq}. 
The study of this action is carried out through the following diagram, which we call a \textit{Comb diagram}\footnote{Looking at the diagram from afar, it resembles a horizontal \textit{hair-comb}.}:
\begin{equation*}
    \begin{tikzcd}
    \symp_c(B_{p,q}\backslash K_{p,q}) \arrow[r] & \fx(K_{p,q}) \arrow[d, "D"] \arrow[r] & \pr(K_{p,q}) \arrow[d, "\rho"] \arrow[r] & \symp_c(B_{p,q}) \arrow[d, "\text{transitive action}" description] \\
                                   & \aut(\nu K_{p,q})\simeq \mathbb{Z}           & \symp_c(K_{p,q})\simeq \mathbb{Z}                & \mathcal{S}(K_{p,q})\simeq pt                            
    \end{tikzcd}
\end{equation*}
The spaces appearing in this diagram are the following: $\pr(K_{p,q})$ (resp.\ $\fx(K_{p,q})$) is the space of compactly supported symplectomorphisms that preserve (resp.\ fix) the central cylinder $K_{p,q}$ and $\aut(\nu K_{p,q})$ is the space of symplectic automorphisms of the normal bundle of $K_{p,q}$.
The diagram is built from the three fibrations obtained from the natural action of $\symp_c(B_{p,q})$ on $\mS(K_{p,q})$; the restriction map $\rho$; and the normal derivative map $D$.
The top row that holds the "teeth" of the comb is formed by inclusions.

The essential feature of the comb diagram is that it relates the homotopy groups of $\symp_c(B_{p,q})$ to those of $\symp_c(B_{p,q}\backslash K_{p,q})$.
The reason to consider the central cylinder as the divisor in the comb diagram is that the Liouville manifolds $B_{p,q}\backslash K_{p,q}$ can be identified with star-shaped domains in a single complete Liouville manifold $T^*Wh$.
We then show, again via the appropriate comb diagram, that the symplectomorphism group of $T^*Wh$ is weakly contractible and therefore $\symp_c(B_{p,q}\backslash K_{p,q})$ is also weakly contractible.

The symplectic geometry of $T^*Wh$, the cotangent bundle of the Whitney sphere, was first studied by Dimitroglou-Rizell in \cite{Dim17}, and we make good use of his analysis of the holomorphic curves therein.

\vspace{4pt}
The comb diagram has been the main tool in computing new symplectomorphism groups from old ones.
The idea is already implicit in Gromov's original paper \cite[Section 2.4]{Gr85} and was used in most subsequent computations, for example \cite{Abreu98,AbMcD00}, as well as \cite[Section 6]{Ev14} and very recently \cite[Equation (2.3)]{AtMaWu25}.
Given a symplectic manifold $(X,\omega)$ and some symplectic divisor $S$, the comb diagram allows one to relate $\symp_c(X)$, $\symp_c(X\backslash S)$ and the orbit space of $S$ under the action of $\symp(X)$, which is commonly shown to be a moduli space $\mathcal{S}(S)$ of symplectic surfaces.
Usually, this allows one to pass from a \textit{compact} manifold $X$ to the \textit{non-compact} $X\backslash S$.
One of the important novelties of our work is to show that this technique can also work when $X$ is not closed.
Indeed, to compute $\symp_c(B_{p,q})$, we will \textit{not} compactify $B_{p,q}$ but, rather, we will be viewing $B_{p,q}$ as a partial compactification of a space $D^*Wh$, the unit cotangent bundle of the Whitney sphere.
The reader is invited to compare our computation of $\symp_c(B_{2,1})$ to that of Evans, in \cite[Theorem 1.5]{Ev14}.\footnote{It is very tempting to try to mimic the Seidel--Evans approach, i.e.\ to first compute the $\mathbb{Z}_p$-equivariant symplectomorphism group and then to relate it to $\symp_c(B_{p,q})$. However, we were not able to make this approach work; in $A_{p-1}$ there are $p$ holomorphic foliations, and when $p>2$, they might intersect with each other in uncontrolled ways.}\footnote{The recent preprint of Keating--Smith--Wemyss \cite{KeSmWe26} contains a thorough discussion of computing certain algebraic aspects of symplectomorphism groups through mirror symmetry and the action of $\symp$ on the Fukaya category.}

\vspace{4pt}
The key hard input from holomorphic curves in this section is the content of Proposition \ref{prop: cylinders in Bpq are rigid}, analysing the properties of rational pseudoholomorphic curves in the compactification $X_{p,q}$.
We show that the compactified curves in the class of the regulation of $X_{p,q}$ cannot break, thereby establishing that $\mS(K_{p,q})$ is contractible.
$\symp_c(B_{p,q})$ retracts to $\pr(K_{p,q})$ and this is shown to be generated by $\tau_{p,q}$.

\vspace{4pt}
In the rest of the paper, we provide various applications of the techniques we develop to prove Theorems \ref{thrm_intro: B} and \ref{thrm_intro: C}.
The first application concerns a pin-version of Gromov's non-squeezing theorem.
Using the almost toric base diagrams of $B_{p,q}$, we can define various domains that generalize the usual toric ones.
In particular, we define the \textit{pin-ellipsoids} $E_{p,q}(\alpha,\beta)$ for $0<\alpha,\beta \leq \infty$.
In analogy with the standard situation, we let $B_{p,q}(r)=B_{p,q}(r,r)$ be the \textit{pin-ball}.
We show that they satisfy a pin-version of Gromov's non-squeezing theorem:
\begin{appINTRO}[Pin-ball non-squeezing]\label{app: 1}
    The pin-ball $B_{p,q}(1)$ symplectically embeds into the pin-cylinders $E_{p,q}(\alpha,\infty)$ and $E_{p,q}(\infty,\beta)$ if and only if $1\leq \alpha,\beta$.
\end{appINTRO}
For the proof, we assume that a symplectic embedding $\iota: B_{p,q}(1)\hookrightarrow E_{p,q}(\infty,\beta)$ exists and then rationally blow up along $\iota$.
The unique broken fibre of the regulation obtained on the blown-up manifold contains a single exceptional curve, whose area provides the desired obstruction.

The second application concerns an alternative proof of a classical theorem of Eliashberg and Polterovich \cite{EliPol96} concerning Lagrangian planes.
This theorem is of fundamental importance in $4$-dimensional symplectic geometry as it shows that one cannot produce Lagrangian knotting by considering some local operation on a Lagrangian surface.
(However, Lagrangian knotting \textit{does} occur on the global level, as was first exhibited by Seidel \cite{Sei99}.) 
\begin{appINTRO}[Eliashberg--Polterovich local unknottedness]\label{app: 2}
    Let $\Delta_{\std}$ be a standard Lagrangian disc with boundary on $S^3\subset \R^4$ and let $\Delta$ be a Lagrangian which agrees with $\Delta_{\std}$ near the boundary.
    Then, there exists a compactly supported Hamiltonian diffeomorphism $\psi\in \text{Ham}_c(B^4)$ such that $\psi(\Delta)=\Delta_{\std}$.
\end{appINTRO}
The idea of our proof follows a relative version of Theorem \ref{thrm_intro: B}, for the case $(p,q)=(2,1)$.
Since the Lagrangian discs $\Delta,\Delta_{\std}$ are standard near the boundary, we can use the symplectic cut operation to compactify the ball $B^4$ into $\C P^2$ and the Lagrangian discs into Lagrangian real projective planes.
Application \ref{app: 2} is equivalent to the fact that there exists a symplectomorphism of $\C P^2$ which is the identity near the compactifying line $\C P^1_{\infty}$ and maps the two Lagrangian projective planes onto each other.
We show this by carefully analysing the rational blow-up in this setting.

\vspace{4pt}
The final application concerns the \textit{nearby Lagrangian subconjecture 2}, adapted for pinwheels, namely answering the natural question: \textit{Which Lagrangian pinwheels exist in $B_{p,q}$?}
For example, the paper of Lekili--Maydanskiy \cite{LeMa14}, which originally studied symplectic aspects\footnote{Interestingly, in \cite{LeMa14} it is also shown that the symplectic homology of $B_{p,q}$ does not vanish in spite of having no exact Lagrangians for $p>2$.
Another notable symplectic feature of $B_{p,q}$, for $q=1$, is that by \cite{Kar18} the Kontsevich cosheaf conjecture is known to hold for them.} of $B_{p,q}$, shows that, for $p>2$, there are no smooth exact Lagrangian submanifolds in $B_{p,q}$.\footnote{Since $B_{2,1}$ is symplectomorphic to $T^*\mathbb{R}P^2$, it contains many Lagrangian $\mathbb{R}P^2$s, all Hamiltonian isotopic by \cite{HPW,Ad25} and by the main theorem of this paper.}
We extend these results to show that the only Lagrangian pinwheels in $B_{p,q}$ are those of $(p,q)$-type, i.e.\ of the same type as the Lagrangian skeleton of $B_{p,q}$.

\begin{appINTRO}
\label{app:3}
    If there is a symplectic embedding of a star-shaped subdomain of $B_{n,m}$ into $B_{p,q}$, then $(n,m)=(p,q)$.
    In particular, a Lagrangian pinwheel in $B_{p,q}$ is of $(p,q)$-type.
\end{appINTRO}

The proof is reminiscent of the one in \cite[Section 5.4]{ABEHS25}, which itself is an elaboration of \cite{EvSm18}.
In particular, we assume the existence of a symplectic embedding $\iota:B_{n,m}\hookrightarrow B_{p,q}$ and then stretch the neck around the contact boundary of $\iota(B_{n,m})$ to find specific rational curves in the rational blow-up $\til{X}_{p,q}$.
The numerics for these curves, encoded in the adjunction formula, imply that $(n,m)=(p,q)$.
The corresponding statement for pinwheels follows by considering the associated Khodorovskiy neighbourhood.
This result fits into the growing literature on embeddings of rational homology balls; see, for instance, Evans and Smith~\cite{EvSm18} and the more topologically flavoured works \cite{EtHyPi2023, LiPa22, LiPa24, Ow2018:equiemb, Ow20:nonsymp, GoOw25}.

\subsection{Relation with ABEHS}
Many of the ideas contained in this paper, especially \cref{sec: constructing the symp}, grew out of the collaboration of the first and third authors with Brendel, Evans, and Schlenk that culminated in \cite{ABEHS25}.
In a certain sense, this paper can be seen as a prequel to \cite{ABEHS25}, as the theorems here can be seen as local versions of the isotopy and embedding theorems of \cite{ABEHS25}, concerning embeddings of $B_{p,q}$ into $\C P^2$.

\vspace{4pt}
Theorem \ref{thrm_intro: B}, which constructs the symplectomorphism relating the pinwheels, should be compared with \cite[Corollary 5.2.2]{ABEHS25}.
In fact, our situation is much more straightforward; $B_{p,q}$ is the natural manifold containing $L_{p,q}$; therefore, the arithmetic appearing in the adjunction formula of the neck-stretched curves in Section \ref{sec: Uniq up to symp} is somewhat more manageable.

\vspace{4pt}
On the other hand, upgrading the symplectomorphism to a Hamiltonian one is automatic in the setting of \cite{ABEHS25}, because $\symp(\C P^2)$ is connected by Gromov \cite{Gr85}.
However, apart from the calculation of $\symp_c(B_{p,1})$, which is due to Gromov \cite{Gr85} for $p=1$ and to Evans \cite{Ev14} for $p=2$, the symplectomorphism group of $\symp_c(B_{p,q})$ for other $p>2$ was entirely unknown before.\footnote{Recall that $B_{1,1}$ is just the standard ball $B^4$ and $B_{2,1}$ is the unit disc bundle of $\R P^2$.}
Partial progress was made by Buck \cite{Bu25}, who showed that $\tau_{p,q}$ has infinite order.

\vspace{4pt}
Finally, the quantitative embedding obstructions of \cref{app: 1} are obtained similarly to \textit{Markov staircases}, in the sense that in both cases the obstructions are extracted from the area of exceptional curves in the broken fibres after blowing up.

\subsection{Acknowledgments}
This article is intellectually deeply indebted to Jonny Evans as well as to the  collaboration \cite{ABEHS25}.
We wish to thank Russell Avdek, Matthew Buck, Jonny Evans, Georgios Dimitroglou-Rizell, Paolo Ghiggini, Marco Golla, Vincent Humilière, Yanki Lekili, Lukas Nakamura, Federica Pasquotto, Felix Schlenk, George Politopoulos, Joel Schmitz and Chris Wendl for many useful and stimulating discussions. 
Finally, N.A.\ wishes to acknowledge the influence of his master's thesis advisor, Will Merry.
In a certain sense, this paper is the conclusion of the project N.A.\ undertook under Will's supervision. 

G.B.i.G. was funded by the Deutsche Forschungsgemeinschaft (DFG, German Research Foundation) under Germany's Excellence Strategy – The Berlin Mathematics Research Center MATH+ (EXC-2046/2, project ID: 390685689).

\section{Background}
\label{section: background}

In this section we will collect basic definitions and properties of the rational homology ball $B_{p,q}$ and Lagrangian pinwheels.
We have tried to keep the exposition tight and self-contained.
However, in order to keep the paper at a reasonable length, for certain technical aspects (which will not be relevant for the rest of the paper), the reader will be referred to the literature.
In particular, we assume working familiarity with almost toric fibrations (ATFs), at the level of \cite{Ev23:Book,Sym02:Fourtwo}.

Unless stated otherwise, we always assume that $0<q<p$ are two coprime integers, and we will denote the group of the $p$th roots of unity by $\bm{\mu}_p:=\{\zeta \in \C \mid \zeta^p=1\}$.

\subsection{The rational homology ball $B_{p,q}$ and its cylinder $K_{p,q}$}
\label{sec: intro bpq}

Here we review some important facts about the symplectic geometry of the rational homology balls $B_{p,q}$.
Following Lekili--Maydanskiy \cite{LeMa14}, we view $B_{p,q}$ as a $p$-fold quotient.
Consider the complex hypersurface $A_{p-1}=\{z_1z_2=z_3^p+1\}\subset \mathbb{C}^3$ with the restriction of the standard symplectic form $\omega_{\std}$ on $\C^3$.\footnote{In fact, this is the Milnor fibre of the $A_{p-1}$-singularity, but for the sake of brevity, we choose this more compact notation.}
There is a free $\mathbb{Z}_p$-K\"ahler action on $A_{p-1}$ given by 
\begin{equation}\label{eq:action_weights}
    \zeta\cdot(z_1,z_2,z_3)=(\zeta z_1,\zeta^{-1}z_2,\zeta^qz_3)\quad\text{where } \zeta\in \bm{\mu}_p.
\end{equation}

\begin{definition}\label{def:Bpq}
    The symplectic manifold $(B_{p,q},\omega_{p,q})$ is defined to be the quotient of $(A_{p-1},\omega_{\std})$ by this $\Z_p$-action.
\end{definition}

There are two different fibrations on $B_{p,q}$ that we will use to understand its symplectic geometry.
One is a \textit{Lefschetz (orbi-)fibration}\footnote{We will not focus on the orbifold aspects of this fibration, since we will mostly work on the relevant honest Lefschetz fibration on the universal cover, which is $A_{p-1}$.} and the other is an \textit{almost toric fibration}.

\subsubsection{The Lefschetz fibration}
\label{subsec:Lefschetz_fibration}

Consider the usual Lefschetz fibration $\tilde{\pi}:A_{p-1}\rightarrow \mathbb{C}$ given by projecting onto the third coordinate $z_3$.
The only nodal fibres sit over the $p$th roots of unity $\bm{\mu}_p$, and the $\mathbb{Z}_p$-action defined in \eqref{eq:action_weights} respects the fibration (in the sense that it maps fibres to fibres).
This is shown in \cref{fig:Lefschetz_fib}.
Notably, the action preserves the central fibre $\tilde{\pi}^{-1}(0)$.

\begin{definition}[Central cylinder]
\label{def: central cylinders}
    The \textit{central cylinder} of $(A_{p-1},\omega_{\std})$ is the fibre $K_{p-1}:=\tilde{\pi}^{-1}(0)$.
    The \textit{central cylinder} of $(B_{p,q},\omega_{p,q})$ is the quotient of $K_{p-1}$ under the corresponding action.
    We will denote this cylinder by $K_{p,q}$.
\end{definition}

\begin{remark}\label{rmk: central cylinders are smoothings}
    The hyperplane section $\{z_3=0\}\cap \{z_1z_2=z_3^p\}=:\widetilde{K}_{p-1}$ is just the usual nodal conic in the plane $\{z_3=0\}$.
    The smoothing $\{z_1z_2=z_3^p+1\}$ of $\{z_1z_2=z_3^p\}$ induces the usual smoothing of $\til{K}_{p-1}$ to a smooth affine conic, which is precisely $K_{p-1}$.
    In other words, the two nested singularities $\til{K}_{p-1}\subset \til{A}_{p-1} = \{z_1z_2=z_3^p\}$ get simultaneously smoothed to $K_{p-1}\subset A_{p-1}$.
    We will use this later to show that the symplectic monodromy of $A_{p-1}$ restricts to the symplectic monodromy of $K_{p-1}$.
\end{remark}

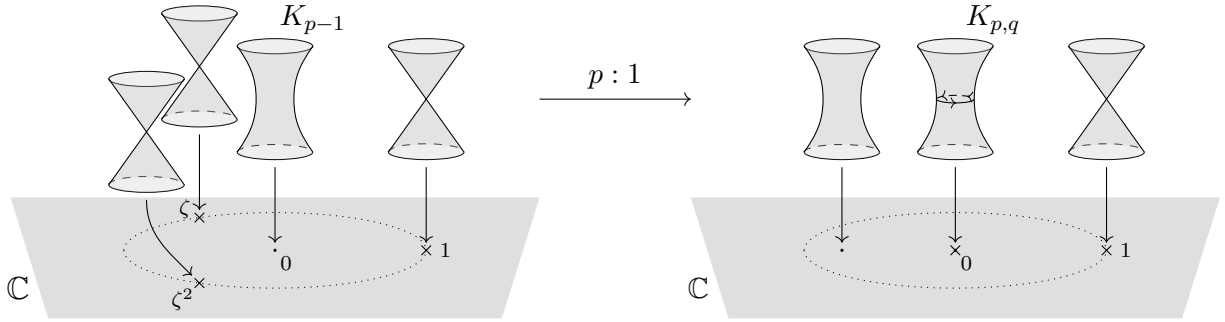
\begin{figure}[ht]
    \centering
    \begin{tikzpicture}
        \begin{scope}[shift={(4.5,0)}]
            \filldraw[fill=lightgray,opacity=0.5,draw=none] 
                (1,0.1) -- (7,0.1) -- (7.5,1.7) -- (0.5,1.7) -- cycle;
            \node at (0.6,0.5) {\small $\mathbb{C}$};
            \node at (4.5,4.05) {\small $K_{p,q}$};
            \conic{(4,3)}
            \nodalconic{(6,3)}
            \conic{(2.5,3)}
            \draw[->] (2.5,2.1) -- (2.5,1.1);
            \draw[->] (6,2.1) -- (6,1.1);
            \draw[->] (4,2.1) -- (4,1.1);
            \fill (2.5,1) circle (0.6 pt);
            \draw[dotted]
                (2,1) arc [x radius=2,y radius=0.5,start angle=180,end angle=0];
            \draw[dotted]
                (2,1) arc [x radius=2,y radius=0.5,start angle=-180,end angle=0];
            \coordinate (a) at (6,1);
            \draw (a) node[crosshor]{};
            \coordinate (z) at (4,1);
            \node[right=1pt] at (a) {\tiny $1$};
            \draw (z) node[crosshor]{};
            \fill (4,1) circle (0.6 pt) node[below right=-2pt] {\tiny $0$};
            \draw
                (3.75,3) arc[x radius=0.25,y radius=0.05, start angle=-180, end angle=0];
            \draw[densely dashed]
                (3.75,3) arc[x radius=0.25,y radius=0.05, start angle=180, end angle=0];
            \path[postaction={decorate}, decoration={markings, mark=at position 0.10 with {\arrow[scale=0.8]{>}}, mark=at position 0.43 with {\arrow[scale=0.8]{>}}, mark=at position 0.76 with {\arrow[scale=0.8]{>}}}]
            (4,3) ellipse[x radius=0.25,y radius=0.05];
        \end{scope}

        \draw[->] (3,3) -- (5,3);
        \node at (4,3) [above] {\small $p:1$};
        
        \begin{scope}[shift={(-4.5,0)}]
            \filldraw[fill=lightgray,opacity=0.5,draw=none] 
                (1,0.1) -- (7,0.1) -- (7.5,1.7) -- (0.5,1.7) -- cycle;
            \node at (0.6,0.5) {$\mathbb{C}$};
            \node at (4.5,4.05) {\small $K_{p-1}$};
            \conic{(4,3)}
            \nodalconic{(2.3,{3-sqrt(3)/4})}
            \nodalconic{(6,3)}
            \nodalconic{(3,{3+sqrt(3)/4})}
            \draw[->] (4,2.1) -- (4,1.1);
            \draw[->] (3,{2.1+sqrt(3)/4}) -- (3,{1.1+sqrt(3)/4});
            \draw[->] (6,2.1) -- (6,1.1);
            \draw[->] (2.3,{2.1-sqrt(3)/4}) to[out=270, in=135] (2.9,{1.1-sqrt(3)/4});
            \draw[dotted]
                (2,1) arc [x radius=2,y radius=0.5,start angle=180,end angle=0];
            \draw[dotted]
                (2,1) arc [x radius=2,y radius=0.5,start angle=-180,end angle=0];
            \coordinate (a) at (6,1);
            \coordinate (b) at (3,{1+sqrt(3)/4});
            \coordinate (c) at (3,{1-sqrt(3)/4});
            \node[right=1pt] at (a) {\tiny $1$};
            \node[xshift=-5.5pt, yshift=3.5pt]  at (b) {\tiny $\zeta$};
            \node[below left=-2pt]  at (c) {\tiny $\zeta^2$};
            \draw (a) node[crosshor]{};
            \draw (b) node[crosshor]{};
            \draw (c) node[crosshor]{};
            \fill (4,1) circle (0.6 pt) node[below right=-2pt] {\tiny $0$};
        \end{scope}
    \end{tikzpicture}
    \caption{The Lefschetz fibrations defined on $A_{p-1}$ and $B_{p,q}$. Note that the central cylinder $K_{p,q}$ is $p$-fold covered by a generic fibre of the Lefschetz fibration defined on $B_{p,q}$.}
    \label{fig:Lefschetz_fib}
\end{figure}

\subsubsection{The almost toric fibration} 

The symplectic manifold $B_{p,q}$ also admits an almost toric fibration, as is thoroughly explained in \cite{Ev23:Book, Ev24:KIAS}.
Its almost toric base diagram $\ATF_{p,q}$ is as shown in \cref{fig:atfBpq}.
Representing $B_{p,q}$ via its almost toric base diagram has many advantages.
For example, it is easy to talk about geometrically interesting objects in $B_{p,q}$ and to come up with potential ways to prove theorems about them.
Domains in $B_{p,q}$ of interest to us are so-called pin-ellipsoids and pin-balls, which were introduced in \cite{AdHa25} and explored further in \cite{ABEHS25}.

\begin{definition}[Pin-ellipsoid and pin-ball]
     For $\alpha,\beta \in \mathbb{R}_{>0}$, the $(p,q)$-\textit{pin-ellipsoid} $E_{p,q}(\alpha,\beta) \subseteq B_{p,q}$ is defined by its almost toric base diagram $\ATF_{p,q}(\alpha,\beta)$, which is shown in \cref{fig:atfBpq}.
     The \textit{pin-ball} is defined to be $B_{p,q}(\alpha)=E_{p,q}(\alpha,\alpha)$.
\end{definition}

\begin{remark}
    As can be seen in the almost toric base diagrams in \cref{fig:atfBpq}, the contact boundary of $E_{p,q}(\alpha,\beta)$ is the lens space $L(p^2,pq-1)$.
    Equivalently, one can argue that the boundary of $B_{p,q}$ is the same (as contact manifolds) as the boundary of a small neighbourhood of the quotient singularity $\frac{1}{p^2}(1,pq-1)$ which is the lens space $L(p^2,pq-1)$ by definition; see \cref{remark: Bpq as smoothed cqs} for more details.
\end{remark}

While the fibrations are in a sense dual to each other, meaning that the Lefschetz fibration is a symplectic fibration and the ATF is a Lagrangian fibration, certain information is visible in both.
For example, the central cylinder $K_{p,q} \subseteq B_{p,q}$ is defined to be the fibre over zero, but can also be seen over the toric boundary of $\ATF_{p,q}$, i.e.\ as a visible surface in the sense of Symington \cite[Section 7]{Sym02:Fourtwo}, as explained in \cite[Section 7]{Ev23:Book}.

\begin{lemma}
    The visible surface over the toric boundary on the usual ATF on $B_{p,q}$ (resp.\ $A_{p-1}$) is the central cylinder $K_{p,q}$ (resp.\ $K_{p-1}$).
    In particular, the central cylinders are Poincaré dual to the first Chern class.
\end{lemma}

\begin{proof}
    The first part is proven in \cite[Example 7.6]{Ev23:Book}.
    The statement about the Chern class follows from \cite[Proposition 8.2]{Sym02:Fourtwo}.\footnote{Symington only states the result for closed symplectic manifolds. However, the proof generalizes immediately. To obtain a precise statement one then has to either work with Borel--Moore homology or with compact subsets, e.g.\ $E_{p,q}(\alpha,\beta)\subseteq B_{p,q}$.}
\end{proof}

\begin{figure}[htb]
  \begin{center}   
    \begin{tikzpicture}
    \begin{scope}[shift={(-6,0)}]
      \filldraw[lightgray,opacity=0.75] (0,1.5) -- (0,0) -- (6,1.5) -- cycle;
      \draw[thick] (0,1.5) -- (0,0) -- (6,1.5);
      \draw[dashed] (1,0.5) node[cross] {} -- (0,0);
    \end{scope}
    \begin{scope}[shift={(1,0)}]
      \filldraw[lightgray,opacity=0.75] (0,1) -- (0,0) -- (6,1.5) -- cycle;
      \draw[thick] (0,1) -- (0,0) -- (6,1.5);
      \draw [decorate,decoration={brace,amplitude=5pt,raise=1ex}] (0,0) -- (0,1) node[midway,xshift=-3ex]{\footnotesize $\beta$};
      \draw [decorate,decoration={brace,amplitude=5pt,raise=1ex}] (6,1.5) -- (0,0) node[midway,xshift=1.5ex,yshift=-2.5ex]{\footnotesize $\alpha$};
      \node at (6,1.5) [right] {\footnotesize $(\alpha p^2,\alpha(pq-1))$};
      \draw[dashed] (1,0.5) node[cross] {} -- (0,0);
      \draw[thick,dash pattern=on 7pt off 3pt] (0,1) -- (6,1.5);
    \end{scope}
    \end{tikzpicture}
    \caption{On the left the almost toric base diagram $\ATF_{p,q}$ of $B_{p,q}$. The branch cut points into the $(p,q)$-direction. Note that over the toric boundary lies the embedded symplectic cylinder $K_{p,q}$ and the Lagrangian pinwheel's core circle is an essential loop on the central cylinder $K_{p,q}$. On the right, the almost toric base diagram $\ATF_{p,q}(\alpha,\beta)$ of the pin-ellipsoid $E_{p,q}(\alpha,\beta)$.}
    \label{fig:atfBpq}
  \end{center}
\end{figure}
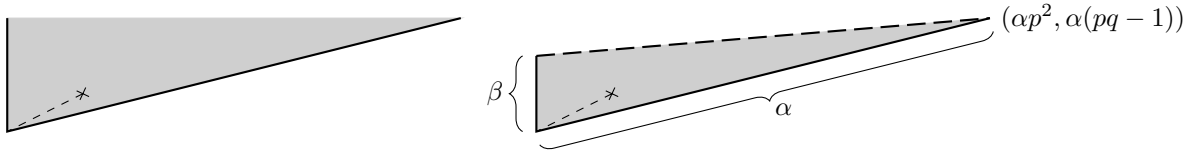

\subsection{Pinwheels}
\label{section: pinwheels}

\begin{definition}
    A \textit{topological $p$-pinwheel $P_p$} is the topological space $D^2/{\sim}$, where the equivalence relation on the unit disc $D^2 \subseteq \C$ is given by $z \sim z'$ if and only if $z,z'\in \partial D^2$ and $z=e^{2\pi i k/p}z'$ for some $k\in \mathbb{Z}$.
\end{definition}

\begin{definition}[Pinwheel]
\label{def:pinwheel_prelim}
    A \textit{Lagrangian $(p,q)$-pinwheel} in a symplectic manifold $(X,\omega)$, usually denoted by $L_{p,q}$, is a Lagrangian immersion $f:D^2 \looparrowright X$ such that
    \begin{enumerate}
        \item 
            $f$ factors through a continuous embedding $P_p\hookrightarrow X$, where $P_p$ is a $p$-pinwheel;
        \item 
            the restriction $f|_{D^2\setminus \partial D^2}$ is a Lagrangian embedding;
        \item 
            there is a Darboux chart of a neighbourhood $U$ of $f(\partial D^2)$ in $(X,\omega)$ to an open neighbourhood $V\subseteq T^*S^1\times \C$ that maps the \textit{core circle} $f(\partial D^2)$ to the zero section of $T^*S^1$ times $0\in \C$, such that in these coordinates the map $f$ takes the form
            \[f(t,s)=\left(pt,\frac{q}{2p}s^2,s e^{iqt}\right),\]
            where $(t,s)\in \partial D^2 \times [0,\varepsilon)\subseteq D^2$ are collar coordinates near the boundary of the parametrizing disc.
            See \cref{fig:pinwheel_core}.
    \end{enumerate}
\end{definition}

\begin{remark}
    A Lagrangian $(2,1)$-pinwheel is just a Lagrangian $\R P^2$, which can be described as a disc capping off a Möbius band.
    Similarly, for a general $L_{p,q}$, a neighbourhood of the core circle consists of $p$ flanges meeting at the core and winding around it $q$ times, as can be seen in the coordinates in \cref{def:pinwheel_prelim} $(3)$.
    The boundary circle of this neighbourhood projects to a $(p,q)$-torus knot\footnote{Here the projection is from $S^1\times D^3$ onto $S^1 \times$ the last two components of $D^3$.}; this circle is then capped off by a disc.
    See \cref{fig:pinwheel_core}.
\end{remark}

\begin{figure}[ht]
    \centering
    \includegraphics[width=0.8\linewidth]{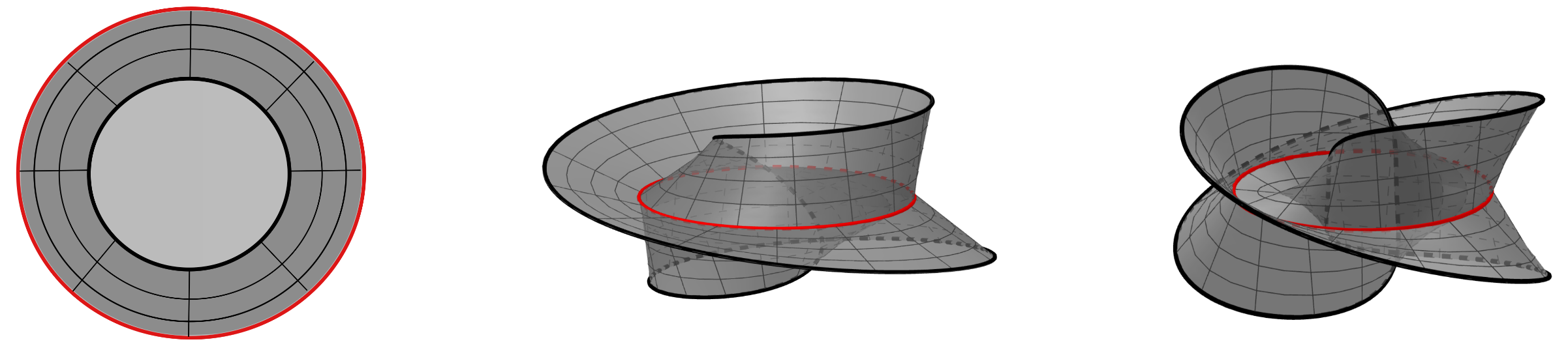}
    \caption{On the left the domain of parametrization of a Lagrangian pinwheel with a collar of its boundary (that becomes the core, in red) highlighted. On the right, the image of the collar neighbourhood for the $(3,1)$- and $(3,2)$-cases.}
    \label{fig:pinwheel_core}
\end{figure}

\begin{remark}
    Our definition of a Lagrangian pinwheel is stronger than the original one given by Khodorovskiy \cite[Definition 3.1]{Kho13}.
    There, the way the flanges meet the core circle is only prescribed up to homotopy, whereas we require that they meet in a specific geometric way.
    The reason for our stronger definition is very natural: for purely linear algebra reasons, two $(p,q)$-pinwheels cannot be related by a smooth ambient Hamiltonian isotopy if their flanges do not form at each point of the core circle $p$-tuples of Lagrangian planes that are related by a linear symplectic map.
    
    In contrast to the problem of Hamiltonian isotopies between pinwheels, for problems of existence of pinwheels or, more generally, of embeddings of rational homology balls $B_{p,q}$, the difference between these two definitions is immaterial.
    Indeed, it is proved in \cite[Lemma 3.3]{Kho13} that any pinwheel can be $C^0$-perturbed, through Lagrangian immersions, so that its flanges meet the core in the standard way.
\end{remark}

\begin{example}(Visible Pinwheel)
    As explained in \cite[Chapter 5]{Ev23:Book} and \cite[Section 7]{Sym02:Fourtwo}, over the branch cut in $\ATF_{p,q}$ one can explicitly construct Lagrangian $(p,q)$-pinwheels.
    If the intersection of the constructed Lagrangian pinwheel with any regular fibre over the branch cut consists of a single circle, then we will call it the \textit{visible Lagrangian $(p,q)$-pinwheel} and denote it by $L_{p,q}^{\text{vis}}$.
\end{example}

\begin{remark}
    A priori calling such a pinwheel "the" visible Lagrangian $(p,q)$-pinwheel is not justified.
    However, one can prove that any two choices of visible pinwheels are related by a Hamiltonian isotopy localised around the branch cut.
    Moreover, having chosen a visible Lagrangian $(p,q)$-pinwheel, the nodal slide can be realised in such a way that it preserves the pinwheel.
    See \cite[Subsection 7.4]{GroVar23} for more details on this.
    These two observations then justify talking about "the" visible Lagrangian $(p,q)$-pinwheel in $B_{p,q}$.
\end{remark}

\begin{definition}[Khodorovskiy neighbourhood]
    A neighbourhood with contact-type boundary $U\subseteq (X,\omega)$ of a Lagrangian pinwheel $L_{p,q} \subseteq (X,\omega)$ is called a \textit{Khodorovskiy neighbourhood} if its completion $\overline{U}$ along its contact boundary yields a pair $(\overline{U},L_{p,q})$ that is symplectomorphic to $(B_{p,q},\Lvis)$.
\end{definition}

Extending Khodorovskiy's neighbourhood construction in \cite[Section 3.4]{Kho13}, we prove that any Lagrangian pinwheel admits a Khodorovskiy neighbourhood.

\begin{theorem}[\normalfont{\cref{thm:pinwheels_locally_visible}}]
\label{thm:loc_vis}
    Suppose that $L_{p,q} \subseteq (X,\omega)$ is a Lagrangian pinwheel in a symplectic manifold. Then for some $\epsilon >0$ there exists a symplectic embedding $\psi: B_{p,q}(\epsilon) \hookrightarrow (X,\omega)$ such that $\psi(\Lvis)=L_{p,q}$.
\end{theorem}

From the dual point of view, one can also construct pinwheels using the Lefschetz-type fibration on $B_{p,q}$.

\begin{example}
    Consider the Lefschetz fibration $\tilde{\pi}:A_{p-1}\rightarrow \C$, as introduced in \cref{subsec:Lefschetz_fibration}.
    The fibration has $p$ critical points, so there are also $p$ vanishing thimbles.
    These are just $p$ Lagrangian discs with common boundary on the central cylinder $K_{p-1}$.
    As shown in \cite[Proposition 5.1.9]{Kar18}, one can take these thimbles to consist of the Lagrangian skeleton of the standard Liouville vector field.
    Passing to the quotient, these Lagrangian discs are identified, and since the central cylinder is preserved, one obtains a Lagrangian $(p,q)$-pinwheel in $B_{p,q}$, with its core circle on $K_{p,q}$.
    This is shown in \cref{fig:Vanishing_thimbles}.
    
    Since the vanishing thimbles in $A_{p-1}$ are visible in the almost toric fibration, the $(p,q)$-pinwheel constructed via the quotient above is also visible for the standard almost toric fibration on $B_{p,q}$.
    The visibility statements are explained in detail in \cite[Remark 7.10]{Ev23:Book}.
\end{example}

\begin{definition}[Standard pinwheel]
    The Lagrangian $(p,q)$-pinwheel resulting after taking the quotient of the Lagrangian skeleton of $A_{p-1}$ will be denoted by $L^{\std}_{p,q}\subseteq B_{p,q}$.
\end{definition}

\begin{remark}
    The standard pinwheel $L^{\std}_{p,q}\subseteq B_{p,q}$ and the visible pinwheel $L^{\vis}_{p,q}\subseteq B_{p,q}$ coincide. 
    They are the same Lagrangian pinwheel from two different viewpoints.
    See for example \cite{Ev24:KIAS}.
\end{remark}

\begin{figure}[htb]
  \begin{center}   
\begin{tikzpicture}
        \begin{scope}[shift={(4.5,0)}]
            \filldraw[fill=lightgray,opacity=0.5,draw=none] 
                (1,0.1) -- (7,0.1) -- (7.5,1.7) -- (0.5,1.7) -- cycle;
            \node at (0.6,0.5) {$\mathbb{C}$};
            \conic{(4,3)}
            \nodalconic{(6,3)}
            \draw[->] (6,2.1) -- (6,1.1);
            \draw[->] (4,2.1) -- (4,1.1);
            \draw[dotted]
                (2,1) arc [x radius=2,y radius=0.5,start angle=180,end angle=0];
            \draw[dotted]
                (2,1) arc [x radius=2,y radius=0.5,start angle=-180,end angle=0];
            \coordinate (a) at (6,1);
            \draw (a) node[crosshor]{};
            \coordinate (z) at (4,1);
            \draw[red] (a) -- (4,1);
            \node[right=1pt] at (a) {\tiny $1$};
            \draw (z) node[crosshor]{};
            \draw
                (3.75,3) arc[x radius=0.25,y radius=0.05,start angle=-180,end angle=0];
            \draw[dashed]
                (3.75,3) arc[x radius=0.25,y radius=0.05,start angle=180,end angle=0];
            \draw[red]
                (4,2.95) arc[x radius=0.25,y radius=0.05, start angle=270, end angle=0]
                (4,3.05) arc[x radius=0.25,y radius=0.05, start angle=90, end angle=-90]
                ;
            \draw[red]
                (5,3.1) arc[x radius=0.05,y radius=0.1, start angle=90,end angle=-90];
            \draw[red,densely dashed]
                (5,2.9) arc[x radius=0.05,y radius=0.1, start angle=270,end angle=90];
            \draw[red]
                (4,3.05) to[out=8,in=180] (5,3.1) arc[x radius=1,y radius=0.1, start angle=90,end angle=-90] to[out=180,in=-8] (4,2.95);
            \fill[red, opacity=0.4]
                (3.75,3) arc[x radius=0.25,y radius=0.05, start angle=180, end angle=90] 
                to[out=8,in=180] (5,3.1) arc[x radius=1,y radius=0.1, start angle=90,end angle=-90]
                to[out=180,in=-8] (4,2.95) arc[x radius=0.25,y radius=0.05, start angle=270, end angle=0]
                -- cycle;
            \fill (4,1) circle (0.6 pt) node[below right=-2pt] {\tiny $0$};
            \path[postaction={decorate}, decoration={markings, mark=at position 0.10 with {\arrow[scale=0.8]{>}}, mark=at position 0.43 with {\arrow[scale=0.8]{>}}, mark=at position 0.76 with {\arrow[scale=0.8]{>}}}]
            (4,3) ellipse[x radius=0.25,y radius=0.05];
        \end{scope}

        \draw[->] (3,3) -- (5,3);
        \node at (4,3) [above] {\small $p:1$};
        
        \begin{scope}[shift={(-4.5,0)}]
            \filldraw[fill=lightgray,opacity=0.5,draw=none] 
                (1,0.1) -- (7,0.1) -- (7.5,1.7) -- (0.5,1.7) -- cycle;
            \node at (0.6,0.5) {$\mathbb{C}$};
            \conic{(4,3)}
            \nodalconic{(2.3,{3-sqrt(3)/4})}
            \nodalconic{(6,3)}
            \nodalconic{(3,{3+sqrt(3)/4})}
            \draw[->] (4,2.1) -- (4,1.1);
            \draw[->] (3,{2.1+sqrt(3)/4}) -- (3,{1.1+sqrt(3)/4});
            \draw[->] (6,2.1) -- (6,1.1);
            \draw[->] (2.3,{2.1-sqrt(3)/4}) to[out=270, in=135] (2.9,{1.1-sqrt(3)/4});
            \draw[dotted]
                (2,1) arc [x radius=2,y radius=0.5,start angle=180,end angle=0];
            \draw[dotted]
                (2,1) arc [x radius=2,y radius=0.5,start angle=-180,end angle=0];
            \coordinate (a) at (6,1);
            \coordinate (b) at (3,{1+sqrt(3)/4});
            \coordinate (c) at (3,{1-sqrt(3)/4});
            \draw[red] (a) -- (4,1);
            \draw[red] (b) -- (4,1);
            \draw[red] (c) -- (4,1);
            \node[right=1pt] at (a) {\tiny $1$};
            \node[xshift=-5.5pt, yshift=3.5pt]  at (b) {\tiny $\zeta$};
            \node[below left=-2pt]  at (c) {\tiny $\zeta^2$};
            \draw (a) node[crosshor]{};
            \draw (b) node[crosshor]{};
            \draw (c) node[crosshor]{};
            \draw[red]
                (3.75,3) arc[x radius=0.25,y radius=0.05,start angle=-180,end angle=0];
            \draw[red, dashed]
                (3.75,3) arc[x radius=0.25,y radius=0.05,start angle=180,end angle=0];
            \reddiscfromcentral{-1}{sqrt(3)/4}
            \reddiscfromcentral{-1.7}{-sqrt(3)/4}
            \reddiscfromcentral{2}{0}
            \fill[red, opacity=0.6]
                (3.75,3) arc[x radius=0.25,y radius=0.05,start angle=-180,end angle=0] -- cycle;
            \fill[red, opacity=0.6]
                (3.75,3) arc[x radius=0.25,y radius=0.05,start angle=180,end angle=0] -- cycle;
            \fill (4,1) circle (0.6 pt) node[below right=-2pt] {\tiny $0$};
        \end{scope}
    \end{tikzpicture}
    \caption{The Lefschetz fibrations on $A_{p-1}$ and on $B_{p,q}$. Also shown are the vanishing thimbles and vanishing paths.}
    \label{fig:Vanishing_thimbles}
  \end{center}
\end{figure}
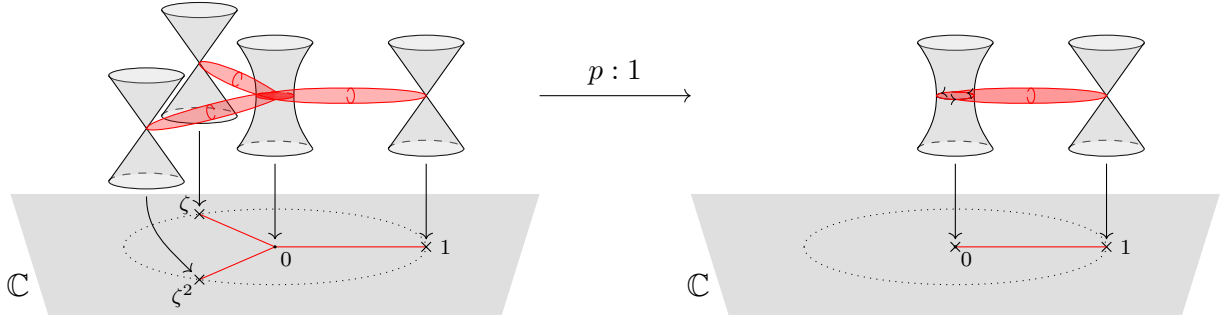

\subsection{Pintwists}
\label{subsec:pin_twist}

We now introduce the map $\tau_{p,q}\in \symp_c(B_{p,q})$ which will act as the analogue of the Dehn--Seidel twist about a Lagrangian sphere.
It was originally defined by Buck \cite[Proposition 5.4]{Bu25}, adapting arguments from \cite{Se20}, and we will follow his approach.
Since $\tau_{p,q}$ has only recently been introduced, we repeat the general strategy of the definition for the benefit of the reader, and refer to \cite{Bu25} for details.
First, we define a twist $\tau_{p-1}\in \symp_c(A_{p-1})$ on $A_{p-1}$ and then push it down to $\symp_c(B_{p,q})$ via the universal $p$-fold covering.
Therefore, most properties of $\tau_{p,q}$ will be derived from those of $\tau_{p-1}$.

For a thorough discussion of how one defines the symplectic monodromy of smoothings of singularities, we refer to \cite{Bu25}, as well as \cite[Section 6.3]{McDSa:Intro} for a more introductory discussion of symplectic connections.

\begin{definition}\label{def: pintwist}
    The map $\tau_{p-1}\in \symp_c(A_{p-1})$ is defined to be the symplectic monodromy map associated to the $A_{p-1}$ singularity.
\end{definition}

\begin{remark}
    Recall that a priori the symplectic monodromy of the fibration $(z_1,z_2,z_3)\rightarrow z_1z_2-z_3^p$ will have non-compact support.
    Seidel's construction in \cite{Se20} uses a cut-off function to produce the compactly supported map $\tau_{p-1}$.
    Also, notice that $\tau_{p-1}$ is very far from generating $\symp_c(A_{p-1})$.
    Indeed, as was shown in \cite{Ev14} and \cite{Wu14}, $\symp_c(A_{p-1})$ is weakly homotopy equivalent to the braid group with $p$ strands.
\end{remark}

Since the $\frac{1}{p}(1,-1,q)$-action, as defined in \eqref{eq:action_weights}, is unitary, it preserves the horizontal connection, and therefore $\tau_{p-1}$ is equivariant with respect to this action and it descends to $B_{p,q}$.

\begin{definition}[Pintwist]
    The \textit{pintwist} $\tau_{p,q}\in \symp_c(B_{p,q})$ is the map induced by ${\tau}_{p-1}$ after taking the quotient of $A_{p-1}$.
\end{definition}

\begin{remark}
    The pintwist can also be understood intrinsically: from the perspective of \cref{remark: Bpq as smoothed cqs}, the pintwist is the monodromy of the smoothing of the $\frac{1}{p^2}(1,pq-1)$ cyclic quotient singularity.
\end{remark}

We now collect from \cite[Section 5.2]{Bu25} the main properties of the pintwist that we will need later.
Points $(1)$ and $(3)$ follow by direct adaptations of arguments in \cite[Lemma 5.6]{Bu25}. 

\begin{lemma}[\normalfont{Basic properties of pintwists}]\label{lem: basic prop pintwists}
    Let $\tau_{p-1}: A_{p-1}\to A_{p-1}$ be the symplectic monodromy map and let $\tau_{p,q}:B_{p,q}\rightarrow B_{p,q}$ be the pintwist.  
    \begin{enumerate}
        \item 
            The map $\tau_{p-1}$ preserves the central cylinder $K_{p-1}$ and $\tau_{p-1}|_{K_{p-1}}=\tau$, where $\tau\in \symp_{c}(K_{p-1})$ is the usual Dehn twist.
        \item 
            The pintwist preserves the central cylinder $K_{p,q}$.
            In particular, $\tau_{p,q}\vert_{K_{p,q}}$ is isotopic to $\tau^p $, where $\tau\in \symp_c(K_{p,q})$ is the usual Dehn twist.
        \item 
            The standard pinwheel $\Lstd$ is Lagrangian, and thus Hamiltonian, isotopic to $\tau_{p,q}( \Lstd )$.
        \end{enumerate}
\end{lemma}

\begin{proof}
    The first point follows from the fact that the natural symplectic connection (coming from the symplectic form on $\C^3$) of the fibration $(z_1,z_2,z_3)\rightarrow z_1z_2-z_3^p$ restricts to a symplectic connection on the sub-fibration given by $K_{p-1}^w=\{z_1z_2-z_3^p=w\}\cap \{z_3=0\}$, where $K_{p-1}^1$ is the central cylinder.
    Therefore, the symplectic monodromy $\tau_{p-1}$ induces a symplectic monodromy of cylinders, which can be seen to be the usual monodromy of the smoothing of the double line to the smooth conic.
    
    The second point follows by inspection of the commutative diagram:
    \[
    \begin{tikzcd}
    	{(A_{p-1},K_{p-1})} &&& {(A_{p-1},K_{p-1})} \\
    	\\
    	{(B_{p,q},K_{p,q})} &&& {(B_{p,q},K_{p,q})}
    	\arrow["{\tau_{p-1}}"{description}, from=1-1, to=1-4]
    	\arrow["{\frac{1}{p}(1,-1,q)}"{description}, from=1-1, to=3-1]
    	\arrow["{\frac{1}{p}(1,-1,q)}"{description}, from=1-4, to=3-4]
    	\arrow["{\tau_{p,q}}"{description}, from=3-1, to=3-4]
    \end{tikzcd}
    \]
    Since $\tau_{p-1}$ preserves $K_{p-1}$, after taking the quotient we see that $\tau_{p,q}$ preserves $K_{p,q}$.
    The map $\tau_{p-1}|_{K_{p-1}}$ twists a latitude of $K_{p-1}$ once.
    Since $K_{p-1}\rightarrow K_{p,q}$ is a $p$-fold cover, $\tau_{p,q}|_{K_{p,q}}$ induces $p$ twists of the latitude of $K_{p-1}$ and is therefore a $p$th power of the usual Dehn twist.
    
    The last point follows verbatim from the proof of \cite[Lemma 5.6]{Bu25}, by observing that the arguments work for a vanishing thimble in $B_{p,q}$ over any path $\gamma$ in the base, and in particular, the pinwheel $L^{\std}_{p,q}$.
    Recall that $L^{\std}_{p,q}$ is the vanishing thimble of the Lefschetz fibration on $B_{p,q}$, which lives over a path $\gamma$ connecting the values $0$ and $1$.
    The \textit{central twist} $\tau$ defined in \cite[Section 5.1]{Bu25}, which is a mapping class of the disc with a marking at $1$, fixes $\gamma$; therefore $\tau_{p,q}(L^{\std}_{p,q})$ is Lagrangian isotopic to~$L^{\std}_{p,q}$. See also \cref{fig:Vanishing_thimbles}.
    
    The Lagrangian isotopy may be upgraded to a Hamiltonian one by appealing to Banyaga's Symplectic isotopy extension theorem (\cite{Ba78}, see also \cite[Theorem 3.3.2]{McDSa:Intro}), since all positive-degree rational cohomology groups of $B_{p,q}$ and of its $(p,q)$-pinwheels vanish.\footnote{One could also show that the usual Flux arguments for Lagrangian isotopies carry through, thus only using that $H^1(L_{p,q},\Q)=0$.}
\end{proof}

\subsection{Resolutions, rational blow-ups and neck-stretching}
\label{subsec:res_blow_neck}

This subsection closely follows the setup of \cite[Sections 2 and 3]{ABEHS25}.

Let $0<a<n$ be two coprime integers and consider the action of the group of $n$th roots of unity $\bm{\mu}_n$ on $\C^2$ with weights $(1,a)$, i.e.\ $\zeta \cdot (z_1,z_2)=(\zeta z_1,\zeta^a z_2)$ for $\zeta\in \bm{\mu}_n$.
Denote the quotient of $\C^2$ by this action by $\C^2/_{1/n(1,a)}$.
This quotient is a toric orbifold, whose singular point is modelled on the cyclic quotient singularity $\frac{1}{n}(1,a)$.
Since the action is K\"ahler, the quotient is an orbifold which admits a K\"ahler form.
Depending on the context, we will focus more on the complex structure or the symplectic form.
\textit{Wahl singularities} are just the family of cyclic quotient singularities of the form $\frac{1}{p^2}(1,pq-1)$.

\begin{remark}\label{remark: Bpq as smoothed cqs}
    By a deep result \cite{KoShBa88}, Wahl singularities are the cyclic quotient singularities which admit $\mathbb{Q}$-Gorenstein smoothings whose Milnor fibre is a rational homology ball.
    Such a Milnor fibre is the rational homology ball $B_{p,q}$ introduced above, and the $(p,q)$-pinwheel is the vanishing cycle of the smoothing.
\end{remark}

Denote $\co_{p,q}:=\C^2/_{1/p^2(1,pq-1)}$ and recall that its moment image $\Delta_{p,q}$ is the wedge bounded by the vectors $(0,1)$ and $(p^2,pq-1)$, i.e.\ the same as $\ATF_{p,q}$ without the branch cut and the node. 
This is discussed, for example, in \cite[Chapter 3]{Ev23:Book}.
By analogy with the almost toric base diagram $\ATF_{p,q}(\alpha,\beta)$, define the compact toric orbifold $\co_{p,q}(\alpha,\beta)$ via its moment image, which is shown in \cref{fig:toricOpq} (a), for positive real numbers $\alpha$ and $\beta$.
Since we will be interested mainly in uniqueness questions of Lagrangian pinwheels, we will fix for $\lambda >0$ the "shape" $\Delta_{p,q}(\lambda,\lambda)$ in the moment image $\Delta_{p,q}$ and denote the corresponding compact orbifold by $\cb_{p,q}(\lambda)$, whose moment image is shown in \cref{fig:toricOpq}~(b).
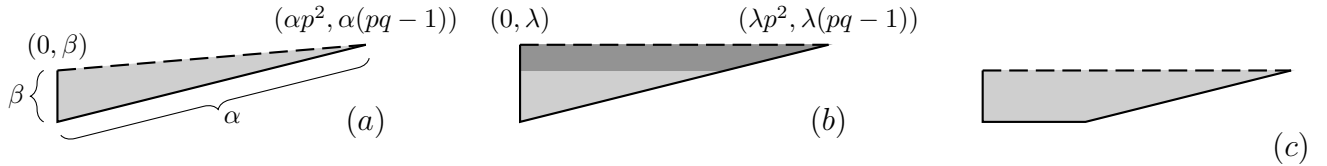
\begin{figure}[htb]
    \begin{center}   
        \begin{tikzpicture}[scale=0.68]
            \begin{scope}[shift={(-11,0)}]
                \node at (6,0) {$(a)$};
                \filldraw[lightgray,opacity=0.75] (0,1) -- (0,0) -- (6,1.5) -- cycle;
                \draw[thick] (0,1) -- (0,0) -- (6,1.5);
                \draw [decorate,decoration={brace,amplitude=5pt,raise=1ex}] (0,0) -- (0,1) node[midway,xshift=-3ex]{\footnotesize $\beta$};
                \draw [decorate,decoration={brace,amplitude=5pt,raise=1ex}] (6,1.5) -- (0,0) node[midway,xshift=1.5ex,yshift=-2.5ex]{\footnotesize $\alpha$};
                \node at (6,1.5) [above] {\footnotesize $(\alpha p^2,\alpha(pq-1))$};
                \node at (0,1) [above] {\footnotesize $(0,\beta)$};
                \draw[thick,dash pattern=on 7pt off 3pt] (0,1) -- (6,1.5);
            \end{scope}
            \begin{scope}[shift={(-2,0)}]
                \node at (6,0) {$(b)$};
                \filldraw[lightgray,opacity=0.75] (0,1.5) -- (0,0) -- (6,1.5) -- cycle;
                \filldraw[gray,opacity=0.75] (0,1.5) -- (0,1) -- (4,1) -- (6,1.5) -- cycle;
                \draw[thick] (0,1.5) -- (0,0) -- (6,1.5);
                \draw[thick,dash pattern=on 7pt off 3pt] (0,1.5) -- (6,1.5);
                \node at (6,1.5) [above] {\footnotesize $(\lambda p^2,\lambda(pq-1))$};
                \node at (0,1.5) [above] {\footnotesize $(0,\lambda)$};
            \end{scope}
            \begin{scope}[shift={(7,-0.5)}]
                \node at (6,0) {$(c)$};
                \filldraw[lightgray,opacity=0.75] (0,1.5) -- (0,0.5) -- (2,0.5) -- (6,1.5) -- cycle;
                \draw[thick] (0,1.5) -- (0,0.5) -- (2,0.5) -- (6,1.5);
                \draw[thick,dash pattern=on 7pt off 3pt] (0,1.5) -- (6,1.5);
            \end{scope}
        \end{tikzpicture}
        \caption{(a) The moment image $\Delta_{p,q}(\alpha,\beta)$ of the compact toric orbifold $\co_{p,q}(\alpha,\beta)$.
        (b) The moment image of $\mathcal{B}_{p,q}(\lambda)$ with a region used to define the neck-stretching setup shaded in darker gray. (c) The moment image for a choice of Kähler structure on $\widetilde{U}_\std$. All of the figures are illustrations of the situation in the $(p,q)=(2,1)$ case.}
        \label{fig:toricOpq}
    \end{center}
\end{figure}

Cyclic quotient surface singularities admit a unique \textit{minimal resolution}, meaning that any other resolution is obtained from the minimal one via a sequence of complex blow-ups.
To describe the minimal resolution of $\C^2/_{1/n(1,a)}$, it is enough to prescribe the self-intersections of the spheres in the chain of exceptional spheres $\mathcal{C}$ introduced by the resolution.
A minimal resolution is given by a chain of spheres $\mathcal{C}=(C_1,\dots, C_m)$ with self-intersections $C_i^2=-b_i$, where $[b_1,\dots,b_m]$ is the \textit{Hirzebruch–Jung continued fraction
expansion} (HJ continued fraction in the following) of $n/a$, i.e.\
\[[b_1,\ldots,b_m]\coloneqq b_1-
\frac{1}{b_2 - \frac{1}{\ddots - \frac{1}{b_m}}}=\frac{n}{a}.\]
In the case of Wahl singularities the HJ continued fraction of $p^2/(pq-1)=[b_1,\ldots,b_m]$ is called a \textit{Wahl chain}.
If the context is clear, we will also refer to the chain of rational curves appearing in the minimal resolution of a Wahl singularity as a \textit{Wahl chain}.
Before we move on to discuss the rational blow-up and the neck-stretching setup that we are going to use later, let us record some useful algebraic properties of HJ continued fractions.
Recall that the \textit{dual} HJ continued fraction of $\frac{p}{q}=[x_1,\ldots,x_r]$ is given by $\frac{p}{p-q}=[y_1,\ldots,y_s]$.

\begin{lemma}[\normalfont{\cite[Proposition-Definition 2.3]{UrZu25:BirMarkov} and \cite[Section 1]{Ur25:Lecture}}]
\label{lma:HJ_facts}
    Assume that $\frac{p}{q}=[x_1,\ldots,x_r]$ and $\frac{p}{p-q}=[y_1,\ldots,y_s]$.
    Denote by $[q]^{-1}$ the unique integer $0<[q]^{-1}<p$ that satisfies $[q]^{-1}q\equiv 1 \mod{p}$.
    \begin{enumerate}[label=\roman*)]
        \item 
            We have $ [x_1,\ldots,x_r,1,y_s,\ldots,y_1]=0$ and $\frac{p}{[q]^{-1}}=[x_r,\ldots,x_1]$.
        \item 
            The Wahl chain, as well as its dual, can be computed from the HJ continued fractions of $\frac{p}{q}$ and $\frac{p}{p-q}$ via
            \begin{equation}\label{lma:Hj_facts_p2}
                \frac{p^2}{pq-1}=[x_1,\ldots,x_{r-1},x_r+y_s,y_{s-1},\ldots,y_1] 
            \end{equation}
            and
            \begin{equation}\label{lma:Hj_facts_p2_dual}
                \frac{p^2}{p(p-q)+1}=[y_1,\ldots,y_s,2,x_r,\ldots,x_1].
            \end{equation}
        \item 
            We have
            \begin{equation}\label{eq:HJ_facts_matrices}
                \begin{pmatrix}
                    p & -[q]^{-1}\\
                    q & \frac{1-q[q]^{-1}}{p}
                \end{pmatrix}
                =
                \begin{pmatrix}
                    x_1 & -1\\
                    1 & 0
                \end{pmatrix}
                \dots 
                \begin{pmatrix}
                    x_r & -1\\
                    1 & 0
                \end{pmatrix}
                .
            \end{equation}
    \end{enumerate}  
\end{lemma}

Since we will mainly be interested in resolutions of $\cb_{p,q}(\lambda)$, let us define a resolution in this case rigorously.

\begin{definition}\label{dfn:resolution}
    A \textit{chain-shaped resolution} of $\mathcal{B}_{p,q}(\lambda)$ is given by a tuple $((\til{U},\til{J}),\pi)$, where $(\til{U},\til{J})$ is a smooth complex manifold and $\pi\colon\til{U}\to \mathcal{B}_{p,q}(\lambda)$ is a holomorphic map that satisfies
    the following property: 
    $\til{U}$ contains a linear chain of embedded rational complex curves $\cc=\{C_1,\ldots,C_m\}$ such that $\pi(\cc)=x$, where $x \in \mathcal{B}_{p,q}(\lambda)$ is the unique orbifold point, and the restriction
    $$
    \til{U}\setminus\cc \to \mathcal{B}_{p,q}(\lambda)\setminus\{x\}
    $$
    is a biholomorphism.
\end{definition}

\begin{remark}
    Since all the resolutions that we will encounter are chain-shaped, a resolution is always understood to be chain-shaped in the following.
\end{remark}

\begin{example}
    The minimal resolution of $\cb_{p,q}(\lambda)$ is just given by the minimal subdivision of the fan spanned by the inward normals to the two edges of $\Delta_{p,q}$.
    This minimal subdivision can be computed by the HJ continued fraction of $p^2/(pq-1)=[b_1,\ldots,b_m]$ and determines the resolution as a normal toric variety, see \cite[Definition 2.1.4]{ABEHS25}.
    We will denote this toric variety by $\til{U}_\std$.
\end{example}

Choosing appropriate additional data consisting of areas for the exceptional spheres in the minimal resolution and a potential function, we can induce a Kähler structure on $\til{U}_\std$.
We will denote the corresponding Kähler form by $\til{\omega}$.
Moreover, the usual moment image of a minimal resolution of $\Delta_{p,q}(\lambda,\lambda)$ obtained by symplectic cuts, as shown in \cref{fig:toricOpq} (c), is the moment image of the symplectic manifold $(\til{U}_\std,\til{\omega})$.\footnote{Of course the moment image will depend on the choices of data on $\til{U}_\std$.}

We will now define the \textit{rational blow-up} in a way that naturally incorporates the neck-stretching setup needed later.
Let $N$ be a neighbourhood of $\partial B_{p,q}(\lambda)$ corresponding to one in $\cb_{p,q}(\lambda)$ as shown in \cref{fig:toricOpq}.
Using such an $N$, we can define a complex structure $J_N$ on $N$ by pulling back the natural complex structure on $\cb_{p,q}(\lambda) \subseteq \co_{p,q}$.
\begin{definition}\label{def:acs_neck}
    Assume that $\iota:B_{p,q}(\lambda) \hookrightarrow (X,\omega)$ is a symplectic embedding. 
    Then define
    $$
    \cj_N:=\{J \in \mathcal{J}_\tau(X,\omega) \mid J|_{\iota(N)} = \iota_*J_N \}
    $$
    to be the space of all $\omega$-tame almost complex structures on $X$ that agree with $\iota_*J_N$ on $\iota(N)$. 
\end{definition}
In particular, the almost complex structures $J \in \cj_N$ are adjusted to the neck $N$ and therefore yield neck-stretching sequences of tame almost complex structures: starting with $J=:J_0 \in \cj_N$ we obtain a neck-stretching sequence $\{J_t\}_t \subseteq \mathcal{J}_\tau(X,\omega)$.\footnote{For details on this, see \cite{BEHWZ03, CieMoh05, ElGiHo00}. Neck-stretching along the contact boundary of normal neighbourhoods of Lagrangian pinwheels is discussed in \cite{Vi16, ABEHS25, EvSm18}.}

Assume that $\iota:B_{p,q}(\lambda) \hookrightarrow (X,\omega)$ is a symplectic embedding and denote $U:= \iota(B_{p,q}(\lambda))$ and $V:= X \setminus U$.
Moreover, let $\overline{U}$ and $\overline{V}$, respectively, be the symplectic completions of $U$ and $V$.
By considering the neck $N$ and its completion along the concave end we obtain a complex manifold $(\overline{N}_-,\overline{J}_N)$ that is isomorphic to a punctured neighbourhood of the orbifold $\co_{p,q}$.
Denote the orbifold obtained by adding the point at negative infinity to $\overline{V}$ by $(\widehat{X},\widehat{J})$.
Note that $(\widehat{N},\widehat{J})$, which we define to be the orbifold obtained by adding the point at negative infinity to $(\overline{N}_-,\overline{J}_N)$, is a copy of $\cb_{p,q}(\lambda)$ inside~$(\widehat{X},\widehat{J})$.

\begin{definition}[Rational blow-up]
    Given a symplectic embedding $\iota: B_{p,q}(\lambda) \hookrightarrow (X,\omega)$, a resolution $\til{X}$ of $\widehat{X}$ is given by exchanging $\widehat{N} \subseteq \widehat{X}$ with the resolution $\til{U}_\std$ of $\cb_{p,q}(\lambda)$.
    The symplectic manifold thereby obtained, $\til{X}=V \cup \til{U}_\std$, is called the \textit{rational blow-up} of $X$ along $\iota$.
\end{definition}

\begin{remark}
    The symplectic form $\til{\omega}$ on the rational blow-up $\til{X}$ coincides with $\omega|_V$ on $V$ and with $\til{\omega}$ in $\til{U}_\std$.
    Moreover, this procedure also respects the almost complex geometry in the sense that $\til{X}$ carries an almost complex structure $\til{J}$ that coincides with $J|_V$ on $V$ and agrees with the integrable complex structure on $\til{U}_\std$.
\end{remark}

The setup we are now in is summarized in the following figure, taken from \cite{ABEHS25}.

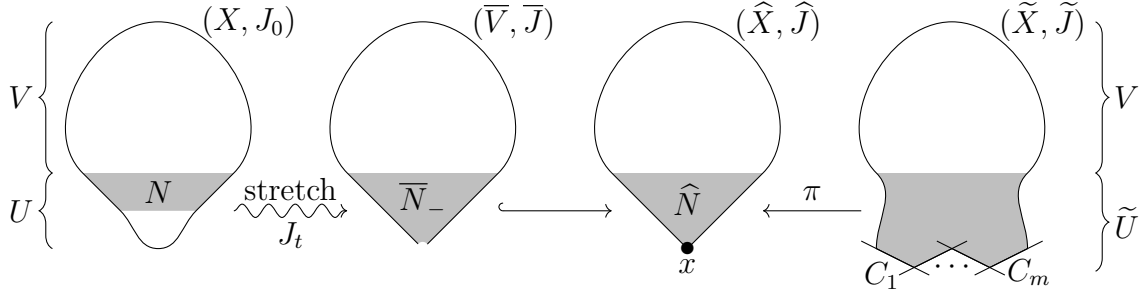
\begin{figure}[htb]
  \begin{center}
    \begin{tikzpicture}
      \filldraw[draw=none,fill=lightgray] (0,0)  -- (-0.5,0.5) -- (1.5,0.5) -- (1,0) -- cycle;
      \draw (0,0) -- (-0.5,0.5) to[out=135,in=180] (0.5,2.5) to[out=0,in=45] (1.5,0.5) -- (1,0) to[out=-135,in=0]
      (0.5,-0.5) to[out=180,in=-45] (0,0);
      \node at (0.5,0.25) {\(N\)};
      \draw[decorate,decoration={brace,amplitude=5pt}] (-0.9,0.5) -- (-0.9,2.5) node[midway,xshift=-1em] {\(V\)};
      \draw[decorate,decoration={brace,amplitude=5pt}] (-0.9,-0.5) -- (-0.9,0.5) node[midway,xshift=-1em] {\(U\)};
      \node at (1.7,2.5) {\((X,J_0)\)};
      \begin{scope}[shift={(3.5,0)}]
        \filldraw[draw=none,fill=lightgray] (0.5,-0.5)  -- (-0.5,0.5) -- (1.5,0.5) -- (1,0) -- cycle;
        \draw (0,0) -- (-0.5,0.5) to[out=135,in=180] (0.5,2.5) to[out=0,in=45] (1.5,0.5) -- (1,0) -- (0.5,-0.5) -- cycle;
        \node at (0.5,0.15) {\(\overline{N}_-\)};
        \node at (1.7,2.5) {\((\overline{V},\overline{J})\)};
        \node[white] at (0.5,-0.5) {\(\bullet\)};
      \end{scope}
      \begin{scope}[shift={(7,0)}]
        \filldraw[draw=none,fill=lightgray] (0.5,-0.5)  -- (-0.5,0.5) -- (1.5,0.5) -- (1,0) -- cycle;
        \draw (0,0) -- (-0.5,0.5) to[out=135,in=180] (0.5,2.5) to[out=0,in=45] (1.5,0.5) -- (1,0) -- (0.5,-0.5) -- cycle;
        \node at (1.7,2.5) {\((\ha{X},\ha{J})\)};
        \node (p) at (0.5,-0.5) {\(\bullet\)};
        \node at (0.5,0.15) {\(\ha{N}\)};
        \node at (p) [below] {\(x\)};
      \end{scope}
      \begin{scope}[shift={(10.5,0)}]
        \filldraw[draw=none,fill=lightgray] (1.5,0.5) to[out=-135,in=90] (1.5,-0.5) --
        (1,-0.75) -- (0.5,-0.5) -- (0,-0.75) -- (-0.5,-0.5) to[out=90,in=-45] (-0.5,0.5);
        \draw (-0.5,0.5) to[out=135,in=180] (0.5,2.5)
        to[out=0,in=45] (1.5,0.5) to[out=-135,in=90] (1.5,-0.5) --
        (1,-0.75) -- (0.5,-0.5) -- (0,-0.75) -- (-0.5,-0.5) to[out=90,in=-45] (-0.5,0.5);
        \node at (1.7,2.5) {\((\til{X},\til{J})\)};
        \draw (-0.5-0.2,-0.5+0.1) -- (0+0.2,-0.75-0.1) node
        [pos=0.3,below] {\(C_1\)};
        \draw (0-0.2,-0.75-0.1) -- (0.5+0.2,-0.5+0.1);
        \draw (0.5-0.2,-0.5+0.1) -- (1+0.2,-0.75-0.1);
        \draw (1-0.2,-0.75-0.1) -- (1.5+0.2,-0.5+0.1) node
        [pos=0.8,below] {\(C_m\)};
        \node at (0.5,-0.5) [below] {\(\cdots\)};
        \draw[decorate,decoration={brace,amplitude=5pt}] (2.4,2.5) -- (2.4,0.5) node[midway,xshift=+1em] {\(V\)};
        \draw[decorate,decoration={brace,amplitude=5pt}] (2.4,0.5) -- (2.4,-0.8) node[midway,xshift=1em] {\(\til{U}\)};
      \end{scope}
      \draw[->] (9.8,0) -- (8.5,0) node [midway,above] {\(\pi\)};
      \draw[decoration=snake,decorate,->] (1.5,0) -- (3,0) node
      [midway,above] {stretch};
      \node at (2.25,0) [below] {\(J_t\)};
      \draw[right hook->] (5,0) -- (6.5,0);
    \end{tikzpicture}
    \caption{The relations between the spaces defined in \cref{subsec:res_blow_neck}.}
    \label{fig:spaces}
  \end{center}
\end{figure}

The upshot of this setup is the following lemma for pseudoholomorphic curves in the three manifolds $(\overline{V},\overline{J}), (\ha{X},\ha{J})$ and $(\til{X},\til{J})$.

\begin{lemma}[\normalfont{\cite[Lemma 3.1.7/3.1.8]{ABEHS25}}]
\label{lma:bijections}
    There are bijections between the following collections of objects:
    (a) irreducible finite-energy punctured \(\overline{J}\)-holomorphic curves \(C\) in \(\overline{V}\);
    (b) irreducible orbifold \(\ha{J}\)-holomorphic curves \(\ha{C}\) in \(\ha{X}\); and 
    (c) irreducible \(\til{J}\)-holomorphic curves \(\til{C}\) in \(\til{X}\) other than \(C_1,\ldots,C_m\).
    Moreover:
    \begin{enumerate}[label=\roman*)]
        \item 
            Any irreducible $\til{J}$-holomorphic curve in $\til{X}$ which is not one of the $C_i$ must enter the interior of~$V$.
        \item 
            By choosing $J$ generically on $V$, we can ensure that all curves that enter $V$ are regular so that the only non-regular $\til{J}$-holomorphic curves are amongst the $C_1,\ldots,C_m$.
    \end{enumerate}
\end{lemma}

\subsection{The topology of the rational blow-up}
\label{subsec:top_rat_blow}

To understand the topology of $\widetilde{X}$, we want to leverage the fact that the boundary of $B_{p,q}(\lambda)$ is given by a rational homology sphere, since it is diffeomorphic to the lens space $L(p^2,pq-1)$.
Therefore, the rational homology of $\widetilde{X}$ splits into two parts: the rational homology of $\widetilde{U}_\std$ and the rational homology of $V$.
We will mostly do this in order to compute intersection numbers of curves in $\widetilde{X}$ by computing them locally in $\widetilde{U}_\std$ and~$V$ and to express the Chern class of $\widetilde{X}$ in terms of the Chern classes of $\widetilde{U}_\std$ and $V$.
Given a class $A\in H_2(\widetilde{X};\Z)$, let $A_\Q$ be the corresponding class in $H_2(\widetilde{X};\Q)$. 
\begin{lemma}[\normalfont{\cite[Lemma 3.2.2]{ABEHS25}}]
\label{lma:new_splitting}
    Let $\widetilde{X}$ be the rational blow-up along a symplectic embedding $\iota:B_{p,q}(\lambda) \hookrightarrow (X,\omega)$.
    Given a class \(A\in H_2(\widetilde{X};\Z)\) there exist rational numbers \(a_1,\ldots,a_m\) and a class \(A_X\in H_2(X;\Z)\) such that
    \begin{equation}\label{eq:new_splitting}
        A_\Q =\frac{1}{p}A_X+\sum_{j=1}^ma_j[C_j].
    \end{equation}
    Moreover, the integral class $pA_X$ can be represented by a cycle contained in $V$.
\end{lemma}
The point of the lemma is that we can now determine the numbers $a_j$ just by knowing how $A$ intersects the classes $C_i$.
Indeed, let $\tau_j=A\cdot C_j$ and consider the tuples $\bm{\tau}=(\tau_1,\dots,\tau_m)$ and $\bm{a}=(a_1,\dots,a_m)$.
Define the intersection matrix of the Wahl chain by $M_{i,j}=C_i\cdot C_j$.
By construction, we have that $\bm\tau=M\bm{a}$. 
Therefore to determine the tuple $\bm{a}$, it is enough to compute $M^{-1}$, if we know the intersection profile of $A$ with the Wahl chain.

In the sequel it will be crucial not only to know the inverse of the intersection matrix of a Wahl chain, but also the inverse of a chain of exceptional curves introduced during the minimal resolution of a general cyclic quotient singularity.
\begin{lemma}[\normalfont{\cite[Lemma 3.2.4]{ABEHS25}}]
\label{lma:M_inverse}
    Suppose that $0<a<n$ are coprime integers and that the minimal resolution of the cyclic quotient singularity $\frac{1}{n}(1,a)$ is governed by the HJ continued fraction $\frac{n}{a}=[b_1,\ldots,b_m]$.
    Consider the intersection matrix $M_{i,j}=C_{i}\cdot C_j$ of the associated chain-shaped resolution, where $C_i^2=-b_i$.
    Then $M$ is invertible and the entries of its inverse $M^{-1}$ are given by
    $$
        M^{-1}_{ij}=
        \begin{cases}
            -(e_if_j)/n\quad\mbox{if }i\leq j\\
            -(e_jf_i)/n\quad\mbox{if }i>j
        \end{cases},
    $$
    where the integers $e_i$ and $f_i$ are defined via the negative continued fractions $e_i/e_{i-1} = [b_{i-1},\ldots,b_{1}]$ and $f_i/f_{i+1}=[b_{i+1},\ldots,b_m]$.
    We will call the numbers $e_i$ and $f_i$ left and right "accompanying numbers" in the following.\footnote{Here "left" and "right" should indicate that the $e$-sequence is defined from the prefixes, whereas the $f$-sequence is defined from the suffixes of $[b_1,\ldots,b_m]$.}
\end{lemma}

\begin{remark}
    \cite[Lemma 3.2.4]{ABEHS25} is stated only for Wahl chains, but the proof does not depend on the HJ continued fraction being Wahl.
    Recall from the proof of \cite[Lemma 3.2.4]{ABEHS25} that these accompanying numbers satisfy the recursive formulas:
    \begin{equation}\label{eq:rec_acomp}
        e_{i+1}=b_ie_i-e_{i-1}\qquad\text{and}\qquad f_{i+1}=b_if_i-f_{i-1}
    \end{equation}
    and the identities: 
    \begin{equation}\label{eq:id_acomp}
        e_{i+1}f_i-e_if_{i+1}=n \qquad\text{and}\qquad e_if_i-e_{i-1}f_{i+1}=a.
    \end{equation}
    The initial conditions for the recursions are
    \begin{equation}\label{eq:initial_acomp}
        e_0=0,\qquad e_1=1,\qquad f_0=n,\qquad f_1=a.
    \end{equation}
    It follows from the definition of the accompanying numbers and the recursive formula that the sequence $(e_i)_i$ is increasing while the sequence $(f_i)_i$ is decreasing, and that all consecutive accompanying numbers are coprime.
    Moreover, the extremal values of the accompanying numbers are given by
        \begin{equation}\label{eq:extremal_acomp}
        e_m=[a]^{-1},\qquad e_{m+1}=n,\qquad f_m=1,\qquad f_{m+1}=0,
    \end{equation}
    where $0<[a]^{-1}<n$ is the unique integer satisfying $[a]^{-1}a \equiv 1 \mod{n}$.
\end{remark}

The accompanying numbers are also very useful in determining the Chern class of a small neighbourhood of the chain of spheres $\mathcal{C}$ in the resolution of a singularity.
Let $\til{U}$ be a neighbourhood of $\mC$.
By Alexander--Lefschetz duality we obtain $H^2(\widetilde{U};\Q)\cong H_2(\widetilde{U};\Q)$, which implies that we can write the canonical class $K_{\widetilde{U}}$ of $\widetilde{U}$ as
\begin{equation*}
    K_{\widetilde{U}}=\sum_{j=1}^m k_jC_j,
\end{equation*}
where the $C_j$ are classes Poincaré dual to the corresponding exceptional curves.
\begin{definition}[Discrepancies]
    The coefficients $k_j$ are called the \textit{discrepancies} of the singularity.
\end{definition}

\begin{lemma}[\normalfont{\cite[Section 2.1.]{haktelurz}}]
\label{lma:discrepancies}
    Suppose that $0<a<n$ are two coprime integers.
    The discrepancies of a cyclic quotient singularity $\frac{1}{n}(1,a)$ can be computed from its accompanying numbers via
    \begin{equation}\label{eq:discrepancies_acomp}
        k_j=-1+\frac{e_j+f_j}{n}.
    \end{equation}
    Moreover, the discrepancies of Wahl singularities, i.e.\ cyclic quotient singularities of the form $\frac{1}{p^2}(1,pq-1)$, satisfy $k_j \in (-1,0]$.
\end{lemma}

\begin{remark}[\normalfont{\cite[Section 5.2]{EvSm20}}]
\label{rmk:splitting_Chern}
    Given a symplectic embedding $\iota: B_{p,q}(\lambda)\hookrightarrow X$, denote the rational blow-up along $\iota$ by $\widetilde{X}=V \cup \widetilde{U}_\std$, where $V=X\setminus \iota(B_{p,q}(\lambda))$, as before. 
    Mayer--Vietoris yields that $H^2(\widetilde{X},\Q)\cong H^2(V,\Q)\oplus H^2(\widetilde{U}_\std,\Q)$, since $\widetilde{U}_\std\cap V$ is just the boundary lens space $\Sigma_{p,q}$ of $B_{p,q}(\lambda)$, which is a rational homology sphere.
    By naturality of the Chern class, we have that the Chern class of $\widetilde{X}$ splits as the sum of the Chern classes of $\widetilde{U}_\std$ and $V$.
    Calculating the discrepancies $k_j$ amounts precisely to calculating the Chern class of $\widetilde{U}_\std$.
\end{remark}

\section{Part 1: Uniqueness up to symplectomorphism}
\label{sec: Uniq up to symp}

In \cref{sec:compactification} we introduce our compactification $X_{p,q}$ of $B_{p,q}$ and explain how the central cylinder $K_{p,q}$ becomes a fibre class under the compactification procedure, as well as how $X_{p,q}$ is birationally derived from a Hirzebruch surface.
In \cref{sec:trafo_regulation} we prove the existence of a regulation of $X_{p,q}$ (\cref{fig:Compactification_Bpq_regulation}).
We then show how this regulation "persists" under rationally blowing up $X_{p,q}$ along an embedding $\iota: B_{p,q}(\lambda)\hookrightarrow X$ (\cref{thm:regulation_of_tilX}).
Finally, in \cref{sec: constructing the symp} we conclude the proof of \cref{thrm_intro: B}.

\subsection{Compactifying $B_{p,q}$}
\label{sec:compactification}

We start by explaining a compactification of the rational homology ball $B_{p,q}$.
There is a lot of freedom in doing so, and we present the compactification most suited for our purposes.\footnote{For example, different ones are considered in \cite{BhuOno12,PaPaShUr18,Bu25}. However, it is not hard to understand how to pass from one to another via blow-ups and blow-downs of appropriate divisors. The one in \cite[Section 9.3]{Ev23:Book} is intimately related to ours and we will comment on this relation later.}
\begin{example}
    It is well known that $B_{2,1}\cong T^*\R P^2$ admits a compactification to $\C P^2$, as discussed in \cite{Ad25,PaPaShUr18}; that\footnote{The space $B_{2;1,1}$ fits into a larger family $B_{d;p,q}$, which specializes to $B_{p,q}$ for $d=1$. In line with \cref{remark: Bpq as smoothed cqs}, these can be characterized as the Milnor fibres of those cyclic quotient singularities that admit $\Q$-Gorenstein smoothings. See \cite[Section 7]{Ev23:Book} for details on this.} $B_{2;1,1} \cong T^*S^2$ compactifies to $S^2 \times S^2$, see \cite{Hi04,LiWu12}; and that the family $B_{n,1}$ admits compactifications to rational and ruled surfaces, as explained in \cite{AdHa25, Sym01:GenRatBlo}.    
\end{example}

Given the almost toric base diagram $\ATF_{p,q}$ of $B_{p,q}$, as shown in \cref{fig:atfBpq}, for a positive real number $\alpha>0$ let $X^{\orb}_{p,q}(\alpha)$ be the possibly singular symplectic manifold obtained by performing a horizontal symplectic cut at height $\alpha$, capping off the almost toric base diagram to a closed triangle. 
Notice that the left upper corner is Delzant while the upper right corner is non-Delzant if $(p,q) \neq (2,1)$.
Let $x_{\orb}$ be the orbifold point defined by the upper right corner.

\begin{lemma}
    The orbifold point $x_{\orb}$ is equivalent to the quotient singularity $\frac{1}{pq-1}(1,p^2)$.
\end{lemma}

\begin{proof}
    After a reflection along the $x=y$ axis, the non-Delzant corner is equivalent to the quotient singularity defined by the vectors $(0,1)$ and $(pq-1,p^2)$, which is the claimed quotient singularity.
\end{proof}

Now, we would like to minimally resolve the singularity $x_{\orb}$ in order to obtain a smooth compactification of $B_{p,q}$. 
However, since $pq-1<p^2$ we cannot directly apply the usual algorithm that computes the Hirzebruch--Jung continued fraction in order to determine a minimal resolution, as described for example in \cite[Chapter 9]{Ev23:Book}.
So we need to further modify the singularity to be in normal form. 

\begin{lemma}\label{le:minimal_res_orbifold_compactification}
    The quotient singularity $x_{\orb}$ is equivalent to $\frac{1}{pq-1}(1,p^2-d_0(pq-1))$, where $d_0$ is the unique positive integer such that $(pq-1)d_0<p^2\leq(pq-1)(d_0+1)$.
    In particular, the thereby introduced compactifying divisor has intersection profile according to the mixed continued fraction
    \[
        \frac{p^2}{pq-1}=[-d_0;d_1,\ldots,d_n]=d_0+\frac{1}{d_1-\frac{1}{\cdots-\frac{1}{d_{n}}}},
    \]
    i.e.\ the spheres $D_1,\ldots,D_n$ introduced in the process have self-intersection $-d_1,\ldots,-d_n$, while the sphere $D_0$ has self-intersection $+d_0$.
\end{lemma}

\begin{proof}
    Let $d_0$ be as in the statement of the lemma and define $0<l=p^2-d_0(pq-1)$. 
    Then, the pair of vectors $\{(0,1),(pq-1,p^2)\}$ is related to the pair of vectors $\{(0,1),(pq-1,l)\}$ by a shear along $(0,1)$.
    Therefore the quotient singularity $x_{\orb}$ is equivalent to the quotient singularity $\frac{1}{pq-1}(1,l)$. 
    Since by construction $l\leq pq-1$ the usual algorithm to compute a minimal resolution can be applied to $\frac{pq-1}{l}$. 
    This is exactly the result of
    \[
        \frac{p^2}{pq-1}=d_0+\frac{1}{\frac{pq-1}{l}}
    \]
    after taking the negative continued fraction expansion of $\frac{pq-1}{l}$.
\end{proof}

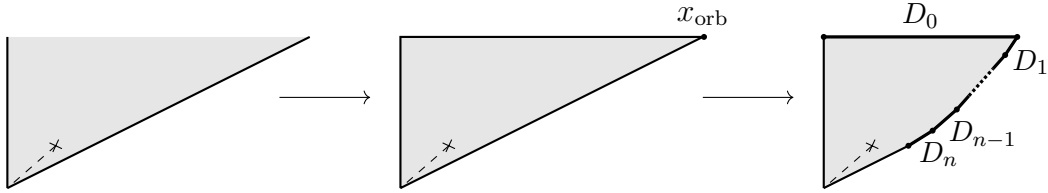
\begin{figure}[ht]
  \centering
  \begin{tikzpicture}[scale=0.8]
      \begin{scope}[shift={(-9,0)}]
        \fill[opacity=0.1] (0,2) -- (0,-0.5) -- (1,0) -- (5,2);
        \draw[thick] (0,2) -- (0,-0.50);
        \draw[thick] (0,-0.5) -- (5,2);
        \draw[dashed] (0.8,0.2) node[cross] {} -- (0,-0.5);
    \end{scope}
    
    \draw[->] (-4.5,1) -- node[anchor=south] {} node[anchor=north] {} (-3,1);
    
    \begin{scope}[shift={(-2.5,0)}]
        \fill[opacity=0.1] (0,2) -- (0,-0.5) -- (1,0) -- (5,2);
        \draw[thick] (0,-0.5) -- (0,2) -- (5,2) -- (0,-0.5);
        \fill (5.02,2) circle (1.4pt);
        \draw[thick, black] 
            (5,2) node[above]{$x_{\orb}$};
        \draw[dashed] (0.8,0.2) node[cross] {} -- (0,-0.5);
    \end{scope}

    \draw[->] (2.5,1) -- node[anchor=south] {} node[anchor=north] {} (4,1);
    
    \begin{scope}[shift={(4.5,0)}]
        \fill[opacity=0.1] (0,2) -- (0,-0.5) -- (1,0) -- (1.4,0.2) -- (1.8,0.45) -- (2.2,0.8) -- (3,1.7) -- (3.2,2);
        \draw[thick] (0,-0.5) -- (1.4,0.2) -- (1.8,0.45) -- (2.2,0.8);
        \draw[thick] (3,1.7) -- (3.2,2) -- (0,2) -- (0,-0.5);
        \fill 
            (1.4,0.2) circle (1.4pt)
            (1.8,0.45) circle (1.4pt)
            (2.2,0.8) circle (1.4pt)
            (3,1.7) circle (1.4pt)
            (3.2,2) circle (1.4pt)
            (0,2) circle (1.4pt);
        \draw[very thick] (1.4,0.2) -- node[below right, xshift=-5pt,yshift=3pt] {\small $D_{n}$} (1.8,0.45) -- node[below right, xshift=-3pt,yshift=4pt] {\small $D_{n-1}$} (2.2,0.8) -- (2.4,0.8+0.9/4);
        \draw[very thick, densely dotted] (2.4,0.8+0.9/4) -- (2.8,1.7-0.9/4);
        \draw[very thick] (2.8,1.7-0.9/4) -- (3,1.7) -- node[below right, xshift=-4pt,yshift=3pt] {\small $D_{1}$} (3.2,2);
        \draw[very thick] (3.2,2) -- node[above, yshift=-1pt] {\small $D_0$}  (0,2);
        \draw[dashed] (0.8,0.2) node[cross] {} -- (0,-0.5);
    \end{scope}

  \end{tikzpicture}
  \caption{The steps taken to compactify the rational homology ball $B_{p,q}$ to $X_{p,q}$. This introduces the compactifying divisor $\mathcal{D}_{p,q}=(D_0,\ldots,D_n)$.}
  \label{fig:Compactification_Bpq}
\end{figure}

\begin{remark}\label{re:type_of_quotient_sing}
    Note that in the computation of the type of the quotient singularity in \cref{le:minimal_res_orbifold_compactification} we went through a transformation with determinant equal to $-1$, i.e.\ the reflection along the $x=y$ axis.
    One can avoid this and just apply a shear along the branch cut given by 
    $$
        \begin{pmatrix} 1-pq & p^2\\ -q^2 & 1+pq \end{pmatrix}
    $$
    in the middle figure of \cref{fig:Compactification_Bpq}.
    Then the horizontal edge will point in the $(pq-1,q^2)$-direction after applying the shear and therefore we see that $x_\text{orb}$ is modelled on the quotient singularity $\frac{1}{pq-1}(1,q^2)$.
    Computing the Hirzebruch--Jung continued fraction, and therefore the minimal resolution of the singularity, gives the Hirzebruch--Jung continued fraction associated to the singularity $\frac{1}{pq-1}(1,p^2-d_0(pq-1))$ in reverse order.
    This is because $0<p^2-d_0(pq-1)\leq pq-1$ is the unique integer such that $q^2(p^2-d_0(pq-1))\equiv 1 \mod{pq-1}$ and hence the claim follows from \cref{lma:HJ_facts}.
    
    We will be switching between these different viewpoints freely, as they highlight different aspects of the compactification.
    For example, viewing the orbifold point as one that is modelled on the $\frac{1}{pq-1}(1,q^2)$ singularity allows for quick and easy computations, while the mixed continued fraction introduced in \cref{le:minimal_res_orbifold_compactification} captures the structure of the compactification more transparently.
\end{remark}

\begin{definition}\label{def:compactification}
    Define the compactification of the rational homology ball $B_{p,q}$ to be the closed symplectic manifold $\left(X_{p,q}(\alpha;\bm{\rho}),\omega(\alpha;\bm{\rho})\right)$ obtained by capping the almost toric base diagram $\ATF_{p,q}$ at height $\alpha>0$ and resolving the quotient singularity $x_\text{orb}$, giving the chain of spheres $D_0,\ldots,D_n$ areas according to $\bm{\rho} \in \mathbb{R}^{n+1}_{>0}$, i.e.\ $\int_{D_i}\omega(\alpha;\bm{\rho})=\rho_i$.
\end{definition}

\begin{remark}
    From now on "compactifying $B_{p,q}$" will mean compactifying $B_{p,q}$ according to \cref{def:compactification} unless specified otherwise, and we will often drop all the decorations in the definition above, whenever it is understood that we just consider some, meaning a suitable, compactification.
    We will write $(X_{p,q},\omega)$ for a compactification as above and write $\mathcal{D}_{p,q}=\{D_0,\ldots,D_n\}$ for the compactifying divisor.
    
    Moreover, even though the horizontal cut on $B_{2,1}$ does not produce a non-Delzant corner, we adopt the convention to introduce a compactifying divisor of the form $\mathcal{D}_{2,1}=(D_0,D_1)$, where $D_0^2=+3$ and $D_1^2=-1$, such that the compactification $X_{2,1}$ fits into the family $X_{n,1}$.
\end{remark}

\begin{example}\label{ex:numbers_HJ}
    We will use two examples to illustrate the theorems and lemmata pictorially throughout the paper.
    We give the relevant associated integers here. 
    If $(p,q)=(5,2)$ then the relevant Hirzebruch--Jung continued fraction and the corresponding $d_0$ are:
    $$
        \frac{p^2}{pq-1}=\frac{25}{9}=[3,5,2], \qquad \frac{pq-1}{q^2}=\frac{9}{4}=[3,2,2,2] \quad \text{and}\qquad d_0=c_1-1=2,
    $$
    where $c_1=3$ is the first component of the Wahl chain.
    This means that the Wahl chain $\mathcal{C}_{5,2}$ has intersection profile $(-3,-5,-2)$ and the compactifying divisor $\mathcal{D}_{5,2}$ has intersection profile $(+2,-2,-2,-2,-3)$.
    
    For the family $(p,q)=(n,1)$ we have:
    $$
        \frac{n^2}{n-1}=[n+2,\underbrace{2,\ldots,2}_{(n-2)\text{-times}}], \qquad \frac{n-1}{1}=[n-1] \quad \text{and}\qquad d_0=c_1-1=n+1.
    $$
\end{example}

We collect some small facts about the compactification we constructed.
The most crucial object is the embedded symplectic sphere that is obtained from the central cylinder $K_{p,q}$ under the compactification procedure, as we will show that this sphere is part of a well-behaved holomorphic foliation of $X_{p,q}$.

\begin{lemma}\label{lma:central_becomes_fibre}
    The central cylinder $K_{p,q}$, living over the toric boundary of $\ATF_{p,q}$, corresponds in the compactification to an embedded symplectic sphere of square zero.
\end{lemma}

\begin{proof}
    After applying a shear along the branch cut to the middle figure of \cref{fig:Compactification_Bpq}, the $\frac{1}{pq-1}(1,q^2)$ singularity will sit on the bottom left edge of the ATF triangle. 
    Then the first symplectic cut of the minimal resolution of the $\frac{1}{pq-1}(1,q^2)$ singularity is horizontal, meaning that the central cylinder yields a square zero embedded symplectic sphere after the compactification.
\end{proof}

\begin{remark}\label{rmk:properties_fibre}
    In the following we will therefore refer to this square zero embedded symplectic sphere $F$ as a "fibre" and call $F \in H_2(X_{p,q};\mathbb{Z})$ the "fibre class".
    See \cref{fig:Compactification_Bpq_regulation} for another representation of the ATF on the compactification $X_{p,q}$.
    
    Moreover, we have that $F\cdot L_{p,q}\equiv q \mod{p}$, which is a direct consequence of the fact that the Poincaré dual of $c_1(B_{p,q})$ is represented by $K_{p,q}$ and that $c_1(B_{p,q})=q \in H^2(B_{p,q};\mathbb{Z})\cong  \mathbb{Z}_p$ as shown in \cite[Section 2]{EvSm18}. This can also be deduced from the way that $F$ and $L_{p,q}$ intersect, as shown in \cref{fig:Compactification_Bpq_regulation}.
\end{remark}

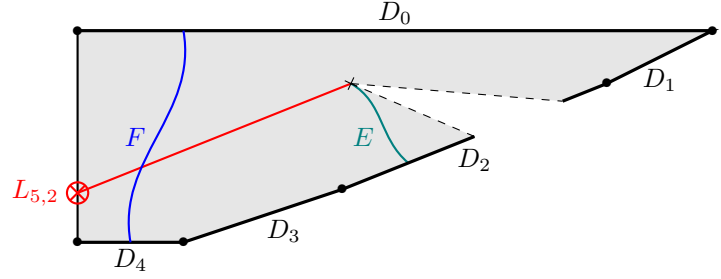
\begin{figure}[ht]
  \centering
  \begin{tikzpicture}[scale=0.7]
        \fill[opacity=0.1] 
            (0,4) -- (0,0) -- (2,0) -- (5,1) -- (15/2,2) -- (31/6,3) -- (15/2+5/3,2+2/3) -- (10,3) -- (12,4) -- (0,4);
        \draw[thick] 
            (0,4) -- (0,0);
        \draw[thick, teal] (31/6,3) to[out=330,in=140] (5+5/4,1.5);
        \node[text=teal] (E) at (5.4,2) {\footnotesize $E$};
        \draw[very thick] 
            (0,0) node{\tiny \(\bullet\)} -- node[anchor=north, yshift=2pt]{\footnotesize $D_{4}$} 
            (2,0) node{\tiny \(\bullet\)} -- node[anchor=north west, xshift=-3pt,yshift=3pt]{\footnotesize $D_{3}$} 
            (5,1) node{\tiny \(\bullet\)} -- (15/2,2) node[anchor=north]{\footnotesize $D_{2}$};
        \draw[very thick] 
            (15/2+5/3,2+2/3) -- 
            (10,3) node{\tiny \(\bullet\)} -- node[anchor=north]{\footnotesize $D_{1}$} 
            (12,4) node{\tiny \(\bullet\)} -- node[anchor=south, yshift=-2pt]{\footnotesize $D_{0}$} 
            (0,4) node{\tiny \(\bullet\)};
        \draw[thick, red]
            (31/6,3) -- (0,3-62/30) node[circlecross]{};
        \node[text=red] (L) at (-0.8,3-62/30) {\footnotesize $L_{5,2}$};
        \draw[thick, blue] (2,4) to[out=280,in=100] node[anchor=east]{\footnotesize $F$} (1,0);
        \draw[dashed] (15/2,2) -- (31/6,3) node[cross]{};
        \draw[dashed] (31/6,3) -- (15/2+5/3,2+2/3);
  \end{tikzpicture}
    \caption{The compactification $X_{5,2}$. This almost toric base diagram is related to the one partially shown on the right in \cref{fig:Compactification_Bpq} via a rotation of the branch cut. All the relevant objects of the compactification are included in the figure: (1) the compactifying divisor $\mathcal{D}_{5,2}=\{D_0,\ldots,D_4\}$; (2) the Lagrangian pinwheel $L_{5,2}$; (3) a fibre class $F$ and (4) the distinguished exceptional divisor $E$. In \cref{fig:compactificationBpq_birational} it is shown how $X_{5,2}$ is related to the Hirzebruch surface $\mathbb{F}_2$.}
  \label{fig:Compactification_Bpq_regulation}
\end{figure} 

\begin{lemma}\label{lma:distinguished_E}
    If $q\neq 1$ there is a visible exceptional divisor $E$ in the ATF defined on $X_{p,q}$ which connects the node to one of the edges appearing in the toric resolution of the $\frac{1}{pq-1}(1,q^2)$ singularity.
    We refer to this exceptional divisor as the distinguished exceptional class $E\in H_2(X_{p,q};\Z)$.
\end{lemma}

\begin{remark}
    The situation of the lemma is shown in the example $(p,q)=(5,2)$ in \cref{fig:Compactification_Bpq_regulation}.
    Recall from \cref{ex:numbers_HJ} that in the case $(p,q)=(n,1)$ the compactification is already a rational and ruled surface, because the compactification only introduces two compactifying divisors.
\end{remark}

\begin{proof}[Proof of \cref{lma:distinguished_E}]
    Assume that $\frac{p}{q}=[x_1,\ldots,x_r]$ and $\frac{p}{p-q}=[y_1,\ldots,y_s]$.
    It is easy to see that a resolution of the form $[x_1,\ldots,x_r,2,y_s,\ldots,y_1,1]$ is a non-minimal resolution of the $\frac{1}{pq-1}(1,q^2)$ singularity.\footnote{This resolution is obtained by introducing a vertical symplectic cut in the middle step of \cref{fig:Compactification_Bpq}. Then the thereby introduced singularity is exactly the dual singularity to $\frac{p^2}{pq-1}$, which has a minimal resolution as discussed in equation \eqref{lma:Hj_facts_p2_dual}.
    The compactification constructed in this manner is exactly the one used in \cite{Bu25} and \cite[Figure 9.11]{Ev23:Book}.
    The string is in the reverse order to the one given in \eqref{lma:Hj_facts_p2_dual} because of a reversal of orientation.\label{foot:compactification}}
    This resolution covers the minimal one by consecutively blowing down $(-1)$-spheres, and this blow-down procedure will contract the whole string $[y_1,\ldots,y_s]$ if and only if it is a string of the form $[2,\ldots,2]$, meaning that $(p,q)=(n,1)$.
    Note that in this case $\frac{n}{1}=[n]$, and since the connecting $(-2)$-sphere is also blown down in the process we see that this yields exactly the compactification $X_{n,1}$, which is a rational and ruled surface.
    If not, the string $[y_1,\ldots,y_s]$ contains a first entry $y_j$ that is not equal to $2$ and therefore
    \begin{equation}\label{eq:res_of_pq-1}
        \frac{1}{pq-1}(1,q^2)=[x_1,\ldots,x_r,2,y_s,\ldots,y_j-1].
    \end{equation}
    This concludes the proof, because $\frac{p}{q}=[x_1,\ldots,x_r]$ implies that the edge associated to the connecting $2$ in the HJ continued fraction of $\frac{1}{pq-1}(1,q^2)$ is pointing in the $(p,q)$-direction.
\end{proof}

\begin{corollary}\label{cor:compactification_birationally_derived}
    The compactification $X_{p,q}$ is obtained from the standard toric model of $\F_{d_0}$ via a sequence of toric blow-ups and a non-toric blow-up of a distinguished edge, meaning that it is birationally derived from the Hirzebruch surface $\F_{d_0}$.
\end{corollary}

\begin{remark}
    This sequence of blow-ups/downs is shown in \cref{fig:compactificationBpq_birational} for the example $(p,q)=(5,2)$.
\end{remark}

\begin{proof}[Proof of \cref{cor:compactification_birationally_derived}]
    In the case $(n,1)$ there is nothing to prove, since the compactification $X_{n,1}$ is already $\F_{d_0}$, as $d_0=n+1$ in this case according to \cref{ex:numbers_HJ}.
    Hence, we assume that $(p,q)\neq (n,1)$.
    In this case the minimal resolution of the singularity $\frac{1}{pq-1}(1,q^2)$ is given by 
    $$
        \frac{1}{pq-1}(1,q^2)=[x_1,\ldots,x_r,2,y_s,\ldots,y_j-1],
    $$
    as deduced in \eqref{eq:res_of_pq-1}.
    It follows from \cref{lma:distinguished_E}, see also \cref{fig:Compactification_Bpq_regulation}, that 
    $$[x_1,\ldots,x_r,1,y_s,\ldots,y_j-1]=0,$$ 
    i.e.\ that the HJ continued fraction is a zero continued fraction, and therefore there is a unique sequence of interior blow-downs to $[1,1]$.
    Now contracting the $(-1)$-sphere that intersects the negative section will leave the positive section unchanged, and therefore the "minimal model" of $X_{p,q}$ is $\F_{d_0}$.
\end{proof}

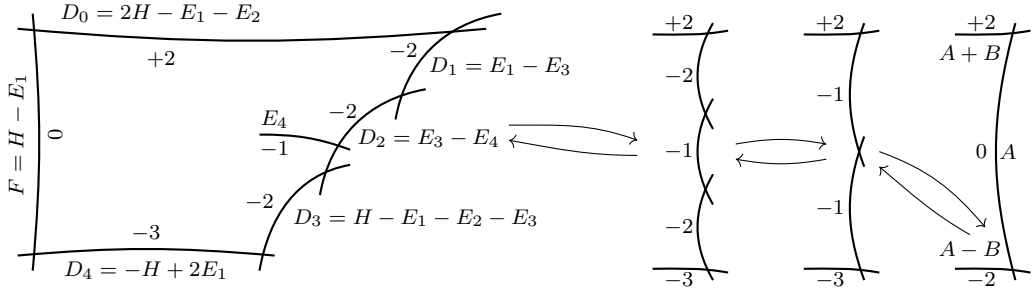
\begin{figure}[htb]
\begin{center}
    \begin{tikzpicture}
    \begin{scope}[shift={(0,-1.375)}]
    
    \draw[thick] (-0.2,0) to[out=5,in=175] (3.2,0);
    \node (d4) at (1.5,-0.2) {\tiny $D_4=-H+2E_1$};
    \node (d4s) at (1.5,0.3) {\tiny $-3$};

    \draw[thick] (0,-0.2) to[out=85,in=275] (0,3.2);
    \node[rotate=90] (f) at (-0.2,1.6) {\tiny $F=H-E_1$};
    \node[rotate=90] (fs) at (0.3,1.6) {\tiny $0$};

    \draw[thick] (-0.2,3) to[out=355,in=185] (6,3);
    \node (d0) at (1.7,3.2) {\tiny $D_0=2H-E_1-E_2$};
    \node (d0s) at (1.7,2.6) {\tiny $+2$};

    \draw[thick] (3,-0.2) to[out=80,in=190] (4.2,1.2);  
    \node[anchor=west] (d3) at (3.3,0.5) {\tiny $D_3=H-E_1-E_2-E_3$};
    \node (d3s) at (3,0.7) {\tiny $-2$};

    \draw[thick] (3.8,0.8) to[out=80,in=190] (5.2,2.2);  
    \node (d2s) at (4.1,1.9) {\tiny $-2$};
    \node[anchor=west] (d1) at (4.15,1.55) {\tiny $D_2=E_3-E_4$};

    \draw[thick] (4.8,1.8) to[out=80,in=190] (6.2,3.2); 
    \node (d1s) at (4.9,2.7) {\tiny $-2$};
    \node[anchor=west] (d1) at (5.1,2.5) {\tiny $D_1=E_1-E_3$};

    \draw[thick] (3,1.6) to[out=0,in=160] (4.2,1.4);
    \node (e4) at (3.2,1.8) {\tiny $E_4$};
    \node (e4s) at (3.2,1.4) {\tiny $-1$};

    \end{scope}
    
    \begin{scope}[shift={(9,-1.5)}]
    \draw[thick] (0,-0.2) to[out=120,in=240] (0,1.2);
    \draw[thick] (0,0.8) to[out=120,in=240] (0,2.2);
    \draw[thick] (0,1.8) to[out=120,in=240] (0,3.2);
    \draw[thick] (-0.8,-0.05) to[out=0,in=170] (0.2,-0.1);
    \draw[thick] (-0.8,3.05) to[out=0,in=190] (0.2,3.1);
    
    \node at (-0.45,0.5) {\tiny $-2$};
    \node at (-0.45,1.5) {\tiny $-1$};
    \node at (-0.45,2.5) {\tiny $-2$};
    \node at (-0.45,-0.2) {\tiny $-3$};
    \node at (-0.45,3.2) {\tiny $+2$};
    \end{scope}

    \draw[->] (6.3,0.35) to[out=0,in=160] (8,0.15);
    \draw[->] (8,-0.05) to[out=180,in=350] (6.3,0.15);
    
    \begin{scope}[shift={(11,-1.5)}]
    \draw[thick] (0,-0.2) to[out=110,in=250] (0,1.7);
    \draw[thick] (0,1.3) to[out=110,in=250] (0,3.2);
    \draw[thick] (-0.8,-0.05) to[out=0,in=170] (0.2,-0.1);
    \draw[thick] (-0.8,3.05) to[out=0,in=190] (0.2,3.1);
    
    \node at (-0.45,0.75) {\tiny $-1$};
    \node at (-0.45,2.25) {\tiny $-1$};
    \node at (-0.45,-0.2) {\tiny $-3$};
    \node at (-0.45,3.2) {\tiny $+2$};
    \end{scope}

    \draw[->] (9.3,0.10) to[out=10,in=170] (10.5,0.10);
    \draw[->] (10.5,-0.10) to[out=190,in=350] (9.3,-0.10);

    \begin{scope}[shift={(13,-1.5)}]
    \draw[thick] (0,-0.2) to[out=105,in=255] (0,3.2);
    \draw[thick] (-0.8,-0.05) to[out=0,in=170] (0.2,-0.1);
    \draw[thick] (-0.8,3.05) to[out=0,in=190] (0.2,3.1);
    
    \node at (-0.45,1.5) {\tiny $0$};
    \node at (-0.1,1.5) {\tiny $A$};
    \node at (-0.45,-0.2) {\tiny $-2$};
    \node at (-0.6,0.2) {\tiny $A-B$};
    \node at (-0.45,3.2) {\tiny $+2$};
    \node at (-0.6,2.8) {\tiny $A+B$};
    \end{scope}

    \draw[->] (11.2,0) to[out=340,in=130] (12.6,-1);
    \draw[->] (12.4,-1.1) to[out=150,in=320] (11.2,-0.2);
    
    \end{tikzpicture}
    \caption{This figure shows how the compactification $X_{5,2}$ is birationally derived from the Hirzebruch surface $\F_2$ via a sequence of blow-ups. On the right we only display the blow-down sequence of the broken fibre. The homology classes of the divisors and the fibre on the left are easily calculated from the sequence of blow-ups, using the usual identification $S^2\times S^2 \# \overline{\C P^2}\cong X_2$. In the language of \cref{lma:distinguished_E} this means that $E_4$ is the distinguished exceptional class.}
    \label{fig:compactificationBpq_birational}
\end{center}
\end{figure}
As a last point on the compactification $X_{p,q}$ we discuss some homological considerations that we will need later.
Assume that $\iota:B_{p,q}(\varepsilon)\hookrightarrow B_{p,q}$ is a symplectic embedding.
After choosing suitable compactification data, meaning that we compactify away from the embedding, we view $\iota$ as an embedding into the compactification $X$ of $B_{p,q}$.

Rationally blowing up along $\iota$ yields a closed symplectic manifold $\widetilde{X}$ that contains the Wahl chain $\cc_{p,q}$ and the compactifying divisor $\cd_{p,q}$.
These two chains form a basis of $H_2(\widetilde{X};\Q)$.
Given any curve $\widetilde{C}$ in $\widetilde{X}$ we can express this curve in rational homology as:
$$
    \widetilde{C}=\sum_{i=0}^n a_iD_i + \sum_{j=1}^m c_j C_j \in H_2(\widetilde{X};\Q).
$$
Define $(M_\cd)_{ij}=D_i\cdot D_j$, the intersection matrix of the compactifying divisor, and write $\xi_i\vcentcolon=\widetilde{C}\cdot D_i$.
This results in $\bm{\xi}=M_\cd \bm{a}$.
Recall from \cref{def:compactification} that the subchain $\cd'\vcentcolon=(D_1,\ldots,D_n)$ of the compactifying divisor corresponds to taking a minimal resolution of the cyclic quotient singularity $\frac{1}{pq-1}(1,[q^2]^{-1})$, where $0<[q^2]^{-1}<pq-1$ is the unique integer that satisfies $q^2[q^2]^{-1}\equiv1 \mod{pq-1}$.
In particular, if we assume that $(pq-1)/[q^2]^{-1}=[d_1,\ldots,d_n]$ and consider the accompanying numbers $g_i,h_j$ as defined in \cref{lma:M_inverse}, where they are denoted by $e_i, f_j$, we obtain
\begin{equation}\label{eq:Inverse_MD'}
(M_{\cd'})^{-1}_{ij}=
    \begin{cases}-(g_i h_j)/(pq-1)\mbox{ if }i\leq j\\
      -(g_j h_i)/(pq-1)\mbox{ if }i>j 
    \end{cases}.
\end{equation}
Defining $\bm{e}_1\vcentcolon=(1,0,\ldots,0)$, the first canonical basis vector of length $n$, and $\bm{v}_{\cd'}\vcentcolon=(pq-1)(M_{\cd'})^{-1}\bm{e}_1$, the first column of $(pq-1)(M_{\cd'})^{-1}$, we can give a compact formula for $(M_{\cd})^{-1}$ using Schur complements:
\begin{equation}\label{eq:inverse_MD}
    M_{\cd}=
    \begin{pmatrix} 
        d_0 & \bm{e}_1^T \\ 
        \bm{e}_1 & M_{\cd'}
    \end{pmatrix}
    \quad\text{and}\quad 
    (M_{\cd})^{-1}=\frac{1}{p^2}
    \begin{pmatrix} 
        pq-1 & \bm{v}_{\cd'}^T \\ 
        \bm{v}_{\cd'} & p^2(M_{\cd'})^{-1}+\frac{1}{(pq-1)}\bm{v}_{\cd'}\cdot\bm{v}_{\cd'}^T
    \end{pmatrix},
\end{equation}
where $d_0$ was defined in \cref{le:minimal_res_orbifold_compactification}.

\subsection{Persistence of regulations}
\label{sec:trafo_regulation}

Recall the structure of the compactification $X_{p,q}$ of $B_{p,q}$ that we have constructed so far: it is birationally derived from a rational and ruled surface; there is a distinguished fibre class $F \in H_2(X;\mathbb{Z})$; the compactifying divisor $\mathcal{D}_{p,q}$ contains a positive section and a negative section; and if $\dim(H_2(X;\mathbb{Q}))>2$, which means that $q\neq 1$, there is a distinguished exceptional divisor $E$, which intersects the compactifying divisor in a single point.
This was proven in \cref{cor:compactification_birationally_derived}, \cref{lma:central_becomes_fibre}, \cref{lma:distinguished_E} and the setup is shown in \cref{fig:Compactification_Bpq_regulation}.

We now move on to showing that $F$ is part of a well-behaved foliation of $X$, which is constructed by holomorphic-curves techniques.
Recall the language of "regulations" and "rulings" from \cite[Section 4.1]{ABEHS25}:
an almost complex manifold $(Y,J)$ is said to admit a $J$-holomorphic regulation in the class $A\in H_2(Y;\mathbb{Z})$ if $A$ is a fibre class, i.e.\ has self-intersection equal to zero, and pairs with the first Chern class of $Y$ according to the adjunction formula, and the evaluation map $\text{ev}:\overline{\mathcal{M}}_{0,1}(A,J) \to Y$ has degree one.
Moreover, curves in $\overline{\mathcal{M}}_{0,0}(A,J)$ are referred to as rulings and they are called "broken" if they belong to $\partial \overline{\mathcal{M}}_{0,1}(A,J)$ and "smooth" otherwise. 
Define the space of almost complex structures
$$
    \mathcal{J}_\tau(\mathcal{D}):=\{J\in \mathcal{J}_\tau(X,\omega)\mid \text{all irreducible components of $\mathcal{D}$ are $J$-holomorphic}\},
$$
where $\mathcal{J}_\tau(X,\omega)$ denotes the space of tame almost complex structures on $(X,\omega)$. 
This space is non-empty since $\cd$ is a normal crossing divisor, and it is weakly contractible \cite[Appendix]{Ev14}.
The following two lemmas were first proven by Buck \cite[Section 2.4]{Bu25} and our proofs are inspired by Buck's proofs. 
In our case we can give simplified proofs of the lemmas, since our setup is significantly simpler.

\begin{lemma}
\label{lma:regulation_lma1}
    If $\dim (H_2(X;\mathbb{Q}))>2$, then the distinguished exceptional class $E$ is $J$-holomorphically represented by a unique smooth rational curve, which we denote by $E_J$, for any $J \in \mathcal{J}_\tau(\mathcal{D})$.
\end{lemma}

\begin{proof}
    The proof proceeds in two steps: we first show that any curve in the moduli space $\mathcal{M}_{0,0}(E,J)$ is smooth, i.e.\ $\overline{\mathcal{M}}_{0,0}(E,J)=\mathcal{M}_{0,0}(E,J)$, and then conclude that $\mathcal{M}_{0,0}(E,J)$ is non-empty for all $J \in \mathcal{J}_\tau(\mathcal{D})$.
    
    For the first step suppose that $J \in \mathcal{J}_\tau(\mathcal{D})$ and that $S \in \overline{\mathcal{M}}_{0,0}(E,J)$ is a stable curve with irreducible components $S_1,\dots,S_N$.
    Since $X\setminus \mathcal{D}$ is a Liouville domain, no component of $S$ can be completely contained in $X\setminus \mathcal{D}$ and, furthermore, since $E$ intersects $\cd$ exactly once by positivity of intersections, we must have that $[S_j]=E$ for one $j$.
    This immediately implies $S=S_j$.
    
    Now denote by $\mathcal{J}_\tau(\cd,E) \subseteq \mathcal{J}_\tau(\cd)$ the subset of all almost complex structures such that $E$ is $J$-holomorphically represented.
    We know by \cref{lma:distinguished_E} that this subset is non-empty, and applying automatic transversality, see \cite{HoLiSi97} or \cite[Chapter 2]{Wen18:Jhollow}, we obtain that each $\mathcal{M}_{0,0}(E,J)$ is a compact smooth manifold of dimension 0 for each $J\in \mathcal{J}_\tau(\cd,E)$ and that $\mathcal{J}_\tau(\cd,E)$ is open. The automatic transversality criterion and dimension statement follow from a direct computation: 
    $$
    \text{ind}(u)=-2+2c_1(E)=0>-2,
    $$
    for $u \in \mathcal{M}_{0,0}(E,J)$.
    Moreover, the first argument shows that $\mathcal{J}_\tau(\cd,E)$ is also closed. 
    Indeed, assume that $(J_\nu)_\nu \subseteq \mathcal{J}_\tau(\cd,E)$ is a sequence converging to $J \in \mathcal{J}_\tau(\cd)$.
    Then for each $\nu$ there is a unique embedded curve $u_\nu \in \mathcal{M}_{0,0}(E,J_\nu)$.
    Gromov compactness implies that there is a subsequence of $(u_\nu)_\nu$ converging to a stable curve $u_\infty \in \overline{\mathcal{M}}_{0,0}(E,J)$. Consequently, $J \in \mathcal{J}_\tau(\cd,E)$ and connectedness of $\mathcal{J}_\tau(\cd)$ then implies the lemma, i.e.\ $\mathcal{J}_\tau(\cd)=\mathcal{J}_\tau(\cd, E)$.
\end{proof}

\begin{lemma}\label{lma:ruling_compactification}
    For any $J \in \mathcal{J}_\tau(\mathcal{D})$, the compactification $X$ admits a non-degenerate $J$-holomorphic regulation in the class $F$ with at most one broken ruling.
    There exists a broken ruling if and only if $\dim(H_2(X;\mathbb{Q}))>2$.
    Moreover, if it exists it is given by $\mathcal{T}_J = (\mathcal{D}\setminus \{D_0,D_n\})\cup E_J$, i.e.\ the compactifying divisor without the positive and negative section and the exceptional divisor representing the distinguished class $E$.
\end{lemma}

\begin{proof}
    The proof is very similar to that of the previous lemma and is easy if there is no distinguished exceptional class; therefore, we omit this case.
    Assume that a distinguished exceptional class exists.
    Fix a point $x \in D_0\setminus D_1$, i.e.\ a point on the positive section away from the intersection of $D_0$ with the chain of negative spheres $\cd'=\{D_1,\ldots,D_n\}$.
    Define
    $$
        \cj_\tau(\cd,F,x):= \{J \in \cj_\tau(\cd) \mid \exists \text{ an embedded $J$-holomorphic sphere in the class $F$ through $x$}\}.
    $$
    By \cref{lma:central_becomes_fibre}, see also \cref{fig:Compactification_Bpq_regulation}, this space is non-empty and it is open by automatic transversality, see again \cite{HoLiSi97} or \cite[Chapter 2]{Wen18:Jhollow}, because for $u \in \mathcal{M}_{0,0}(F,J)$ we have
    $$
        \text{ind}(u)=-2 + 2c_1(F)=2>0.
    $$
    We now want to show that $\cj_\tau(\cd,F,x)$ is closed.
    To do so, assume that $(J_\nu)_\nu \subseteq \cj_\tau(\cd,F,x)$ is a sequence of almost complex structures converging to $J\in \mathcal{J}_\tau(\cd)$.
    Then consider for all $\nu$ the unique (by positivity of intersection) element $u_\nu \in \mathcal{M}_{0,0}(F,J_\nu)$ passing through $x$.
    By Gromov compactness, there is a subsequence of $(u_\nu)_\nu$ that Gromov-converges to a stable $J$-holomorphic curve $u_\infty \in \overline{\mathcal{M}}_{0,0}(F,J)$.
    Now write $S$ for the union of all irreducible components of $u_\infty$ that are not contained in $\cd' \cup E_J$.
    Then we can write
    \begin{equation}\label{eq:stable_curve_sum_ii}
        [F]=[S]+a[E_J]+\sum_{i=1}^n b_i [D_i],
    \end{equation}
    where $a$ and the $b_i$ are the covering multiplicities.
    Intersecting \eqref{eq:stable_curve_sum_ii} with $D_0$ yields $1=D_0\cdot S + b_1$, but we know that $x$ is on some component of $S$, which implies $D_0\cdot S \geq 1$ and so we obtain $b_1=0$ and $D_0\cdot S = 1$.
    Now intersect \eqref{eq:stable_curve_sum_ii} with $\cd'=\{D_1,\ldots,D_n\}$.
    Writing this as a system of equations we have
    \begin{equation}\label{eq:intersection_eq_stable_ii_1}
        \bm{e}_n=\bm{w}+a\bm{e}_r + M_{\cd'}\bm{b},
    \end{equation}
    where $\bm{e}_j$ is the $j$th canonical basis vector, $\bm{w}=S \cdot \cd'$, $r$ is the unique index associated to the intersection between the distinguished exceptional class $E$ and $\cd'$; and $M_{\cd'}$ is the intersection matrix of $\cd'$.
    Then \eqref{eq:intersection_eq_stable_ii_1} is equivalent to
    \begin{equation}\label{eq:intersection_eq_stable_ii_2}
        \bm{b}=(M_{\cd'})^{-1}(\bm{e}_n-\bm{w}-a\bm{e}_r).
    \end{equation}
     Using the formula for the inverse of $M_{\cd'}$ given in \eqref{eq:Inverse_MD'} and that $b_1=0$, the first row of this equation is given by 
    \begin{equation}\label{eq:intersection_eq_stable_ii_3}
        \sum_{i=1}^n h_i w_i + a h_r = 1,
    \end{equation}
    where the $h_i$ are the right accompanying numbers of the $\frac{1}{pq-1}(1,[q^2]^{-1})$ singularity, meaning in particular that $h_n=1$ and $h_i > 1$ for $i < n$.
    Because $r< n-1$ by \cref{lma:distinguished_E}, which proves that the distinguished exceptional class does not intersect $\cd'$ at one of its ends, we see that \eqref{eq:intersection_eq_stable_ii_3} forces $\bm{w}=\bm{e}_n$ as well as $a=0$.
    This implies that \eqref{eq:intersection_eq_stable_ii_1} reads
    $$
        M_{\cd'}\bm{b}=\bm{0},
    $$
    which means $\bm{b}=\bm{0}$, i.e.\ $[F]=[S]$.
    
    What is left to prove is that $S$ contains exactly one component.
    As before, every component in $S$ has to intersect $\cd$, as otherwise it would be contained in $B_{p,q}$, which is Liouville.
    Therefore, by positivity of intersection, there are at most two components of $S$ because $F\cdot D_0=1$ and $F\cdot D_n=1$.
    So assume that there are two components and pick the component $A$ that is intersecting $D_0$.
    Then we can write $A=\sum_{i=1}^n a_i [D_i]$ for some rational coefficients $a_i$, and because the intersection profile of $A$ with $\cd$ is given by $(1,0,\ldots,0)$ we obtain $\bm{e}_1=M_\cd \bm{a}$.\footnote{Note that if $S$ had one component the intersection profile would be $(1,0,\ldots,0,1)$.}
    The self-intersection of the rational curve $A$ is therefore given by
    $$
        A\cdot A=\bm{a}^T M_\cd \bm{a} = a_1.
    $$
    We can compute $a_1$ via $\bm{a}=M_\cd^{-1} \bm{e}_1$, which yields $a_1=\frac{pq-1}{p^2}$ by \eqref{eq:inverse_MD},
    an obvious contradiction. 
    Hence the limit curve $u_\infty$ is an embedded rational curve passing through $x$ and $\cj_\tau(\cd,F,x)$ is closed.
    The connectedness of $\cj_\tau(\cd)$ therefore implies $\cj_\tau(\cd,F,x)=\cj_\tau(\cd)$ as desired.
    
    By \cite[Proposition 4.1.4]{ABEHS25} this means that for all $J \in \cj_\tau(\cd)$ the compactification $X$ admits a non-degenerate $J$-holomorphic regulation in the class $F$.
    Also, since the point $x\in D_0\setminus D_1$ was arbitrary, this means that through every point $x \in D_0\setminus D_1$ and for every $J \in \cj_\tau(\cd)$ there is a smooth $J$-holomorphic ruling passing through $x$.
    
    Moreover, by \cite[Proposition 4.1.4]{ABEHS25} we know that for $J\in \cj_\tau(\cd)$ there is a ruling passing through $x \in D_0 \cap D_1$. 
    But then inductively we can show that this ruling has to contain all of $(\cd'\setminus D_n) \cup E_J$: it has to contain $D_1$, because $D_1\cdot F =0$ and then this reasoning propagates.  
    This means that this ruling $\ct_J$ is broken and is given by $(\cd'\setminus D_n) \cup E_J$.  
\end{proof}

With these two lemmas in place we now proceed to study the behaviour of the regulation after rationally blowing up along an embedding $\iota:B_{p,q}(\varepsilon)\hookrightarrow X$.
The precise description of how the blown-up fibre class interacts with the compactifying divisor and Wahl chain is stated in \cref{thm:regulation_of_tilX}.
As in \cref{subsec:res_blow_neck}, denote $U=\iota(B_{p,q}(\varepsilon)) \subseteq X$ and $V=X\setminus U$, and consider a neck-stretching sequence of almost complex structures $\{J_t\}_t$ such that $J_0 \in \cj_N(\mathcal{D})$, where $\cj_N(\mathcal{D})\subseteq \cj_\tau(\mathcal{D})$ denotes the subspace of almost complex structures that are standard on a neck of the embedded ball $B_{p,q}(\varepsilon)$, as defined in \cref{def:acs_neck}.

\begin{lemma}[\normalfont{\cite[Lemma 5.4.1]{ABEHS25}}]
    Fix a point $x$ on the positive section $D_0\subseteq X$ and for every $t$ consider the unique $J_t$-holomorphic curve $F_t$ in the class $F\in H_2(X;\Z)$ passing through $x$.
    For a subsequence $t_j\to \infty$ the sequence $F_{t_j}$ converges to a limit building with components in both $\overline{U}$ and~$\overline{V}$.
\end{lemma}

\begin{proof}
    The existence of the subsequence that yields a holomorphic limit building follows from the SFT compactness theorem \cite{BEHWZ03,CieMoh05}.
    That the completions contain components follows from the fact that each $F_t$ has non-trivial intersection with the divisor $D_0$ and the Lagrangian pinwheel $\iota(L_{p,q})$, as discussed in \cref{rmk:properties_fibre}.
\end{proof}

Now consider the component $C$ of the holomorphic limit building that is contained in $\overline{V}$ and intersects the divisor $D_0$.
We will continue investigating this component in the following lemmas and it should always be understood that $C$ is exactly this component.
Rationally blowing up along the embedding $\iota$ yields a symplectic manifold $\widetilde{X}$.
We denote the curve that corresponds (cf. \cref{lma:bijections}) to $C$ under this procedure by $\widetilde{C}$.
The Wahl chain~$\cc_{p,q}$ and the compactifying divisor~$\cd_{p,q}$ form a basis of $H_2{(\widetilde{X};\Q)}$.
We express $\widetilde{C}$ as well as the canonical class $K=-\text{PD}(c_1(\widetilde{X}))$ of $\widetilde{X}$ in this basis:
\begin{equation}\label{eq:KandC}
    K=-F-\sum_{i=0}^n D_i+\sum_{j=1}^m k_jC_j,\qquad \widetilde{C}=\sum_{i=0}^n a_i D_i +\sum_{j=1}^m c_j C_j,
\end{equation}
where $\bm{k}=(k_1,\ldots,k_m$) are the discrepancies.\footnote{The formula for the canonical class $K$ of $\widetilde{X}$ follows from the fact that the canonical class $K_X$ of $X$ is given by $K_X=-F-\sum_{i=0}^n D_i$, since $X$ is toric and the toric boundary is Poincaré dual to the first Chern class. Moreover, the splitting of the first Chern class is discussed in \cref{rmk:splitting_Chern}.}

Recall that $(M_\cc)_{ij}=C_i\cdot C_j$ denotes the intersection matrix of the Wahl chain and that defining $\chi_j\vcentcolon=\widetilde{C}\cdot C_j$ yields the equation $\bm{\chi}=M_\cc \bm{c}$.
Analogously we write $(M_\cd)_{ij}=D_i\cdot D_j$ for the intersection matrix of the compactifying divisor, and writing $\xi_i\vcentcolon=\widetilde{C}\cdot D_i$ means $\bm{\xi}=M_\cd \bm{a}$.
The two pairings that will play a role in the following are the self-intersection of $\widetilde{C}$
\begin{equation}\label{eq:C_selfpairing}
    \widetilde{C}\cdot\widetilde{C}=\bm{\xi}^T M_\cd^{-1} \bm{\xi} + \bm{\chi}^T M_\cc^{-1}\bm{\chi}
\end{equation}
and the pairing of $\widetilde{C}$ with the canonical class $K$
\begin{equation}\label{eq:KC_pairing}
    K\cdot \widetilde{C}=-a_0-a_n-\bm{1}^T \bm{\xi}+\bm{k}^T\bm{\chi}.
\end{equation}
These two equations immediately follow from the fact that $\cc_{p,q}$ and $\cd_{p,q}$ are disjoint and that $F$ intersects $\cd_{p,q}$ once positively at each end.

\begin{lemma}[\normalfont{\cite[Proposition 5.4.3]{ABEHS25}}]
\label{lma:intersection_profile_CD}
    The intersection profile of $\widetilde{C}$ with $\cd$ is given by $\bm{\xi}=(1,0,\ldots,0)$.
    Moreover, $\widetilde{C}$ is an embedded curve of self-intersection $\widetilde{C}^2\leq 0$.
\end{lemma}

\begin{proof}
    Assume for the sake of a contradiction that the intersection profile with the compactifying divisor $\cd$ is given by $\bm{\xi}=(1,0\ldots,0,1)$.\footnote{Recall that $\widetilde{C}$ is the component of the limit building that intersects $D_0$, therefore this intersection is forced and hence the two possible intersection profiles are $(1,0\ldots,0)$ and $(1,0\ldots,0,1)$.}
    Then \eqref{eq:C_selfpairing} reads $\widetilde{C}\cdot\widetilde{C}=a_0+a_n + \bm{\chi}^T M_\cc^{-1}\bm{\chi}$ and \eqref{eq:KC_pairing} reads $K\cdot\widetilde{C}=-a_0-a_n-2+\bm{k}^T\bm{\chi}$.
    Plugging these two equations into the adjunction formula we obtain
    $$
        \sum_{x\in\text{Sing}(\widetilde{C)}} \delta_x = \frac{1}{2}(\widetilde{C}\cdot\widetilde{C}+K\cdot\widetilde{C})+1=\frac{1}{2}(\bm{k}^T\bm{\chi}+\bm{\chi}^T M_\cc^{-1}\bm{\chi}).
    $$
    Since all the $\delta_x$ contribute positively and the terms in the bracket on the right-hand side are non-positive\footnote{Recall that the discrepancies can be computed via $k_j=-1+\frac{e_j+f_j}{p^2}$, see \cref{lma:discrepancies}, which means that $k_j\in (-1,0]$; that all the entries of the matrix $M_\cc^{-1}$ are negative, which was shown in \cref{lma:M_inverse}; and that, by positivity of intersection, the components of $\bm{\chi}$ are non-negative integers.}, this means that $\bm{\chi}=\bm{0}$, which is a contradiction since this would imply that $\widetilde{C}$ is an embedded holomorphic sphere that corresponds to a holomorphic representative of $F$ in $X$ that does not intersect the class of the pinwheel defined via the embedding $\iota$.
    
    Since we picked $\widetilde{C}$ to be the component in the limit building that intersects $D_0$ we therefore have $\bm{\xi}=(1,0,\ldots,0)$.
    This implies that \eqref{eq:C_selfpairing} becomes $\widetilde{C}\cdot\widetilde{C}=a_0+ \bm{\chi}^T M_\cc^{-1}\bm{\chi}$ and \eqref{eq:KC_pairing} becomes $K\cdot\widetilde{C}=-a_0-a_n-1+\bm{k}^T\bm{\chi}$.
    Consider the adjunction formula for $\widetilde{C}$:
    $$
        \sum_{x\in\text{Sing}(\widetilde{C)}} \delta_x = \frac{1}{2}(\widetilde{C}\cdot\widetilde{C}+K\cdot\widetilde{C})+1=\frac{1}{2}(-a_n-1+\bm{k}^T\bm{\chi}+\bm{\chi}^T M_\cc^{-1}\bm{\chi})+1.
    $$
    From \eqref{eq:inverse_MD} it immediately follows that $a_n=1/p^2$ and therefore $-a_n-1=-(p^2+1)/p^2$.
    As before, the term in the bracket is negative, which implies that $\text{Sing}(\widetilde{C})$ is empty, because the $\delta_x$ contribute positively.
    In particular, $\widetilde{C}$ is a smoothly embedded sphere. 
    The claim about the self-intersection number of $\widetilde{C}$ is then a consequence of the fact that $a_0=\frac{pq-1}{p^2}<1$, which follows from \eqref{eq:inverse_MD} and shows that 
    \begin{equation*}
        \widetilde{C}\cdot\widetilde{C}=a_0+ \bm{\chi}^T M_\cc^{-1}\bm{\chi} \leq a_0 < 1.\qedhere
    \end{equation*}
\end{proof}

As a result of this lemma, choosing the almost complex structure generically on the complement of the embedding $\iota$, we can ensure that $\widetilde{C}$ is a $(0)$-sphere or a $(-1)$-sphere.
However, since there is a $\C$-family of rational curves representing the fibre class $F$ that we can stretch, we must find a $(0)$-sphere.\footnote{That there is a $\C$-family of rational curves was proven in \cref{lma:ruling_compactification}.}
In the following we will therefore assume that $\widetilde{C}$ is a curve of square zero.

\begin{lemma}[\normalfont{\cite[Lemma 5.4.6]{ABEHS25}}]
\label{lma:intersection_profil_CC}
    The curve $\widetilde{C}$ intersects the Wahl chain precisely once.
    Moreover, the intersection is with the first sphere in the Wahl chain.
\end{lemma}

\begin{proof}
    By \cref{lma:intersection_profile_CD} we know that the adjunction formula for $\widetilde{C}$ reads
    \begin{equation}\label{eq:adjunction_TilC}
        1=\frac{p^2+1}{2p^2}-\frac{1}{2}(\bm{k}^T\bm{\chi}+\bm{\chi}^T M_\cc^{-1}\bm{\chi}),
    \end{equation}
    where the first term is calculated using $a_n=1/p^2$, which follows from the equation $\bm{a}=(M_\cd)^{-1}\bm{\xi}$ and \eqref{eq:inverse_MD}.
    Observe that all the entries in $\bm{\chi}$ are non-negative integers by positivity of intersection and that all the entries of $M_\cc^{-1}$ are negative.
    Therefore, the second summand on the right-hand side contributes at least
    \begin{equation}\label{eq:contributions_adjunction}
        \sum_{j=1}^m\left(\frac{1}{2p^2}\chi^2_je_jf_j + \frac{1}{2}\chi_j\left(1 - \frac{e_j+f_j}{p^2}\right)\right),
    \end{equation}
    which is the diagonal term, to a sum that is equal to $1$.
    Assume that there is a $\chi_j\geq 2$.
    Then the sum \eqref{eq:contributions_adjunction} is at least
    $$
        \frac{1}{2p^2}\chi^2_je_jf_j +
        \frac{1}{2}\chi_j\left(1 - \frac{e_j+f_j}{p^2}\right)\geq
        1+\frac{1}{p^2}\left(2e_jf_j-e_j-f_j\right) \geq 1,
    $$
    since $e_j,f_j\geq 1$, in contradiction to \eqref{eq:adjunction_TilC}.
    
    If $\chi_j=1$ for some index $j$ we have
    $$
        \frac{1}{2p^2}\chi^2_je_jf_j + \frac{1}{2}\chi_j\left(1 - \frac{e_j+f_j}{p^2}\right)=
        \frac{1}{2}\left(1+\frac{1}{p^2}\left(e_jf_j-e_j-f_j\right)\right)\geq
        \frac{1}{2}\left(1-\frac{1}{p^2}\right).
    $$
    This means that if there are at least two $\chi_j$ that are equal to $1$ the right-hand side of \eqref{eq:adjunction_TilC} is at least
    $$
        \frac{p^2+1}{2p^2} + 1-\frac{1}{p^2}>1,
    $$
    which is a contradiction because $p\geq 2$.
    In conclusion we see that $\chi_j=1$ for exactly one $j$ and we are left to determine this index.
    By \eqref{eq:adjunction_TilC} we have 
    $$
        0=\frac{p^2+1}{2p^2}+\frac{1}{2p^2}(p^2+e_jf_j-(e_j+f_j))-1=\frac{1}{2p^2}(e_jf_j-e_j-f_j+1)=\frac{1}{2p^2}(e_j-1)(f_j-1)
    $$
    which is the case if and only if $e_j=1$ or $f_j=1$.
    By definition of the accompanying numbers this is only the case if either $j=1$ or $j=m$, i.e.\ $\widetilde{C}$ must intersect the Wahl chain at one of its ends.
    Since $\widetilde{C}$ is a $(0)$-curve we also know that
    $$
        0=a_0+\bm{\chi}^T M_\cc^{-1}\bm{\chi}=\frac{pq-1}{p^2}-\frac{1}{p^2}e_jf_j.
    $$
    This implies that $j=1$, because $(e_1,f_1)=(1,pq-1)$ and $(e_m,f_m)=([pq-1]^{-1},1)$, where $0<[pq-1]^{-1}<p^2$ is the integer such that $(pq-1)[pq-1]^{-1}\equiv 1 \mod{p^2}$ as discussed in \eqref{eq:initial_acomp} and \eqref{eq:extremal_acomp}, i.e.\ $[pq-1]^{-1}=p^2-(pq+1)$.\footnote{Of course there is one case in which $[pq-1]^{-1}=pq-1$, namely $(p,q)=(2,1)$. However, this is the case in which the Wahl chain is just a single $(-4)$-sphere.}
\end{proof}

To summarise, we have found a square zero curve $\widetilde{C}$ in $\widetilde{X}$ that intersects both $D_0$ and $C_1$ once positively and does not intersect any of the other spheres in $\cd_{p,q}$ and $\cc_{p,q}$, which implies that $\widetilde{C}$ is a smooth ruling of a regulation of $\widetilde{X}$ in the class $\widetilde{C}$.
What these considerations amount to is that the regulation of $X$ in the class $F$ "persists" under the rational blow-up along $\iota$.
In the case that $\iota=\iota_{\text{vis}}$ this is obvious from the almost toric base diagram, as illustrated in \cref{fig:trafo_regulation}.

\begin{figure}[ht]
    \centering
    \begin{tikzpicture}
    \begin{scope}[shift={(-5.5,-2)}]
        \fill[opacity=0.2] 
            (0,4) -- (0,0) -- (2,0.5) -- (3,1) -- (3.8,3) -- (4,4);
        \draw 
            (0,4) node{\tiny \(\bullet\)} -- (0,0) -- (2,0.5) node{\tiny $\bullet$} -- (3,1) node{\tiny \(\bullet\)};
        \draw
            (3.8,3) node{\tiny \(\bullet\)} -- (4,4) node{\tiny \(\bullet\)};
        \draw[dashed] 
            (1,0.5) node[cross] {} -- (0,0);
        \draw[thick, blue]
            (0.6,4) to[out=-70,in=70] (0.6,3) to[out=-110,in=135] (0.6,2) to[out=-45,in=90](0.6,0.3) -- (2.6,0.8);
        \node[text=blue] at (0.8,2.5) {\small $F$};
        \draw[thick]
            (0,4) -- node[anchor=south, yshift=-2pt] {\small $D_0$} (4,4) -- (3.8,3) -- (3.6,2.5);
        \draw[thick, densely dotted]
            (3.6,2.5) -- node[anchor=west] {\small $\cd_{p,q}$} (3.2,1.5);
        \draw[thick]
            (2,0.5) -- node[anchor=north west, xshift=-3pt, yshift=3pt] {\small $D_n$} (3,1) -- (3.2,1.5);
        \end{scope}

        \begin{scope}[shift={(1.5,-2)}]
        \fill[opacity=0.2] 
            (0,4) -- (0,0.25) -- (0.8,0.25) -- (1.2,0.3) -- (2,0.5) -- (3,1) -- (3.8,3) -- (4,4);
        \draw 
            (0,4) node{\tiny $\bullet$} -- (0,0.25) node{\tiny $\bullet$} -- (0.8,0.25) node{\tiny $\bullet$} -- (1.2,0.3) node{\tiny $\bullet$} -- (2,0.5) node{\tiny $\bullet$} -- (3,1) node{\tiny $\bullet$};
        \draw
            (3.8,3) node{\tiny $\bullet$} -- (4,4) node{\tiny $\bullet$};
        \draw[thick] 
            (0,0.25) -- node[anchor=north, yshift=2pt] {\small $C_1$} (0.8,0.25) -- node[anchor=north west, xshift=-2pt, yshift=-2pt]{\small $\cc_{p,q}$} (1.2,0.3);
        \draw[thick, blue]
            (0.6,4) to[out=-70,in=70] (0.6,3) to[out=-110,in=135] (0.6,2) to[out=-45,in=90](0.6,0.3) -- (0.6,0.25);
        \node[text=blue] at (0.8,2.5) {\small $\widetilde{C}$};
        \draw[thick]
            (0,4) -- node[anchor=south, yshift=-2pt] {\small $D_0$} (4,4) -- (3.8,3) -- (3.6,2.5);
        \draw[thick, densely dotted]
            (3.6,2.5) -- node[anchor=west] {\small $\cd_{p,q}$} (3.2,1.5);
        \draw[thick]
            (2,0.5) -- node[anchor=north west, xshift=-3pt, yshift=3pt] {\small $D_n$} (3,1) -- (3.2,1.5);
        \end{scope}
        \draw[->] (-1,0) to (1,0);
        \node at (0,0.3) {\small rational};
        \node at (0,-0.3) {\small blow-up};
        \end{tikzpicture}
        \caption{The transformation of the regulation in the visible case.}
    \label{fig:trafo_regulation}
\end{figure}
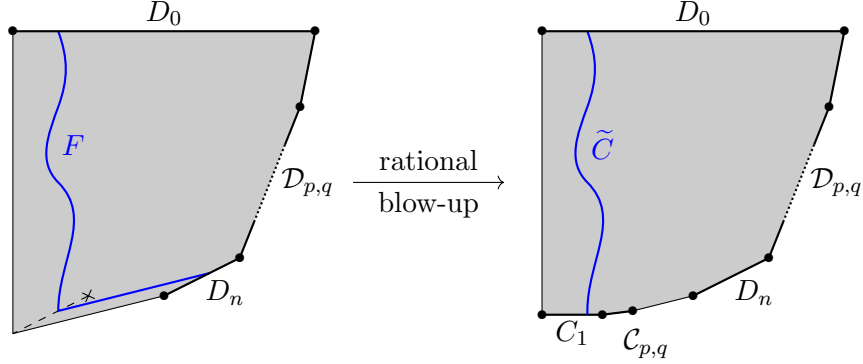

\begin{remark}\label{rmk:only_2_broken_rulings}
    Choosing $\widetilde{J}$ generically we can ensure that the irreducible components of $\cc_{p,q}$ and $\cd_{p,q}$ are the only $\widetilde{J}$-holomorphic spheres with self-intersection strictly less than $-1$, according to \cref{lma:bijections}. 
    In particular, this means that all the other $\widetilde{J}$-holomorphic embedded rational curves in $\widetilde{X}$ are either exceptional spheres or have non-negative square and that all the curves in $\cc_{p,q}\setminus C_1$ and $\cd_{p,q}\setminus D_0$ have to appear as irreducible components of a broken ruling.
    Since every broken ruling contains an $(-1)$-sphere, successively contracting these in the broken ruling will yield a "minimal model" $Y$ of $\widetilde{X}$, which contains no broken rulings. 
    The broken ruling before the last step will consist of two transversally intersecting $(-1)$-spheres.\footnote{See \cite[Corollary 4.2.8]{ABEHS25} and the surrounding discussion.}
\end{remark}

\begin{lemma}\label{lma:number_of_contractions}
    The number of contractions that lead from $\widetilde{X}$ to the minimal model $Y$ is equal to $m+n-1$.
    In particular, the second Betti number of $\widetilde{X}$ is equal to $m+n+1$.
\end{lemma}

\begin{proof}
    We know that a basis of $H_2(\widetilde{X},\Q)$ is given by $\cc_{p,q}$ and $\cd_{p,q}$, and therefore we have 
    $$
        \dim(H_2(\widetilde{X},\Q))=n+(m+1).
    $$
    Denoting the number of contractions that lead from $\widetilde{X}$ to $Y$ by $k$, the equation $k+2=n+m+1$ follows, since the second Betti number of $Y$ is equal to $2$.
\end{proof}

The discussion above then implies the following theorem.

\begin{theorem}\label{thm:regulation_of_tilX}
    Assume that the almost complex structure $\widetilde{J}$ is chosen generically on the subset $V \subseteq \widetilde{X}$.
    Then there is an embedded $\widetilde{J}$-holomorphic sphere $\widetilde{C}\subset\widetilde{X}$ of square zero such that
    $$
        \widetilde{C}\cdot C_j=
        \begin{cases}
            0 \quad\mbox{if }j\neq1,\\1\quad\mbox{if }j=1
        \end{cases}
        \qquad\text{and }\qquad
        \widetilde{C}\cdot D_i=
        \begin{cases}
            0\quad\mbox{if }i\neq 0,\\ 1 \quad\mbox{if }i=0
        \end{cases}.
    $$  
    Moreover, the regulation defined by the $\widetilde{J}$-holomorphic sphere $\widetilde{C}$ has exactly one broken ruling that is given by
    $$\ct=(C_2 \cup\ldots\cup C_m) \cup (D_1 \cup \ldots \cup D_n) \cup E_{\widetilde{J}}=(\cc_{p,q}\setminus C_1) \cup (\cd_{p,q} \setminus D_0) \cup E_{\widetilde{J}},$$
    where $E_{\widetilde{J}}$ is an exceptional sphere that intersects the Wahl chain and the compactifying divisor at their terminal components.
\end{theorem}

The structure of the regulation in the class $\widetilde{C}$ in \cref{thm:regulation_of_tilX} is shown in \cref{fig:regulation_of_tilX}.
Before proving this theorem, let us formulate a crucial corollary of \cref{thm:regulation_of_tilX}.

\begin{corollary}\label{cor:complement_of_T_is_minimal}
    In the setup of \cref{thm:regulation_of_tilX} the complement of $\ct$, i.e.\ $\widetilde{X}\setminus \ct$, is minimal.    
\end{corollary}

\begin{proof}
    By \cref{thm:regulation_of_tilX} the symplectic manifold $\widetilde{X}$ is derived from a Hirzebruch surface $Y$ by consecutive blow-ups of a single ruling.
    This means that $\widetilde{X}\setminus \ct$ is contained in $Y \setminus F$, where $F$ is a fibre of the Hirzebruch surface.\footnote{Note that denoting this fibre by $F$ makes sense, since a smooth ruling of $\widetilde{X}$ is a fibre of $Y$ under the sequence of blow-downs.}
    Therefore, $\widetilde{X}\setminus \ct$ is minimal.
\end{proof}

\begin{proof}[Proof of \cref{thm:regulation_of_tilX}]
    The first part of the theorem is a summary of the discussion before and the second part is obvious if $(p,q)=(2,1)$, since in that case $\mathcal{C}_{2,1}$ consists of a single sphere and $\mathcal{D}_{2,1}$ consists of a positive sphere and an exceptional divisor.
    Therefore we move on to prove the second part and assume that $(p,q)\neq (2,1)$.
    
    By \cref{rmk:only_2_broken_rulings} we know that there has to be at least one broken ruling since the tails of the chains $\cc_{p,q}$ and $\cd_{p,q}$ have to be part of a broken ruling. 
    The goal is to show that there is precisely one broken ruling given by the tails of $\cc_{p,q}$ and $\cd_{p,q}$, denoted by $\cc':=\cc_{p,q}\setminus C_1$ and $\cd':=\cd_{p,q}\setminus D_0$ in the following, connected by a single exceptional sphere that intersects the tails at their ends.
    
    Assume that both tails are contained in two distinct broken rulings,\footnote{Recall that broken rulings are either disjoint, or they coincide.} which we will call $\mathcal{T}_\mathcal{C}$ and $
    \mathcal{T}_\mathcal{D}$.
    We know that the broken rulings have to contain at least one exceptional sphere and hence we have
    $$
        \dim(H_2(\widetilde{X};\Q))\geq ((n-1)+1)+(m+1).
    $$
    However, we know by \cref{lma:number_of_contractions} that this is an equality and therefore the regulation of $\widetilde{X}$ has exactly two broken rulings, namely $\ct_\mathcal{C}$ and $\ct_\mathcal{D}$, each containing a single $(-1)$-sphere, $E_\mathcal{C}$ and $E_\mathcal{D}$.
    We will now show that such a configuration cannot exist.
    Consider the expression
    $$
        \mathcal{T}_\mathcal{C}=\sum_{j=2}^m \alpha_j C_j + \epsilon E_{\mathcal{C}},
    $$
    where the $\alpha_j$ and $\epsilon$ are the covering multiplicities.
    Because $\mathcal{T}_\mathcal{C}$ is a broken ruling of the regulation in the class $[\widetilde{C}]$ we know that
    $$
        \ct_\cc \cdot C_1 = \widetilde{C}\cdot C_1 = 1 \qquad\text{and}\qquad \ct_\cc \cdot D_0 = \widetilde{C}\cdot D_0 = 1,
    $$
    which implies $\alpha_2=1$ and $\epsilon=1$.
    Now assume that $E_\cc$ intersects $\cc'$ in $C_r$.
    Then $\ct_\cc \cdot C_j = \widetilde{C} \cdot C_j= 0$ implies
    $$
        M_{\cc'} \bm{\alpha} + \bm{e}_r =0,\quad \text{i.e.\ } \bm{\alpha}=-(M_{\cc'})^{-1} \bm{e}_r,
    $$
    where $M_{\cc'}$ is the intersection matrix of $\cc'$ and $\bm{e}_r$ is the $r$th standard basis vector.
    In particular, we must have $1=\alpha_2=-(M_{\cc'})^{-1}_{1r}$.
    In the notation of \cref{lma:M_inverse} this implies $1=f_r/(pq-1)$, because the HJ continued fraction for $\cc'$ is $(pq-1)/(b_1(pq-1)-p^2)$.
    However, this is a contradiction, because $f_r<f_0=pq-1$.
    
    Therefore, the regulation of $\widetilde{X}$ has a single broken ruling $\ct$ which contains the tails $\cc'$ and $\cd'$ and a single exceptional sphere $E$.\footnote{That there cannot be more than one exceptional sphere again follows from the fact that we can compute the dimension of $H_2(\widetilde{X},\Q)$ by counting the number of contractions and alternatively by adding the lengths of the chains $\cc'$ and $\cd'$. This also shows that there is a single broken ruling.}
    What is left to show is the intersection pattern of $E$ and the tails.
    As before we write
    $$
        \mathcal{T}=\sum_{j=2}^m \alpha_j C_j + \sum_{i=1}^n \beta_i D_i+ \epsilon E,
    $$
    where again the $\alpha_j$, $\beta_i$ and $\epsilon$ are the multiplicities.
    Because $[\ct]=[\widetilde{C}]$ we have
    $$
        \ct \cdot C_1 = \widetilde{C}\cdot C_1 = 1 \qquad\text{and}\qquad \ct \cdot D_0 = \widetilde{C}\cdot D_0 = 1,
    $$
    which implies $\alpha_2=\beta_1=1$.
    Assume that $E$ meets $\cc'$ in $C_r$ and $\cd'$ in $D_s$.
    This means that we have to show that $r=m$ and $s=n$.
    The intersection identities $\mathcal{T}\cdot C_j=\widetilde{C} \cdot C_j$ and $\mathcal{T}\cdot D_i=\widetilde{C} \cdot D_i$ imply
    $$
        M_{\cc'} \bm{\alpha} + \epsilon \bm{e}_r =0,\quad\text{i.e.\ } \bm{\alpha}=-\epsilon (M_{\cc'})^{-1} \bm{e}_r,\qquad\text{ and }\qquad M_{\cd'} \bm{\beta} + \epsilon \bm{e}_s =0,\quad\text{i.e.\ } \bm{\beta}=-\epsilon(M_{\cd'})^{-1} \bm{e}_s.
    $$
    Considering the first component of each of these equations we get 
    $$
        \epsilon=-\frac{1}{(M_{\cc'})^{-1}_{1r}}=\frac{pq-1}{f_r}\qquad\text{and}\qquad \epsilon=-\frac{1}{(M_{\cd'})^{-1}_{1s}}=\frac{pq-1}{h_s},
    $$
    which implies that $f_r=h_s$.
    Since $(pq-1)/[q^2]^{-1}$ is the dual of $(pq-1)/(c_1(pq-1)-p^2)$, they only share one right accompanying number, namely $1$.
    This implies that $r=m$ and $s=n$, which concludes the proof.\footnote{Recall that $[q^2]^{-1}=p^2-d_0(pq-1)$ and that $d_0=b_1-1$. The fact that these dual HJ continued fractions only share a single right accompanying number is easy to see. Write $N=(pq-1)$ and $a=[q^2]^{-1}$. Then we have $f_r \equiv -e_r a \mod{N}$ and $h_s \equiv -g_r a \mod{N}$, because of the recursive formula \eqref{eq:rec_acomp}, which yields $(e_r+g_s)a \equiv 0 \mod{N}$. Because $\text{gcd}(a,N)=1$ this implies $e_r+g_s \equiv 0 \mod{N}$. However, the left accompanying numbers are increasing and they terminate at $e_{m-1}=N-[a]^{-1}$ and $g_n=[a]^{-1}$. This implies $0 < e_r+g_s \leq N$ and therefore the claim.}
\end{proof}

\begin{figure}[htb]
\begin{center}
    \begin{tikzpicture}
    \begin{scope}[shift={(-1,-1.375)}]
    \node at (-0.7,3) {(a)};
    \draw[thick] (0,0) -- (3,0) node (c1) [pos=0.5] {};
    \node at (c1) [above, yshift=-2pt] {\tiny $-6$};
    \node at (c1) [below] {\small $C_1$};
    
    \draw[thick] (2.75,-0.25) to[out=80,in=190] (3.75,0.75);
    \node (c2) at (3.25,0.35) {};
    \draw[thick] (3.5,0.5) to[out=80,in=190] (4.5,1.5);
    \node (c3) at (4,1.11) {};
    \draw[thick] (4.25,1.25) to[out=80,in=190] (5.25,2.25);
    \node (e) at (4.75,1.85) {};
    \draw[thick] (5,2) to[out=80,in=190] (6,3);
    \node (d1) at (5.5,2.6) {};
    
    \node at (c2) [above left] {\small $C_2$};
    \node at (c2) [xshift=1pt, yshift=-2pt] {\tiny $-2$};
    
    \node at (c3) [above left] {\small $C_3$};
    \node at (c3) [xshift=1pt, yshift=-2pt] {\tiny $-2$};
    
    \node at (e) [above left] {\small $E$};
    \node at (e) [xshift=1pt, yshift=-2pt] {\tiny $-1$};
    
    \node at (d1) [above left, xshift=3pt, yshift=3pt] {\small $D_1$};
    \node at (d1) [xshift=1pt, yshift=-2pt] {\tiny $-3$};

    \draw[thick] (0,2.7) -- (5.75,2.7) node (d0) [pos=0.5] {};
    \node at (d0) [above, yshift=-1pt] {\small $D_0$};
    \node at (d0) [below] {\tiny $+5$};

    \draw[thick] (0.25,-0.25) -- (0.25,3) node (f) [pos=0.5] {};
    \node at (f) [left] {\small $\widetilde{C}$};
    \end{scope}
    
    \begin{scope}[shift={(8,-1.5)}]
    \node at (-1.5,3.125) {(b)};
    \draw[thick] (0,-0.25) to[out=120,in=240] (0,0.75);
    \draw[thick] (0,0.5) to[out=120,in=240] (0,1.5);
    \draw[thick] (0,1.25) to[out=120,in=240] (0,2.25);
    \draw[thick] (0,2) to[out=120,in=240] (0,3);
    \node at (-0.5,0.25) {\small $C_2$};
    \node at (0.1,0.25) {\tiny $-2$};
    \node at (-0.5,1) {\small $C_3$};
    \node at (0.1,1) {\tiny $-2$};
    \node at (-0.5,1.75) {\small $E$};
    \node at (0.1,1.75) {\tiny $-1$};
    \node at (-0.5,2.5) {\small $D_1$};
    \node at (0.1,2.5) {\tiny $-3$};
    \draw[->] (0.4,1.75) to[out=10,in=170] (1.7,1.75);
    \draw[thick] (0,0) -- (-0.5,-0.5) node (c1) [pos=0.5] {};
    \draw[thick] (0,2.75) -- (-0.5,3.25) node (d1) [pos=0.5] {};
    \node at (c1) [left,xshift=-2pt, yshift=-1pt] {\small $C_1$};
    \node at (c1) [below, xshift=1pt, yshift=-1pt] {\tiny $-6$};
    \node at (d1) [left, xshift=-2pt, yshift=1pt] {\small $D_0$};
    \node at (d1) [above, xshift=1pt, yshift=1pt] {\tiny $+5$};

    \draw[thick] (2,-0.25) to[out=120,in=240] (2,0.75);
    \draw[thick] (2,0.5) to[out=120,in=240] (2,1.88);
    \draw[thick] (2,1.62) to[out=120,in=240] (2,3);
    \node at (2.1,0.25) {\tiny $-2$};
    \node at (2.1,1.19) {\tiny $-1$};
    \node at (2.1,2.31) {\tiny $-2$};
    \draw[->] (2.4,1.19) to[out=0,in=180] (3.7,1.51);
    \draw[thick] (2,0) -- (1.5,-0.5) node (c1) [pos=0.5] {};
    \draw[thick] (2,2.75) -- (1.5,3.25) node (d1) [pos=0.5] {};
    \node at (c1) [below, xshift=1pt, yshift=-1pt] {\tiny $-6$};
    \node at (d1) [above, xshift=1pt, yshift=1pt] {\tiny $+5$};

    \draw[thick] (4,-0.25) to[out=120,in=240] (4,1.62);
    \draw[thick] (4,1.38) to[out=120,in=240] (4,3);
    \node at (4.1,0.65) {\tiny $-1$};
    \node at (4.1,2.1) {\tiny $-1$};
    \draw[->] (4.4,0.65) to[out=-10,in=180] (5.4,-0.05);
    \draw[thick] (4,0) -- (3.5,-0.5) node (c1) [pos=0.5] {};
    \draw[thick] (4,2.75) -- (3.5,3.25) node (d1) [pos=0.5] {};
    \node at (c1) [below, xshift=1pt, yshift=-1pt] {\tiny $-6$};
    \node at (d1) [above, xshift=1pt, yshift=1pt] {\tiny $+5$};

    \draw[thick] (5.7,-0.25) to[out=110,in=250] (5.7,3);
    \node at (5.6,1.375) {\tiny $0$};
    \draw[thick] (5.7,0) -- (5.2,-0.5) node (c1) [pos=0.5] {};
    \draw[thick] (5.7,2.75) -- (5.2,3.25) node (d1) [pos=0.5] {};
    \node at (c1) [below, xshift=1pt, yshift=-1pt] {\tiny $-5$};
    \node at (d1) [above, xshift=1pt, yshift=1pt] {\tiny $+5$};
    
    \end{scope}
    \end{tikzpicture}
    \caption{(a) An example that illustrates the structure of the regulation in the class $\widetilde{C}$ proved in \cref{thm:regulation_of_tilX} in the case $(p,q)=(4,1)$. (b)
    The contraction process of the broken ruling for this example. In this example $d_0=5$ and the figure shows how $\widetilde{X}_{4,1}$ is derived from $\F_{5}$.}
    \label{fig:regulation_of_tilX}
\end{center}
\end{figure}
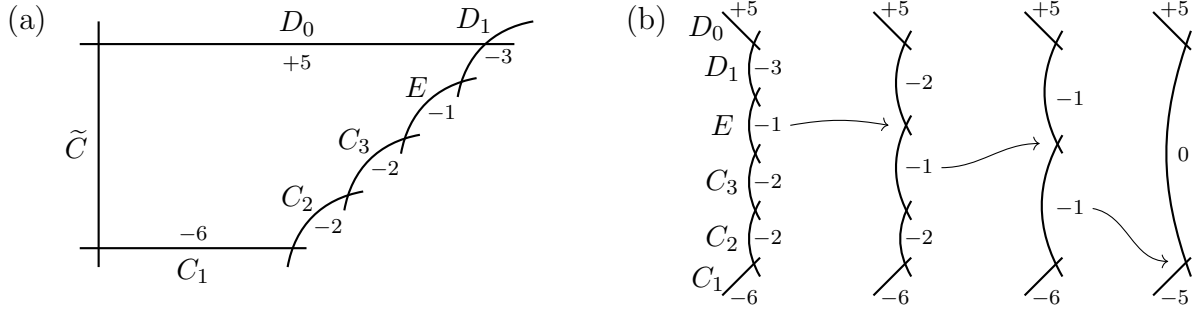

\begin{remark}
    \cref{thm:regulation_of_tilX} also implies that the minimal model of $Y$ can be arranged to be the Hirzebruch surface $\F_{d_0}$.
    This follows from the structure of $\mathcal{T}$ proven in \cref{thm:regulation_of_tilX}. 
    The last contraction is a contraction of a broken fibre that consists of two $(-1)$-spheres.
    Contracting the curve that intersects the negative section yields the Hirzebruch surface $\F_{d_0}$.
\end{remark}

\subsection{Constructing the symplectomorphism}\label{sec: constructing the symp}

In this subsection we want to show that Lagrangian pinwheels in their normal neighbourhoods are unique up to symplectomorphism.
More precisely, we want to prove the following theorem.
\begin{theorem}\label{th:symp}
    Suppose that $\iota:B_{p,q}(\varepsilon) \hookrightarrow B_{p,q}$ is a symplectic embedding.
    Then there exists a compactly supported symplectomorphism $\Phi \in \symp_c(B_{p,q})$ such that $\Phi \circ \iota =\iota_\vis$.
\end{theorem}

\begin{remark}
    Note that the uniqueness of Lagrangian $(p,q)$-pinwheels in $B_{p,q}$ up to symplectomorphism follows from this theorem, since we showed that every such pinwheel is "locally visible" in \cref{thm:loc_vis}.
\end{remark}

This theorem is very similar to \cite[Theorem 1.4.1]{ABEHS25} and its proof will indeed also follow along similar lines.
However, we will base the proof of \cref{th:symp} on a "different" method.\footnote{It will be clear from the method of proof and from the references cited therein that the root of this method is Gromov's \cite{Gr85} $S^2\times S^2$-foliation result, which is also the underlying result for the proof of \cite[Theorem 1.4.1]{ABEHS25}.} 
We shall need the following analogue, for star-shaped sets in $\TSR$, of the Gromov--McDuff Theorem \cite[Theorem 9.4.2]{McDSal02:Jcurves}, which is formulated for star-shaped sets in $\R^4$.

\begin{definition}\label{def:star_shaped}
    A star-shaped subset $W$ of a Liouville manifold $(X,d\lambda)$ is a subset that contains the Liouville skeleton of $\lambda$ and $\phi^{\lambda}_t(x)\in W$ for all $t\leq 0$ and $x\in W$.
\end{definition}

Let \(p\) be the coordinate in the cotangent direction in \(T^*S^1\) and \(r\) be the radius function on \(\R^2\).
Let \(\mathcal{X}=p\partial_p+\frac{1}{2}r\partial_r\) be the standard Liouville vector field on \(\TSR\). 

\begin{theorem}\label{th:McDSalalternative}
    Let $(X,\omega)$ be a connected minimal symplectic 4-manifold and $\mathcal{K} \subset X$ be a compact set such that there exists a symplectomorphism $\psi \colon \left(\TSR\right) \setminus \mathcal{V} \to X \setminus \mathcal{K}$, where $\mathcal{V} \subset \TSR$ is a star-shaped compact set.
    Then $(X,\omega)$ is symplectomorphic to $(\TSR,\omega_0)$.
    Moreover, for every open neighbourhood $\mathcal{U} \subset X$ of~$\mathcal{K}$, the symplectomorphism can be chosen to agree with~$\psi^{-1}$ on~$X \setminus \mathcal{U}$.
\end{theorem}

For the proof of Theorem~\ref{th:McDSalalternative} we need the following result.

\begin{theorem}[\normalfont{\cite[Lemma 3.1.3]{CrGaMaMcD25} and \cite[Theorem 5.2.4]{ABEHS25}}]
\label{th:GrMcD_threespheres}
    Let $\omega_{a,b}$ be the usual split symplectic form on $S^2 \times S^2$ that gives the factors areas~$a$ and~$b$, respectively. 
    Denote by $p_N$ and $p_S$ the North and South poles of~$S^2$, and consider the three spheres $S_1 = S^2 \times \{p_S\}$ and 
    $S_2^N = \{p_N\} \times S^2$, $S_2^S = \{p_S\} \times S^2$ in~$S^2 \times S^2$.
    Let $\omega$ be a symplectic form on $S^2 \times S^2$ that agrees with $\omega_{a,b}$
    on a neighbourhood of the configuration $\mathcal{S}:=S_1 \cup S_2^N \cup S_2^S$. 
    Then there exists a symplectomorphism $\psi \colon (S^2 \times S^2, \omega_{a,b}) \to (S^2 \times S^2, \omega)$
    that is the identity on a neighbourhood of $\mathcal{S}$. 
\end{theorem}

Theorem~\ref{th:McDSalalternative} follows from Theorem~\ref{th:GrMcD_threespheres} exactly as the Gromov--McDuff Theorem \cite[Theorem~9.4.2]{McDSal02:Jcurves} follows from \cite[Theorem 9.4.7]{McDSal02:Jcurves}. 
Note that \cite[Lemma 9.4.10]{McDSal02:Jcurves} is also readily adapted; one just needs to replace the scaling maps $x \mapsto tx$ on $\R^4$ by the dilations 
on $T^*S^1 \times \R^2$ induced by the Liouville flow of~$\mathcal{X}$.

With these preparations in place we are ready to prove \cref{th:symp}.

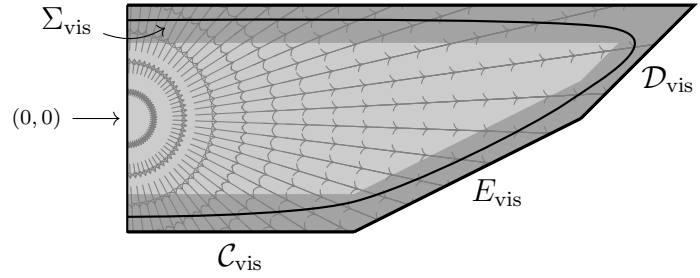
\begin{figure}[ht]
    \centering
    \begin{tikzpicture}
        \fill[opacity=0.2] 
            (0,3) -- (0,0) -- (3,0) -- (6,1.5) -- (7.5,3) -- (0,3);
        \fill[opacity=0.2] 
            (0,0.5) -- (0,0) -- (3,0) -- (6,1.5) -- (7.5,3) -- (0,3) -- (0,2.5) -- (6.5,2.5) -- (6,2) -- (3,0.5) -- (0,0.5);
        \begin{scope}[shift={(0,1.5)}]
            \clip (0,1.5) -- (0,-1.5) -- (3,-1.5) -- (6,0) -- (7.5,1.5) -- (0,1.5);
            \foreach \t in {88,83,78,...,-88} {\foreach \r in {0.3,0.6,...,9} {\draw[gray,->] (\t:\r) -- ++ (\t:\r/3);};};
        \end{scope}
        \draw[thick] 
            (0,3) -- (0,0) -- (3,0) -- (6,1.5) -- (7.5,3) -- (0,3);
        \draw[very thick]
            (0,0) -- node[anchor=north] {$\mathcal{C}_\vis$} (3,0);
        \draw[very thick]
            (3,0) -- node[anchor=north west, xshift=-3pt, yshift=3pt] {$E_\vis$} (6,1.5);
        \draw[very thick]
            (6,1.5) -- node[anchor=north west, xshift=-3pt, yshift=3pt] {$\mathcal{D}_\vis$} (7.5,3) -- (0,3);
        \draw[thick] plot [smooth, tension=0.5] 
            coordinates {(0,0.2) (3,0.4) (5.8,1.7) (6.4,2.7) (0,2.8)};
        \node (S) at (-0.8,2.8) {$\Sigma_\vis$};
        \node (Z) at (-1.2,1.5) {\tiny $(0,0)$};
        \draw[->] (Z) to (-0.1,1.5);
        \draw[->] (S) to[out=-20, in=210] (0.5,2.7);
        \end{tikzpicture}
        \caption{The visible situation in the case $(p,q)=(2,1)$. The linear chain of spheres $\mathcal{T}_\vis$ is formed by the Wahl chain $\mathcal{C}_\vis$, consisting of a single $(-4)$-sphere in this case, and the compactifying divisor $\mathcal{D}_\vis$, formed by a $(+3)$-sphere and a $(-1)$-sphere, connected by an exceptional divisor $E_\vis$. The boundary $\Sigma_\vis$ of a normal neighbourhood of the configuration $\mathcal{T}_\vis=\mathcal{C}_\vis \cup E_\vis \cup \mathcal{D}_\vis$ is a contact hypersurface. $\Sigma_\vis$ is diffeomorphic to $S^1 \times S^2$ and is star-shaped with respect to the Liouville field shown in gray.}
    \label{fig:vis_case}
\end{figure}

\begin{proof}[Proof of \cref{th:symp}]
\label{prf:symp}
    Pick compactification data $(\alpha,\bm{\rho})$, as in \cref{def:compactification}, such that the embedding $\iota$ also defines an embedding of $B_{p,q}(\varepsilon)$ into $X_{p,q}(\alpha,\bm{\rho})$.
    We obtain two symplectic manifolds by rationally blowing up $X$ along $\iota$ and $\iota_\vis$, which we denote by $\widetilde{X}$ and $\widetilde{X}_\vis$.
    Recall that both of these symplectic manifolds contain a linear chain of symplectic spheres $\mathcal{T}$ and $\mathcal{T}_\vis$ that consists of the Wahl chain and the compactifying divisor, connected by an exceptional sphere, as shown in \cref{thm:regulation_of_tilX}.
    \cref{fig:vis_case} shows the visible situation.
    
    The symplectic neighbourhood theorem gives symplectic embeddings 
    \(\psi_{\vis} \colon \nu \hookrightarrow \widetilde{X}_\vis\) and
    \(\psi \colon \nu \hookrightarrow \widetilde{X}\) of a plumbing of normal bundles, whose boundary $\Sigma_\vis$ and $\Sigma$, respectively, are contact-type hypersurfaces contactomorphic to \(S^1\times S^2\).
    The hypersurface~$\Sigma_\vis$ can be assumed to be a visible contact hypersurface in a toric domain whose completion is $\TSR$ and is star-shaped with respect to a specific Liouville field, see \cref{fig:vis_case}.
    Let $\mathcal{K}_\vis$ be the closure in $\widetilde{X}_\vis$ of the complement of $\psi_{\vis}(\nu)$; its boundary is $\Sigma_{\vis}$. 
    Similarly, let $\mathcal{K}$ be the closure in~$\widetilde{X}$ of the complement of~$\psi (\nu)$. 
    The boundary of~$\mathcal{K}$ is~$\Sigma$, 
    and $\mathcal{K}$ is minimal by Corollary \ref{cor:complement_of_T_is_minimal}.
    Our goal is to extend the symplectomorphism $\psi \circ \psi_{\vis}^{-1} \colon \psi_{\vis} (\nu) \to \psi(\nu)$ to a symplectomorphism
    $$
    \Psi \colon \widetilde{X}_\vis = \psi_\vis(\nu) \cup \mathcal{K}_{\vis} \to \psi(\nu) \cup \mathcal{K} = \widetilde{X}
    $$
    that takes $\ct_\vis$ to $\ct$.
    Choose coordinates on the almost toric base diagram in \cref{fig:vis_case} so that the source of the Liouville vector field is at $(0,0)$, and take $\delta >0$ so small that the "ball" of radius $\delta$ is contained in this almost toric base diagram.
    Write $S_\vis(\delta)$ for the set in $\widetilde{X}_{\vis}$ represented by this ball.
    It is canonically symplectomorphic to the set $S(\delta):=\{p^2+r^2\leq \delta^2\} \subseteq\TSR$, where $p$ is the cotangent direction in $T^*S^1$ and $r$ is the radius function on $\mathbb{R}^2$. 
    Let $\mathcal{X}_\vis := \sum_{i=1}^2 p_i \partial_{p_i}$ be the Liouville vector field centred on $(0,0)$ defined by the action-angle coordinates, and let 
    $\mathcal{X} := p\partial_p + \frac 12 r \partial_{r}$ be the standard Liouville vector field on $\TSR$.
    
    Using the flows of $\mathcal{X}_\vis$ and $\mathcal{X}$ we construct a symplectic embedding $\mathcal{K}_{\vis} \hookrightarrow \TSR$ onto a closed star-shaped region with smooth boundary.
    We call the image $\mathcal{K}_{\str}$ and its boundary $\Sigma_{\str}$.
    Explicitly, 
    $$
    \mathcal{K}_{\str} =S^1\times\{(0,0)\} \cup \left\{ \phi_{\mathcal{X}}^t (x) \mid 
    x \in \partial S(\delta), \,-\infty < t \leq t^*(x)\right\}
    $$ 
    where $t^*(x)$ is defined by $\phi_{\mathcal{X}}^{t^*(x)}(x) \in \Sigma_{\vis}$.
    Since $\mathcal{X}_\vis$ is transverse to $\Sigma_{\vis}$, we find $\epsilon >0$ such that the trajectories $\phi_{\mathcal{X}_\vis}^t(x)$ exist for $x \in \Sigma_{\vis}$ and $0 \leq t \leq \epsilon$.
    Set
    $$
    \mathcal{U}_{\vis} \coloneqq \left\{ \phi_{\mathcal{X}_\vis}^t(x) \mid x \in \Sigma_{\vis}, \, 0 < t < \epsilon \right\} \,\subset\, \widetilde{X}_{\vis} \setminus \ct_{\vis} \hspace{0.1cm}\text{ and }\hspace{0.1cm}
    \mathcal{U}_{\str} := \left\{ \phi_{\mathcal{X}}^t(x) \mid x \in \Sigma_{\str}, \, 0 < t < \epsilon \right\}.
    $$
    Using again the flows of $\mathcal{X}_\vis$ and $\mathcal{X}$ we construct the symplectomorphism $\xi_{\vis} \colon \mathcal{U}_{\str} \to \mathcal{U}_{\vis}$ that maps flow segments to flow segments.
    Also set $\mathcal{U} \,:=\, (\psi \circ \psi_{\vis}^{-1})(\mathcal{U}_{\vis}) \subset \widetilde{X} \setminus \ct$.
    Then we also have the symplectomorphism 
    $$
    \xi :=  \psi \circ \psi_{\vis}^{-1} \circ \xi_{\vis} \colon \mathcal{U}_{\str} \to \mathcal{U} .
    $$
    Define two open symplectic manifolds $X_{\vis}$ and $X$ by
    $$
    X_{\vis} = \mathcal{U}_{\vis} \bigcup_{\xi_{\vis}} (\TSR \setminus \mathcal{K}_{\str})\quad\text{and}\quad
    X = \mathcal{U} \bigcup_{\xi} (\TSR \setminus \mathcal{K}_{\str}) .
    $$
    Then the maps 
    \begin{align*}
        \Psi_{\vis}  \colon & \TSR \setminus \mathcal{K}_{\str} \to X_{\vis}, &
        \Psi_{\vis} |_{\mathcal{U}_{\str}} &= \xi_{\vis}, & \Psi_{\vis} |_{\TSR \setminus \mathcal{K}_{\str}} & = \text{id}, \\ 
        \Psi  \colon &\TSR \setminus \mathcal{K}_{\str} \to X, & 
        \Psi |_{\mathcal{U}_{\str}} &= \xi, & \Psi |_{\TSR \setminus \mathcal{K}_{\str}} &= \text{id},
    \end{align*}
    are symplectomorphisms.
    They both extend over $\mathcal{K}_{\str}$ and hence to $\TSR$.
    More precisely, take $\mathcal{V}_{\str} := \bigcup_{\frac{\epsilon}2 < t < \epsilon} \phi_{\mathcal{X}'}^t (\Sigma_{\str})$.
    Since $\mathcal{K}$ is minimal, \cref{th:McDSalalternative} guarantees the existence of symplectomorphisms
    $$
    \widehat \Psi_{\vis} \colon \TSR \to X_{\vis} \cup \mathcal{K}_{\vis}
    \quad \mbox{and} \quad  \widehat \Psi \colon \TSR \to X \cup \mathcal{K}
    $$
    such that 
    $$
    \widehat \Psi_{\vis} |_{\mathcal{V}_{\str}} = \xi_{\vis}
    \quad \mbox{and} \quad  
    \widehat \Psi |_{\mathcal{V}_{\str}} = \xi .
    $$
    Set $\widetilde{\mathcal{K}}_{\vis} := \mathcal{K}_{\vis} \cup \mathcal{U}_{\vis}$.
    Then the map $\widehat \chi \colon \widetilde{X}_{\vis} \setminus \ct_{\vis} \to \widetilde{X} \setminus \ct$
    defined by
    $$
    \widehat{\chi}|_{\widetilde{\mathcal{K}}_{\vis}} = \widehat{\Psi} \circ \widehat{\Psi}_{\vis}^{-1}
    \quad\mbox{and}\quad  
    \widehat{\chi}|_{\widetilde{X}_{\vis} 
    \setminus (\ct_{\vis} \cup \widetilde{\mathcal{K}}_{\vis})}
    = \psi \circ \psi_{\vis}^{-1}
    $$
    is a symplectomorphism that agrees near $\ct_\vis$ with $\psi \circ \psi_{\vis}^{-1}$.
    Therefore, $\widehat \chi$ descends to a symplectomorphism
    $\chi \colon \widetilde{X}_{\vis} \to \widetilde{X}$ taking $\ct_{\vis}$ to $\ct$.
    The last step is then to glue this symplectomorphism appropriately with the symplectomorphism $\iota \circ \iota_\vis^{-1}$ on the interface that is given by the linear chain of spheres to obtain a symplectomorphism $\Phi \in \symp_c(B_{p,q})$ that intertwines the embeddings $\iota_\vis$ and $\iota$.
    The proof of this step is verbatim the proof of \cite[Corollary 5.2.5]{ABEHS25} and we therefore omit it here.
\end{proof}

\begin{figure}[ht]
    \centering
    \begin{tikzpicture}[scale=0.9]
       \begin{scope}[shift={(1,0)}]
            \node at (-0.6,3.2) {(b)};
            \fill[opacity=0.2] 
                (0,3.5) -- (0,-0.5) -- (8,-0.5) -- (8,3.5) -- (0,3.5);
            \fill[opacity=0.2] 
                (0,0.5) -- (0,0) -- (3,0) -- (6,1.5) -- (7.5,3) -- (0,3) -- (0,2.5) -- (6.5,2.5) -- (6,2) -- (3,0.5) -- (0,0.5);
            \begin{scope}[shift={(0,1.5)}]
                \clip (0,2) -- (0,-2) -- (8,-2) -- (8,2) -- (0,2);
                \foreach \t in {88,83,78,...,-88} {\foreach \r in {0.3,0.6,...,9} {\draw[gray,->] (\t:\r) -- ++ (\t:\r/3);};};
            \end{scope}
            \fill[white] 
                (-0.1,0.5) -- (3,0.5) -- (6,2) -- (6.5,2.5) -- (-0.1,2.5);
            \draw[very thick]
                (0,0) -- node[anchor=north, yshift=2pt] {$\mathcal{C}$} (3,0);
            \draw[very thick]
                (3,0) -- node[anchor=north west, xshift=-3pt, yshift=3pt] {$E$} (6,1.5);
            \draw[very thick]
                (6,1.5) -- node[anchor=north west, xshift=-3pt, yshift=3pt] {$\mathcal{D}$} (7.5,3) -- (0,3);
            \draw[thick] plot [smooth, tension=0.5] 
                coordinates {(0,0.2) (3,0.4) (5.8,1.7) (6.4,2.7) (0,2.8)};
            \draw[very thick]
                (0,0) -- (0,0.5)
                (0,3) -- (0,2.5);
            \draw[thick]
                (0,0) -- (0,-0.5)
                (0,3) -- (0,3.5);
            \node (S) at (-0.2,2) {$\Sigma_\text{star}$};
            \draw[->] (S) to[out=20, in=-80] (0.5,2.7);
        \end{scope}

        \begin{scope}[shift={(-9,0)}]
            \node at (-0.6,3.2) {(a)};
            \fill[opacity=0.2] 
                (0,3.5) -- (0,-0.5) -- (8,-0.5) -- (8,3.5) -- (0,3.5);
            \fill[opacity=0.2] 
                (0,0) -- (3,0) -- (6,1.5) -- (7.5,3) -- (0,3) -- (0,0);
            \begin{scope}[shift={(0,1.5)}]
                \clip (0,2) -- (0,-2) -- (8,-2) -- (8,2) -- (0,2);
                \foreach \t in {88,83,78,...,-88} {\foreach \r in {0.3,0.6,...,9} {\draw[gray,->] (\t:\r) -- ++ (\t:\r/3);};};
            \end{scope}
            \draw[very thick]
                (0,-0.5) -- (0,3.5);
            \draw[thick, dashed] 
                (0,3) -- (0,0) -- (3,0) -- (6,1.5) -- (7.5,3) -- (0,3);
            \draw[thick, dashed] plot [smooth, tension=0.5] 
                coordinates {(0,0.2) (3,0.4) (5.8,1.7) (6.4,2.7) (0,2.8)};
        \end{scope}
        \end{tikzpicture}
         \caption{The "visible" part of the constructions in \cref{sec: constructing the symp}. (a) shows how the Delzant polytope of $\til{X}_\vis\setminus \ct_\vis$ embeds into the standard Delzant polytope of $\TSR$, respecting the Liouville vector fields. How $\Sigma_\vis$ embeds is also shown. (b) shows how $\psi(\nu)$ can be completed to match the situation in (a).}
\end{figure}
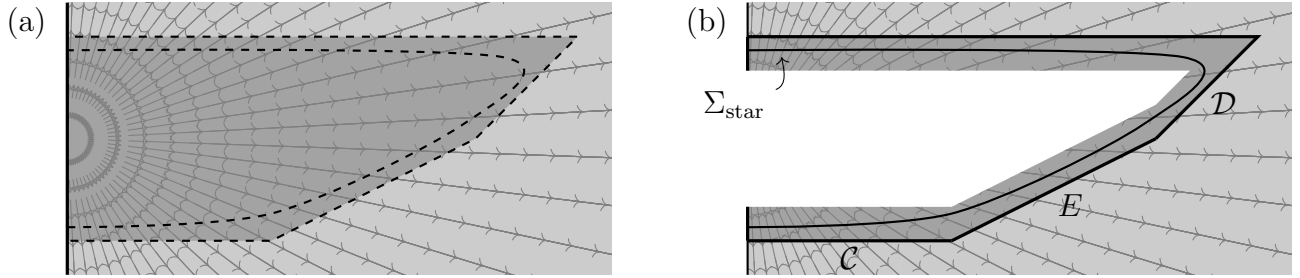

\section{Part 2: From symp to Ham}\label{sec: symp to ham}

Our goal in this section is to compute the homotopy groups of $\symp_c(B_{p,q})$, in order to show that any compactly supported symplectomorphism is, up to Hamiltonian isotopy, a power of a pintwist (see Definition \ref{def: pintwist}).
First, we show how the space $W_{p,q}=B_{p,q}\backslash K_{p,q}$ can be seen as a star-shaped domain in $T^*Wh$, the cotangent bundle of the Whitney sphere, and then we show how this inclusion induces a weak homotopy equivalence (w.h.e.) $\symp_c(W_{p,q})\hookrightarrow \symp_c(T^*Wh)$.
Then we explain how to compute $\symp_c(B_{p,q})$ by relating it to $\symp_c(W_{p,q})$ and $\mS(K_{p,q})$, the space of symplectic cylinders which agree with $K_{p,q}$ outside a compact subset.

\begin{remark}
    We remind the reader that, since $H^1(B_{p,q};\R)=0$, any symplectic isotopy is also Hamiltonian.
    We therefore will often blur the distinction between the two notions as we see fit.
\end{remark}

\begin{remark}
    Since the compactly supported symplectomorphism groups we are going to consider have the homotopy type of countable CW-complexes, the notions of homotopy equivalence and weak homotopy equivalence coincide for them.
    However, working with homotopy equivalences is enough for our purposes.
\end{remark}

\subsection{$T^*Wh$ as a Liouville completion of $B_{p,q}\setminus K_{p,q}$}\label{sec: liouville completions}
The fundamental observation that will allow us to compute the weak homotopy type of $\symp_c(B_{p,q})$ is the fact that the manifold $W_{p,q}=B_{p,q}\setminus K_{p,q}$, equipped with the restriction of the standard symplectic form $\omega_{p,q}$, has contractible symplectomorphism group.
Here, we show that there exists a (complete) Liouville manifold $(T^*Wh,d\lambda_{\text{can}})$ such that $(W_{p,q},\omega_{p,q})$ symplectically embeds in it as a star-shaped open subset.
Recall from \cref{def:star_shaped} that a star-shaped subset of a Liouville manifold is a subset that contains the Liouville skeleton and is negatively invariant.
The main property of star-shaped subsets we will use is the following.

\begin{lemma}\label{lemma: star shaped symp homotopy}
    Let $W$ be a star-shaped subset of a Liouville manifold $(X,d\lambda)$.
    Then the inclusion $i:W\hookrightarrow X$ induces a weak homotopy equivalence $i_*:\symp_c(W,d\lambda)\rightarrow\symp_c(X,d\lambda)$.
\end{lemma}

Our proof follows that of \cite[Proposition 2.1]{Ev14} with minor changes.
We include it for the convenience of the reader.

\begin{proof}
    To prove surjectivity of $i_*$ at the level of homotopy groups, assume that $g_s$ is a continuous family of compactly supported symplectomorphisms of $X$, with $s\in S^n$. 
    Since $S^n$ is compact, there is a compact subset $K\subset X$ such that the support of all elements in the image of $g$ is contained in $K$.
    Since $W$ contains the $\lambda$-skeleton, there exists a time $T>0$ such that $\phi^{\lambda}_{t}(K)\subset W$, for all $t\leq -T$.
    Therefore, for any element $g_s$ the endpoint $g_{s,-T}$ of the symplectic isotopy 
    $g_{s,t}=\phi_{-t}^\lambda\circ g_s\circ \phi^\lambda_t$ has support in $W$.
    Hence conjugation with $\phi_t^\lambda$ homotopes the family $g\in \pi_n(\symp_c(X))$ to a family in $\pi_n(\symp_c(W))$, and so $i_*$ is surjective at the level of homotopy groups.
    
    Similarly, for injectivity, consider two families $g,f:S^n\rightarrow \symp_c(W)$ and suppose they are homotopic through maps in $\symp_c(X)$.
    Conjugating by $\phi_{\tau(t)}^\lambda$, where $\tau:[0,1]\to (-\infty,0]$ is a sufficiently U-shaped smooth function satisfying $\tau(0)=\tau(1)=0$, we obtain a homotopy from $g$ to $f$ supported in $W$.
\end{proof}

We will now show that $W_{p,q}$ can be viewed as a star-shaped subset of the Liouville manifold $T^*Wh$, which we will introduce in a moment.
In combination with the lemma above, this implies that $\symp_c(T^*Wh)$ and $\symp_c(W_{p,q})$ are weakly homotopy equivalent.

The symplectic manifold $T^*Wh$ was introduced by Dimitroglou-Rizell in \cite{Dim17}, with an eye towards classification of Lagrangian submanifolds therein.
Let us swiftly recall the points that are relevant for us,
referring to \cite[Section 3]{Dim17} for details.
Let $Wh$ be a Lagrangian \textit{Whitney sphere}, i.e.\ an immersed sphere with a single transverse positive self-intersection.
The space $(D^*Wh,d\lambda)$ is the model neighbourhood of a Lagrangian Whitney sphere and it is constructed by self-plumbing the cotangent disc bundle of a sphere.
By carefully completing this Liouville domain, one obtains a Liouville manifold $(T^*Wh,d\lambda_\text{can})$ whose Liouville skeleton is the Lagrangian Whitney sphere, and in this sense $(T^*Wh,d\lambda_\text{can})$ can be viewed as the cotangent bundle of the Whitney sphere.
For a detailed exposition of this construction, see \cite[Subsections 3.1 and 3.2]{Dim17}.

The space $(T^*Wh,d\lambda_\text{can})$ also carries an explicit almost toric fibration $T^*Wh\xrightarrow{\pi} \mathbb{R}^2$, where the skeletal Whitney sphere of $T^*Wh$ is the unique singular torus fibre over the point~$(1,1)$.
Furthermore, the Liouville flow of $\lambda_\text{can}$ respects $\pi$, in the sense that it takes fibres to fibres, and the differential $D\pi$ pushes forward the Liouville vector field to the radial vector field on $\mathbb{R}^2$ emanating from the unique node at $(1,1)$, outside a small disc around the node.
This is proven in \cite[Propositions 3.7 and 3.8]{Dim17}.
The discussion is summarized in \cref{fig:atbdLiouville}.
\begin{figure}[htb]
  \begin{center}   
    \begin{tikzpicture}[scale=0.8]
    \begin{scope}[shift={(-9,0)}]
        \node at (-1.1,4.2) {(a)};
        \fill[opacity=0.2] 
            (0,4) -- (0,0) -- (4,0) -- (0,4);
        \draw[opacity=0.5]
            (-0.5,0) -- (4.5,0)
            (0,-0.5) -- (0,4.5);
        \draw[very thick] (0,4) -- (0,0) -- (4,0) -- (0,4);
        \draw[dashed, thick] (1.2,1.2) node[cross] {} -- (0,0);
        \node at (2,2) [above right] {\small $\ell_{\infty}$};
        \node at (0,0) [below left] {\small $C$};
        \node at (1.2,1.2) [below right] {\small $(1,1)$};
    \end{scope}

    \draw[->] (-7,3) to[out=20,in=160] (-4,3);
    \node at (-5.5,3.3) [above] {\small $\cdot \setminus (\ell_{\infty} \cup C)$};
    
    \begin{scope}[shift={(-3,0)}]
        \fill[opacity=0.2] 
            (0,4) -- (0,0) -- (4,0) -- (0,4);
        \draw[opacity=0.5]
            (-0.5,0) -- (4.5,0)
            (0,-0.5) -- (0,4.5);
        \clip (0,4) -- (0,0) -- (4,0) -- (0,4);
        \begin{scope}[shift={(1.2,1.2)}]
            \draw[gray] (0,0) circle (0.5);
            \fill[opacity=0.1] (0,0) circle (0.5);
            \foreach \t in {0,10,20,...,360} {\foreach \r in {0.5,0.8,...,3.5} {\draw[gray,->] (\t:\r) -- ++ (\t:\r);};};
            \foreach \t in {0,10,20,...,360} {\draw[gray,->] (\t:0) to[out=\t+60+20*sin(5*\t),in=\t+180] (\t:0.5);};
        \end{scope}
        \draw[dashed, thick] (1.2,1.2) node[cross] {} -- (0,0);
    \end{scope}
    
    \begin{scope}[shift={(5,0)}]
        \node at (-1.1,4.2) {(b)};
        \fill[opacity=0.2] 
            (-0.5,4.5) -- (-0.5,-0.5) -- (4.5,-0.5) -- (4.5,4.5);
        \draw[opacity=0.5]
            (-0.5,0) -- (4.5,0)
            (0,-0.5) -- (0,4.5);
        \clip (-0.5,4.5) -- (-0.5,-0.5) -- (4.5,-0.5) -- (4.5,4.5);
        \begin{scope}[shift={(1.2,1.2)}]
            \draw[gray] (0,0) circle (0.5);
            \fill[opacity=0.1] (0,0) circle (0.5);
            \foreach \t in {0,10,20,...,360} {\foreach \r in {0.5,0.8,...,3.5} {\draw[gray,->] (\t:\r) -- ++ (\t:\r);};};
            \foreach \t in {0,10,20,...,360} {\draw[gray,->] (\t:0) to[out=\t+60+20*sin(5*\t),in=\t+180] (\t:0.5);};
        \end{scope}
        \draw[dashed, thick] (1.2,1.2) node[cross] {} -- (-0.5,-0.5);
    \end{scope}
    \end{tikzpicture}
    \caption{(a) an almost toric base diagram of $\mathbb{C}P^2$ and the same almost toric base diagram with the line $\ell_\infty$ and the conic $C$ removed. Moreover, the Liouville vector field emanating from the Whitney sphere is indicated. The Liouville vector field is not radial in a neighbourhood of the node, which is shown schematically in the figure. (b) The almost toric base diagram of $\C P ^2\backslash (\ell_\infty\cup C)$ extends to an almost toric base diagram of $T^*Wh$ by completing with respect to the Liouville flow.} 
    \label{fig:atbdLiouville}
  \end{center}
\end{figure}
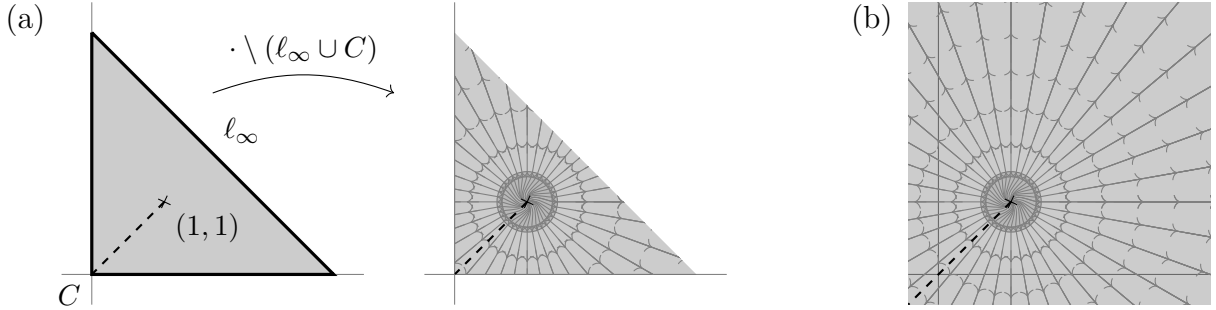

Having collected these facts, we can now use the general theory of almost toric fibrations, as explained for example in \cite{Sym02:Fourtwo, Ev23:Book} to construct the wanted symplectic embeddings $W_{p,q}\hookrightarrow T^*Wh$.
First, we do this on the level of the corresponding base diagrams.

\begin{proposition}\label{prop: w_pq in T*wh}
    The almost toric base diagram of $W_{p,q}$ admits an affine embedding into the almost toric base diagram of $T^* Wh$.
    Therefore, there exists a symplectic embedding $W_{p,q}\hookrightarrow T^*Wh$ which is fibred in the complement of an arbitrarily small neighbourhood of the singular fibre.
    In particular, we can view $W_{p,q}$ as a star-shaped subspace of $(T^*Wh,d\lambda_{can})$.
\end{proposition}

\begin{proof}
    Since $0<q<p$ are coprime, there exist integers $a,b\geq 0$ such that $bq-ap=1$.
    For such $a,b$, the affine transformation 
    \[
        M=
        \left(
            \begin{matrix}
            q-a & b-p\\
            -a & b
            \end{matrix}
        \right)\in \text{SL}_2(\mathbb{Z})
    \] 
    maps the standard wedge $\ATF_{p,q}$ to one in which the branch cut points in the $(1,1)$-direction.
    Since we have removed the boundary from $\ATF_{p,q}$, it is clear that the sheared base diagram is an affine subset of the affine base diagram of $T^*Wh$.
    This is shown in \cref{fig:atbdLiouville_Bpq}.
    
    Now \cite[Proposition 4.8]{Sym02:Fourtwo} guarantees that the affine embedding of the base diagrams lifts to the desired symplectic embedding.
    Finally, the fact that the embedding thereby obtained $\iota:W_{p,q}\hookrightarrow T^*Wh$ is star-shaped follows because the base diagram contains the node.
    Furthermore, since $\iota$ respects the almost toric fibrations outside of a small disc around the node, the star-shaped property for $\iota(W_{p,q})$ follows directly from the fact that the base diagram of $W_{p,q}$ is a star-shaped subset of $\mathbb{R}^2$.
\end{proof}

\begin{figure}[htb]
  \begin{center}   
    \begin{tikzpicture}[scale=0.8]
    \begin{scope}[shift={(-9.5,0)}]
        \fill[opacity=0.2] 
            (0,2) -- (0,0) -- (8,2) -- (0,2);
        \draw[opacity=0.5]
            (-0.5,0) -- (8,0)
            (0,-0.5) -- (0,2);
        \draw[thick, mid arrow,  opacity=0.2] 
            (0,0) -- node[left, opacity=1]{\tiny $\pmat{0 \\ 1}$} (0,2);
        \draw[thick, mid arrow,  opacity=0.2]
            (0,0) --  node[below, xshift=16pt, yshift=4pt, opacity=1]{\tiny $\pmat{p^2 \\ pq-1}$}(8,2);
        \draw[dashed] (2,1) node[cross] {} -- (0,0);
    \end{scope}

    \draw[->] (-1,1.3) to[out=-20,in=200] (1,1.3);
    \node at (0,1.1) [below] {\small $M$};
    
    \begin{scope}[shift={(3.5,0)}]
        \fill[opacity=0.2] 
            (-2,2) -- (0,0) -- (6,2);
        \draw[thick, mid arrow, opacity=0.2] 
            (0,0) -- node[below, xshift=-14pt, yshift=4pt, opacity=1]{\tiny $\pmat{b-p \\ b}$} (-2,2);
        \draw[thick, mid arrow, opacity=0.2]
            (0,0) --  node[below, xshift=14pt, yshift=4pt, opacity=1]{\tiny $\pmat{2p-b \\ p-b}$} (6,2);
        \draw[opacity=0.5]
            (-2,0) -- (6,0)
            (0,-0.5) -- (0,2);
        \clip (-2,2) -- (0,0) -- (6,2);
        \begin{scope}[shift={(1.2,1.2)}]
            \draw[gray] (0,0) circle (0.5);
            \fill[opacity=0.1] (0,0) circle (0.5);
            \foreach \t in {0,10,20,...,360} {\foreach \r in {0.5,0.8,...,3.5} {\draw[gray,->] (\t:\r) -- ++ (\t:\r);};};
            \foreach \t in {0,10,20,...,360} {\draw[gray,->] (\t:0) to[out=\t+60+20*sin(5*\t),in=\t+180] (\t:0.5);};
        \end{scope}
        \draw[dashed] (1.2,1.2) node[cross] {} -- (0,0);
    \end{scope}
    \end{tikzpicture}
    \caption{The almost toric base diagram of $W_{p,q}$, as a subset of the toric base diagram of $T^*Wh$.}
    \label{fig:atbdLiouville_Bpq}
  \end{center}
\end{figure}
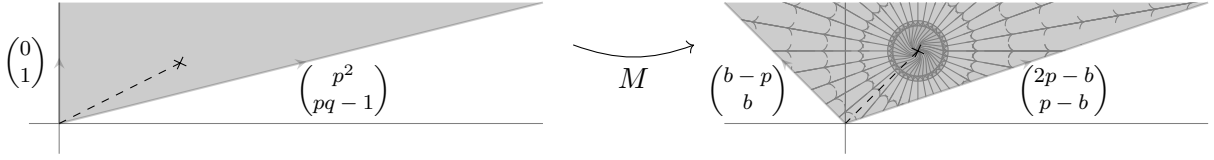

\begin{remark}
    For the special case $(p,q)=(1,1)$, the above proof shows that $D^*Wh$, the unit cotangent disc bundle of the Whitney sphere, can be symplectically identified with $B_{1,1}(1,1)\backslash C_{1,1}$, which is exactly Rizell's point of view in \cite{Dim17}.
\end{remark}

Combining \cref{lemma: star shaped symp homotopy} and \cref{prop: w_pq in T*wh} we immediately get:

\begin{corollary}
    The symplectic embedding $W_{p,q}\hookrightarrow T^*Wh$ induces a weak homotopy equivalence at the level of compactly supported symplectomorphism groups.
\end{corollary}

\subsection{Comb diagrams}\label{sec: comb diagrams}

Let $C$ be a symplectic submanifold of $(M,\omega)$ and let $\mathcal{S}(C)$ be the space of (unparametrized) symplectic submanifolds of $M$ symplectomorphic to $C$ which agree with $C$ outside a (not fixed) compact set.
If $\mathcal{S}(C)$ is path-connected, the natural action of $\symp_c(M)$ on $\mathcal{S}(C)$ is transitive since any path of symplectic submanifolds induces an ambient symplectic isotopy.
See \cite[Proposition 4]{Au97}.
In that case, we have the homotopy-action sequence
\begin{equation}\label{eq:action_fib}
    \begin{tikzcd}
    \pr(C) \arrow[r] & \symp_c(M) \arrow[d] \\
                       & \mS(C)      
    \end{tikzcd}
\end{equation}
where $\pr(C)$\footnote{The notation $\pr$ refers to the fact that we are considering the subgroup of symplectomorphisms that set-wise preserve the submanifold $C$, in contrast with the subgroup of symplectomorphisms $\fx$ of those symplectomorphisms that point-wise fix $C$.} is the stabilizer of $C$ under the action of $\symp_c$.
To relate $\symp_c(M\setminus C)$ to $\pr(C)$ we consider the restriction map $\rho$:
\begin{equation}\label{eq:restriction_fib}
    \fx(C)\rightarrow \pr(C)\xrightarrow{\rho} \symp_c(C)
\end{equation}
and the normal derivative map $D$:
\begin{equation}\label{eq:derivative_fib}
    \ker(D)\hookrightarrow \fx(C)\xrightarrow{D} \aut(\nu C),
\end{equation}
where $\aut(\nu C)$ is the group of symplectic automorphisms of the normal bundle of $C$, which are the identity outside of a compact set.
By a standard Moser argument, $\ker(D)$ is weakly homotopy equivalent to $\symp_c(M\setminus C)$.
Therefore all these groups fit together in the following diagram which we will refer to as a \textit{Comb diagram}.
\begin{equation}
    \begin{tikzcd}
        \symp_c(M\setminus C) \arrow[r] & \fx(C) \arrow[r] \arrow[d, "D"] & \pr(C) \arrow[r] \arrow[d, "\rho"] & \symp_c(M) \arrow[d] \\
                              & \aut(\nu C)                         & \symp_c(C)                          & \mathcal{S}(C)    
    \end{tikzcd} 
\end{equation}
The key structural property of the Comb diagram that we will use multiple times is that it allows us to relate the (compactly supported) symplectomorphisms of $M$ and those of $M\setminus C$, based on the behaviour of $C$ under the action of the symplectomorphism group.
Thus comb diagrams help us to leverage known results about symplectomorphism groups to compute new ones.

We now demonstrate how to use this diagram to compute $\symp_c(T^*Wh)$, which will be needed for the computation of $\symp_c(B_{p,q})$.
Again we will rely on \cite{Dim17} for certain technical aspects of the symplectic geometry of $T^*Wh$.
The key input is that the relevant moduli space of symplectic cylinders has the homotopy type of a punctured plane. 

\begin{theorem}\label{thrm: symp of T*Wh}
    The space $\symp_c(D^*Wh)$ is weakly contractible. 
\end{theorem}
\begin{proof}
    For $M=B^4$ and $C$ a smooth conic tangent to the coordinate planes at the boundary, we have that $M\setminus C=D^*Wh$, as explained in detail in \cref{prop: w_pq in T*wh}; compare also with the original argument in \cite[Section 2 and 3]{Dim17}.
    Gromov shows that $\symp_c(B^4)$ is contractible, so our plan is to use the Comb diagram to deduce that $\symp_c(D^*Wh)$ is also weakly homotopy equivalent to a point.
    
    We start by showing that $\mS(C)\simeq \C^*$.
    Let $C_s$ be a family of smooth conics tangent to the coordinate planes at the boundary of the ball $B^4$, where $s$ is a spherical parameter of $S^{k}$ for some $k$.
    Pick a family of tame almost complex structures $J_s$ on $B^4$ such that $C_s$ is $J_s$ holomorphic.
    Since the space of almost complex structures is contractible, we may extend $J_s$ to a bigger family $J_{t,s}$ parametrized by $(t,s)\in I\times S^{k}$ such that $J_{0,s}$ is constantly the standard integrable complex structure and $J_{1,s}=J_{s}$.
    By Theorem 4.3 in \cite{Dim17}, the family $J_{t,s}$ gives a corresponding family of symplectic fibrations $f_{t,s}:\mathbb{C}P^2\backslash \ell_\infty\rightarrow \mathbb{C}$ with a unique singular value, say $w_{t,s}$.
    Because this singular value is unique, we can find smoothly varying complex numbers $z_{t,s}\neq w_{t,s}$.
    The $I\times S^{k}$-family of smooth symplectic conics
    \[C_{t,s}=f_{t,s}^{-1}(z_{t,s})\]
    interpolates between the original family $C_s=C_{1,s}$ and the family $C_{0,s}=f_{0,s}^{-1}(z_{0,s})$.
    Therefore any symplectic family of smooth conics can be isotoped to a family of smooth fibres of the standard Lefschetz fibration on $B^4$.
    Since the space of smooth fibres is just the smooth locus of the standard Lefschetz fibration, it is readily identified with $\mathbb{C}^*$ and thus our desired result follows.
    
    Let us move to investigating the restriction map $\rho:\pr(C) \to \symp_c(C)$ via the fibration \eqref{eq:restriction_fib}.
    Since $\symp_c(B^4)$ is contractible, we have that $\pr(C)$ is w.h.e.\ to $\mathbb{Z}$ via the fibration \eqref{eq:action_fib}, where a generator of $\Pr(C)$ is given by the monodromy over a simple loop of smooth symplectic conics around the singular one.
    This is the standard monodromy of the Milnor fibration of the $A_1$-curve singularity, which is known to be the standard Dehn twist of a cylinder, which also generates the mapping class group of $C$.
    Therefore $\rho$ sends a generator of $\pr(C)$ to a generator of $\symp_c(C)$ and so induces a weak homotopy equivalence.
    From this we also deduce that $\fx(C)=\ker \rho $ is weakly contractible.
    
    For the last step of the proof, we need to understand the derivative map $D:\text{Fix}(C)\to \aut(\nu C)$ by investigating the fibration \eqref{eq:derivative_fib}.
    Since $\fx(C)$ is weakly contractible and $\aut(\nu C)\simeq \mathbb{Z}$, the map $D$ necessarily induces a homotopy equivalence to its image and therefore its kernel, namely $\symp_c(M\setminus C)\simeq \symp_c(D^*Wh)$ must have trivial homotopy groups.\footnote{Our group $\aut(\nu C)$ is equivalent to the group $\mathcal{G}_2$ in \cite{Ev14}, i.e.\ the group of automorphisms of a symplectic rank $2$ bundle over a sphere, where over the north and south pole the automorphism equals the identity.}
    
    As explained in Section \ref{sec: liouville completions}, $T^*Wh$ is the Liouville completion of $D^*Wh$ and so $\symp_c(D^*Wh)$ is w.h.e.\ to $\symp_c(T^*Wh)$, which finishes the proof.
\end{proof}

\subsection{The extrinsic geometry of the central cylinders $K_{p,q}$}
The Comb diagram for $M=B_{p,q}$ and $C=K_{p,q}$ takes the following form: 
\begin{equation}\label{eq:comb_Bpq}
    \begin{tikzcd}
    \symp_c(D^*Wh) \arrow[r] & \fx(K_{p,q}) \arrow[d, "D"] \arrow[r] & \pr(K_{p,q}) \arrow[d, "\rho"] \arrow[r] & \symp_c(B_{p,q}) \arrow[d, "\text{transitive action}" description] \\
                                   & \aut(\nu K_{p,q})\simeq \mathbb{Z}           & \symp_c(K_{p,q})\simeq \mathbb{Z}                & \mathcal{S}(K_{p,q})\simeq pt                            
    \end{tikzcd}
\end{equation}
Since we will view $B_{p,q}$ as a partial compactification of $D^*Wh$ by adding the cylinder $K_{p,q}$ we need to understand the symplectic moduli space $\mS(K_{p,q})$, as well as the maps $D$ and $\rho$ appearing in the Comb diagram, which will then be used to deduce the weak homotopy equivalences claimed in \eqref{eq:comb_Bpq}.

\begin{proposition}
\label{prop: cylinders in Bpq are rigid}
    The symplectic moduli space $\mS(K_{p,q})$ associated to the central cylinder $K_{p,q}$ in $B_{p,q}$ for $p\geq 2$ is weakly contractible.
\end{proposition}
\begin{proof}
    The proof follows the same strategy as that of \cref{thrm: symp of T*Wh}.
    We momentarily drop $p,q$ from the notation for the sake of readability.
    Take an $S^k$-family of cylinders $K_s$ in $\mathcal{S}(K)$.
    Since the indexing set is compact there is a large compact set outside which all of the $K_s$ agree with the central cylinder $K$.
    Compactify the rational homology ball to $X$ such that the compactification does not interact with the compact set and write $F$ for the symplectic sphere corresponding to $K$ and $F_s$ for those corresponding to $K_s$.\footnote{Recall that the compactification procedure was explained in \cref{sec:compactification} and the geometry of the sphere $F$ was discussed in \cref{lma:central_becomes_fibre}.}
    By construction we have that $x_0=F \cap D_0 =F_s \cap D_0$ and $x_n=F \cap D_n= F_s \cap D_n$ and that $[F_s]=[F]$ for all $s$.
    Now pick an almost complex structure $J \in \cj_\tau(\cd)$ that makes $F$ holomorphic.
    Recall that by \cref{lma:ruling_compactification} this implies that $F$ is the unique smooth ruling passing through $x_0$.
    Choose a family $J_s \in \cj_\tau(\cd)$ of almost complex structures that agree near $\cd$ with $J$ and that make the corresponding $F_s$ holomorphic.
    We can now extend this family to an $I\times S^k$-family $J_{t,s}$ such that $J_{0,s}=J$ and $J_{1,s}=J_s$ since the space $\mathcal{J}_\tau(\mathcal{D})$ is weakly contractible.\footnote{That $\mathcal{J}_\tau(\mathcal{D})$ is weakly contractible is, for example, shown in \cite[Appendix]{Ev14}.}
    
    By \cref{lma:ruling_compactification}, there is a unique family of embedded symplectic spheres $F_{t,s}$ through $x_0$ and $x_n$ that are $J_{t,s}$-holomorphic.
    This provides a deformation of the family $F_s$ supported away from the compactifying divisor $\mathcal{D}$ through embedded symplectic spheres to $F$, and hence it provides the desired compactly supported homotopy of $K_s$ to $K$. 
\end{proof}

\begin{remark}
    The essential difference between the geometry of the ball and the geometry of a $B_{p,q}$ when $p\geq 2$ is exactly the fact that the cylinders $K_{p,q}$ are \textit{rigid}, i.e.\ there exists a \textit{unique} one for each $J$ when $p\geq 2$. In contrast, the central cylinder $K_{1,1}$ comes in a $\mathbb{C}^*$-family. 
\end{remark}

We now move on to investigate the restriction map $\rho$ and the derivative map $D$ of the Comb diagram.
These maps are usually easier to understand when the ambient manifold is compact because the symplectic curve $K$ is then a closed sphere and so $\symp_c(K)$ and $\aut(\nu K)$ are actually connected.
Therefore it is easy to show that the maps $\rho$ and $D$ are surjective, by building local Hamiltonian isotopies.
However, when $K$ is a cylinder, both $\symp_c(K)$ and $\aut(\nu K)$ have infinitely many (contractible) connected components, so understanding which of these components are hit by the maps becomes more delicate.

\begin{proposition}[Infinitesimal symplectic variations of $K_{p,q}$]
\label{prop: infinitesimal variations}
    The image of the restriction map $\rho:\pr(K_{p,q})\rightarrow \symp_c(K_{p,q})$ is exactly the connected components containing the multiples of $\tau^p$, the $p$-th power of the standard Dehn twist along a cylinder.
    The image of the derivative map $D:\fx(K_{p,q})\rightarrow \aut(\nu K_{p,q})$ is exactly the connected component of $\aut(\nu K_{p,q})$ containing the identity.
\end{proposition}

\begin{proof}
    For both statements, we will consider the universal $p$-fold covering $A_{p-1}\xrightarrow{\pi}B_{p,q}$ introduced in \cref{sec: intro bpq}.
    
    The first part follows from topological considerations.
    First, we notice that any $\psi\in \pr(K_{p,q})$ lifts to an element $\tilde{\psi}\in \symp(A_{p-1})$ that preserves the central cylinder of $A_{p-1}$.
    Since $\pi$ is a $p$-fold covering, if the restriction of $\tilde{\psi}$ on the central cylinder is $\tau^k$, then the restriction of $\psi$ is $\tau^{pk}$, as explained in \cref{lem: basic prop pintwists}.
    Also by this lemma, for the pintwist $\tau_{p,q}$ we have that the restriction to the central cylinder is $\tau^{p}$ and so $\rho(\tau_{p,q})$ generates the biggest subgroup actually possible.
    
    Moving to the derivative $D$, let us consider the corresponding, lifted, map on the universal cover $\widetilde{D}:\fx(K_{p-1})\rightarrow\aut(\nu K_{p-1})$.
    Since $K_{p-1}$ is exactly a $p$-fold cover of $K_{p,q}$, if the image of $D$ contained a component of $\aut(\nu K_{p,q})$ other than the identity, so would the image of $\widetilde{D}$.
    Therefore, it is enough to show that $\widetilde{D}$ is homotopically constant.
    
    To do this we argue as follows:
    it is well known that the Milnor fibre of the $A_{p-1}$ singularity is symplectomorphic to the linear plumbing of the cotangent bundles of $p-1$ Lagrangian spheres.
    Therefore there exists a Liouville form, say $\lambda$, whose flow has as skeleton a chain of Lagrangian spheres, which can also be taken as the matching spheres of the usual Lefschetz fibration on $A_{p-1}$. 
    Evidently, the matching spheres are disjoint from the central cylinder, so by conjugating with the flow of $\lambda$, we see that for any $\psi\in\fx(K_{p-1})\subseteq \symp_c(A_{p-1})$ there exists some Hamiltonian map $\phi$ such that $\psi\circ \phi$ is supported away from $K_{p-1}$.
    Since $\psi$ and $\psi\circ \phi$ are in $\fx(K_{p-1})$, so is $\phi$.
    Therefore our claim will follow if we can show that $[\widetilde{D}(\psi)]=[\widetilde{D}(\psi\circ\phi)]$ for such a Hamiltonian symplectomorphism $\phi\in \text{Ham}_c(A_{p-1})$.
    Equivalently, we will show that $[\widetilde{D}(\phi)]=0\in \Z\simeq \aut(\nu K)$.
    Since \cref{prop: cylinders in Bpq are rigid} can also be proven for the central $K_{p-1}$ in $A_{p-1}$, we have that $\pi_0(\mathcal{S}(K_{p-1}))\cong 1$.\footnote{This can be proven in exactly the same way as \cref{prop: cylinders in Bpq are rigid}. However, one then needs the result analogous to \cref{lma:ruling_compactification}, which is proven in \cite[Section 2]{Bu25}.}
    Thus, the fact $[\widetilde{D}(\phi)]=0$ can be deduced by examining the long exact sequence on homotopy groups: 
    \[\begin{tikzcd}[column sep=small]
        \pi_1(\symp_c(A_{p-1})) \arrow[r] & \pi_1(\mS(K_{p-1})) \arrow[r] & \pi_0(\pr(K_{p-1})) \arrow[r] & \pi_0(\symp_c (A_{p-1})) \arrow[r] &  \pi_0(\mathcal{S}(K_{p-1})).
    \end{tikzcd}\]
    Since $\phi$ is Hamiltonian, it is trivial in $\pi_0(\symp_c(A_{p-1}))$ and therefore the exact sequence shows that it is the image of the symplectic monodromy around some loop in $\mathcal{S}(K_{p-1})$. Any such monodromy, by definition, preserves the horizontal vector field defined by the velocity of the loop.
    Therefore $\widetilde{D}(\phi)$ has to be nullhomotopic, concluding the proof.
\end{proof}

\subsection{Computing $\symp_c(B_{p,q})$ and finishing the proof}
We now have all the ingredients to perform our calculation.

\begin{theorem}\label{thrm: symp B_pq}
    The group $\symp_c(B_{p,q})$ is weakly homotopy equivalent to $\mathbb{Z}$, where the generator can be taken to be the \textit{pintwist} $\tau_{p,q}$.
\end{theorem}

\begin{proof}
    We know all the necessary parts of the Comb diagram for a $B_{p,q}$:
    \[
        \begin{tikzcd}
            \symp_c(D^*Wh) \arrow[r, hook] & {\fx(K_{p,q})} \arrow[r, hook] \arrow[d, "D"] & {\pr(K_{p,q})} \arrow[r, hook] \arrow[d, "\rho"] & \symp_c(B_{p,q}) \arrow[d, two heads] \\
                                         & {\aut(\nu K_{p,q})}                           & {\symp_c(K_{p,q})}                                & {\mS(K_{p,q})}                      
        \end{tikzcd}
    .\]
    By \cref{prop: infinitesimal variations} we have that $D$ is homotopic to the constant map and therefore the inclusion map $\symp_c(D^*Wh)\hookrightarrow 
    \fx(K_{p,q})$ is a w.h.e., so Theorem \ref{thrm: symp of T*Wh} implies that $\fx(K_{p,q})$ is weakly contractible.
    Therefore $\pr(K_{p,q})$ is w.h.e.\ to its image under $\rho$, which by Proposition \ref{prop: infinitesimal variations} and the fact that the components of $\symp_c(K_{p,q})$ are contractible is w.h.e.\ to $\mathbb{Z}\langle\tau_{{p,q}}\rangle$.
    
    Finally, since by \cref{prop: cylinders in Bpq are rigid} $\mS(K_{p,q})$ is weakly contractible, the inclusion of $\mathbb{Z}\langle\tau_{{p,q}}\rangle$ to $\symp_c(B_{p,q})$ is a weak homotopy equivalence.
\end{proof}

\begin{corollary}[Extracting the Hamiltonian]
\label{cor: ham extract}\label{cor:uniqueness_star_shaped}
    Suppose that $Y \subseteq B_{p,q}$ is star-shaped and that $L\subseteq Y$ is a Lagrangian $(p,q)$-pinwheel.
    Then, there exists a compactly supported Hamiltonian diffeomorphism $\phi^H \in \text{\emph{Ham}}_c(Y)$ such that $\phi^H(L)=L_{p,q}^\std$.
\end{corollary}

\begin{proof}
    Pick some compactification $X_{p,q}$ as in \cref{def:compactification}.
    The subset $Z=X_{p,q}\backslash \mathcal{D}_{p,q}$ is a star-shaped domain of $B_{p,q}$.
    Therefore, there exists $T\leq 0$ such that $\phi_T^{\lambda}(Z):=Z'\subset Y$.
    Since $Y$ is star-shaped, $\phi_T^{\lambda}(L):=L'\subset Z'$.
    Recall that $\phi_t^{\lambda}$ acts on the symplectic form only by rescaling, so all the arguments of \cref{sec: Uniq up to symp}, in particular \cref{thrm: symp B_pq}, can be applied to $Z'$ in order to construct a symplectomorphism $\psi\in \symp_c(Z')$ such that $\psi (L')=L^{\std}_{p,q}$.
    
    Similarly, the arguments of \cref{sec: symp to ham} apply: $Z'$ is a star-shaped subset of $B_{p,q}$ and therefore the inclusion $Z'\hookrightarrow B_{p,q}$ induces a weak homotopy equivalence on the groups of compactly supported symplectomorphisms.
    Combined with the fact that the pintwist $\tau_{p,q}$ can be arranged to have support arbitrarily close to $L^{\std}_{p,q}$, we get that $\tau_{p,q}$ generates the symplectic mapping class group of $Z'$ and thus $\psi$, by \cref{thrm: symp B_pq}, factors as $\psi=\tau_{p,q}^k\circ \phi^H$, for $k\in \Z$ and $H$ a time-dependent Hamiltonian.
    
    The Hamiltonian diffeomorphism $\phi^H=\tau_{p,q}^{-k}\circ\psi$ maps $L'$ to $\tau_{p,q}^{-k}(L^{\std})$ which is Lagrangian isotopic to $L^{\std}$, by \cref{lem: basic prop pintwists}.
    Therefore, $L$ is Lagrangian isotopic to $L^{\std}$ and ultimately Hamiltonian isotopic to it, since $H^1(L^{\std},\R)=0$, as in \cref{lem: basic prop pintwists} (3).
\end{proof}

\section{Applications}
\label{sec:application}

\subsection{Non-squeezing theorems}
\label{subsec:nonsqueezing}

This section sheds light on the non-squeezing theorem proven by the first and third authors in \cite{AdHa25}.
There, a non-squeezing theorem for $B_{n,1}$ was proven for all $n \geq 1$. 
We reprove this theorem here and extend the theorem to all $B_{p,q}$. 
Moreover, we prove the non-squeezing theorem for two symplectically different pin-cylinders.

Recall that the pin-ellipsoids and pin-balls were defined via the almost toric base diagram $\ATF_{p,q}$ of $B_{p,q}$ in analogy to how the usual ball and ellipsoids in $\mathbb{C}^2$ can be defined by their moment image as shown in \cref{fig:atfBpq}, and that $B_{p,q}(\lambda):=E_{p,q}(\lambda,\lambda)$.
By a slight abuse of notation, we define for $\lambda>0$ the two pin-cylinders $E_{p,q}(\lambda,\infty)$ and $E_{p,q}(\infty,\lambda)$ via the almost toric base diagram $\ATF_{p,q}(\lambda,\infty)$ and $\ATF_{p,q}(\infty,\lambda)$, as shown in \cref{fig:atfrationalhomologycylinders}. 
\begin{figure}[ht]
  \centering
  \begin{tikzpicture}
    \begin{scope}[shift={(-6,-2)}]
        \fill[opacity=0.2] (0,1) -- (0,0) -- (4,1) -- (4,2);
        \draw[thick] (0,1) -- (0,0) -- (4,1);
        \draw[thick,dash pattern=on 7pt off 3pt] (0,1) -- (4,2);
        \draw[dashed] (1,0.5) node[cross] {} -- (0,0);
        \draw [decorate,decoration={brace,amplitude=5pt,raise=1ex}] 
            (0,0) -- (0,1) node[midway,xshift=-3ex]{\footnotesize $\lambda$};
        \node at (0,1) [above left] {\footnotesize $(0,\lambda)$};
    \end{scope}

    \begin{scope}[shift={(1,-2)}]
        \fill[opacity=0.2] (0,2) -- (0,0) -- (4,1) -- (4,2);
        \draw[thick] (0,2) -- (0,0) -- (4,1);
        \draw[dashed] (1,0.5) node[cross] {} -- (0,0);
        \draw[thick,dash pattern=on 7pt off 3pt] (4,1) -- (4,2);
        \draw [decorate,decoration={brace,amplitude=5pt,raise=1ex}] 
            (4,1) -- (0,0) node[midway,xshift=.8ex,yshift=-3.5ex]{\footnotesize $\lambda$};
        \node at (4,1) [right] {\footnotesize $(\lambda p^2,\lambda (pq-1))$};
    \end{scope}
  \end{tikzpicture}
  \caption{On the left, the almost toric base diagram $\ATF_{p,q}(\infty,\lambda)$ defining $E_{p,q}(\infty,\lambda)$; on the right, the corresponding almost toric base diagram for $E_{p,q}(\lambda,\infty)$.}
  \label{fig:atfrationalhomologycylinders}
\end{figure}
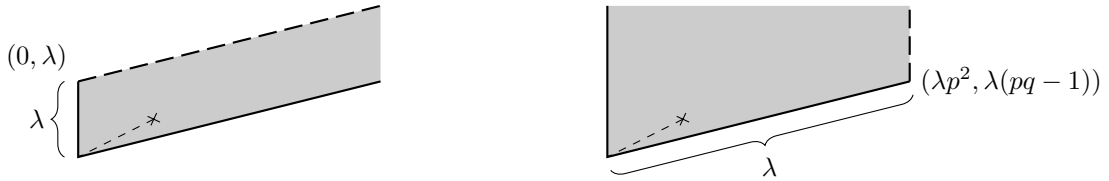

\begin{remark}
    It is important to observe that if $p\geq 3$, the two pin-cylinders $E_{p,q}(\lambda,\infty)$ and $E_{p,q}(\infty,\lambda)$ are not symplectomorphic.
    That the usual standard cylinder $E(\lambda,\infty)$ is symplectomorphic to $E(\infty,\lambda)$ is obvious by exchanging the coordinate planes and that $E_{2,1}(\lambda,\infty)$ and $E_{2,1}(\infty,\lambda)$ are symplectomorphic follows from the fact that $\ATF_{2,1}(\lambda,\infty)$ and $\ATF_{2,1}(\infty,\lambda)$ are related by an integral affine transformation.
    That this symmetry no longer holds true for $p\geq 3$ will be clear from the proofs of \cref{th:nonsqueezing_1} and \cref{th:nonsqueezing_2} and is carried out explicitly in \cref{re:symetry_cylinders}.
\end{remark}

\begin{theorem}\label{th:nonsqueezing_1}
    If there exists a symplectic embedding $B_{p,q}(\alpha) \xhookrightarrow{s} E_{p,q}(\lambda,\infty)$, then $\alpha \leq \lambda$.
\end{theorem}

\begin{remark}
    It should be mentioned that the proof here follows, in spirit, the strategy of \cite[Theorem 1.5.2]{ABEHS25} rather than the one in \cite[Theorem D]{AdHa25}.
    Namely, it uses \cref{th:symp}---the fact that Lagrangian pinwheels are unique up to symplectomorphism in their neighbourhoods---and then employs visible curves to obstruct the symplectic embedding, using techniques akin to those used in \cite{McDSie25:singalg,McDSie25:superpot}.
    The main difference is that in \cite{AdHa25}, the obstruction arises in a more ad hoc way, by finding a well-chosen exceptional curve in the compactification $\widetilde{X}_{p,q}$.
    The main observation of \cite{ABEHS25} is that the exceptional curve that provides the obstruction naturally appears as a component of a broken fibre of a regulation of $\widetilde{X}_{p,q}$.
    Thus we avoid the complicated combinatorics of \cite{AdHa25}.
\end{remark}

\begin{proof}[Proof of \cref{th:nonsqueezing_1}]
    Assume that there is a symplectic embedding $\iota: B_{p,q}(\alpha) \hookrightarrow E_{p,q}(\lambda,\infty)$.
    By \cref{cor:uniqueness_star_shaped} we can assume that $\iota$ and the visible embedding $\iota_\vis$ coincide on a small neighbourhood of $L^\std_{p,q}$, i.e.\ on $B_{p,q}(\varepsilon)$ for some small $0<\varepsilon$.
    Now, view $\iota$ as an embedding into a "non-minimal" compactification as shown in \cref{fig:nonsqueezingpictorially} (a).\footnote{This compactification is exactly the one that was discussed in Footnote \ref{foot:compactification}.}
    \begin{figure}[ht]
        \centering
        \begin{tikzpicture}
        \begin{scope}[shift={(-1,1)}]
            \node at (-0.7,1.3) {(b)};
            \fill[opacity=0.2] 
                (0,1.5) -- (0,0) -- (0.4,0)
                -- ++ (0.2,0.01)
                -- ++ (3,0.3)
                -- ++ (2,0.4) 
                -- ++ (1.9,0.49)
                -- ++ (1,0.3);
            \draw[thick]
                (0,0) node{\tiny \(\bullet\)} 
                -- (0.4,0) node{\tiny \(\bullet\)}
                -- ++ (0.2,0.01) node{\tiny \(\bullet\)}
                -- ++ (3,0.3) node{\tiny \(\bullet\)} 
                -- ++ (2,0.4) node{\tiny \(\bullet\)}
                -- ++ (1.9,0.49) node{\tiny \(\bullet\)}
                -- ++ (1,0.3);
            \draw[thick, mid arrow] 
                (0.6,0.01) -- node[below, text=black]{\tiny $\pmat{p^2-(pq+1) \\ pq-(q^2+1)}$} node[above]{\footnotesize $C_m$}
                (3.6,0.31);
            \draw[thick, mid arrow]
                (3.6,0.31) -- node[below, text=black]{\tiny $\pmat{p^2 \\ pq-1}$} node[above]{\footnotesize $E_t$}
                (5.6,0.71);
            \draw[thick, mid arrow]
                (5.6,0.71) -- node[below, text=black]{\tiny $\pmat{pq+1 \\ q^2}$} node[above]{\footnotesize $D_{n}$}
                (7.5,1.2);
            \draw[thick]
                (0,1.5) -- (0,0);
        \end{scope}

        \begin{scope}[shift={(-1,-2.5)}]
            \node at (-0.7,8.5/5-0.3) {(c)};
            \draw[thick, mid arrow] 
                (0,0.5) -- node[above left, text=black, fill=white,rounded corners=2pt,inner sep=1pt]{\tiny $\pmat{p^2-(pq+1) \\ pq-(q^2+1)}$}
                (8.5,8.5/5);
            \fill[opacity=0.2] (0,1) -- (0,0) -- (8.5,8.5/5);
            \draw[thick] (0,1) -- (0,0) -- (8.5,8.5/5);
            \draw[dashed] (0.8,0.4) node[cross] {} -- (0,0);
            \draw[thick,dash pattern=on 7pt off 3pt] (0,1) -- (0.25*8.5,1+0.25*3.5/5);
            \draw[thick,dash pattern=on 7pt off 3pt] (0.5*8.5,1+0.5*3.5/5) -- (8.5,8.5/5);
            \node at (6,0.5) {\footnotesize $\mathfrak{A}_{p,q}(t\alpha,t\alpha)$};
        \end{scope}

        \begin{scope}[shift={(-8,-2.7)}]
            \node at (-0.7,5.3) {(a)};
            \fill[opacity=0.2] 
                (0,5.6) -- (0,0) -- (3.6,0.9) -- (3.6,5.6) -- (0,5.6);
            \fill[opacity=0.2] 
                (0,5) -- (0,0) -- (4,1) -- ++ (0.1,0.04) -- ++ (0.08,0.08) -- ++ (0.06,0.12) -- ++ (0.04,0.18) -- ++ (0.02,0.3) -- (4.3,5) -- cycle;
            \draw[thick] 
                    (0,5) -- (0,0) -- (4,1);
            \draw[very thick]
                    (4,1) node{\tiny \(\bullet\)}
                    -- ++ (0.1,0.04) node{\tiny \(\bullet\)}
                    -- ++ (0.08,0.08) node{\tiny \(\bullet\)}
                    -- ++ (0.06,0.12) node{\tiny \(\bullet\)}
                    -- ++ (0.04,0.18) node{\tiny \(\bullet\)}
                    -- ++ (0.02,0.3) node{\tiny \(\bullet\)}
                    -- (4.3,5) node{\tiny \(\bullet\)}
                    -- (0,5) node{\tiny \(\bullet\)};
            \fill[opacity=0.2]
                      (0,0) -- (0,1)
                      .. controls +(1.0,0) and +(-0.6,0.6) ..
                      (2,3)
                      .. controls +(0.6,-0.6) and +(-0.6,-0.6) ..
                      (2,1)
                      .. controls +(0.6,0.6) and +(-1.0,0) ..
                      (4*0.8,0.8) -- cycle;
            \draw[dashed] (0.8,0.4) node[cross] {} -- (0,0);
            \node (big) at (5,0.4)  {\tiny $((\lambda+\delta) p^2,(\lambda+\delta) (pq-1))$};
            \node (small) at (2.5,-0.2) {\tiny $(\lambda p^2,\lambda (pq-1))$};
            \draw[->] (big) to[out=100, in=-30] (4,0.9);
            \draw[->] (small) to[out=80, in=220] (3.55,0.8);
            \node at (4.8,3) {\footnotesize $\mathcal{D}_{p,q}$};
            \node at (1,5.3) {\footnotesize $E_{p,q}(\lambda,\infty)$};
            \node at (1,2.4) {\footnotesize $\iota\big(B_{p,q}(\alpha)\big)$};
        \end{scope}
        \end{tikzpicture}
        \caption{(a) The almost toric base diagram of the compactification used in the proof of \cref{th:nonsqueezing_1}. The darker shaded region indicates how the embedded pin-ball $\iota\big(B_{p,q}(\alpha)\big)$ sits inside the compactification after the initial adjustment. (b) After rationally blowing up the ball $B_{p,q}(t\alpha)$ along the embedding $\iota$ for $t\alpha \leq \varepsilon$, which means precisely along $\iota_\vis$, the Wahl chain and the compactifying divisor are connected by a visible exceptional divisor on the toric boundary. The vectors can easily be computed from \eqref{eq:HJ_facts_matrices}. (c) The data used for the rational blow-up in the proof of \cref{th:nonsqueezing_1}.}
        \label{fig:nonsqueezingpictorially}
    \end{figure}
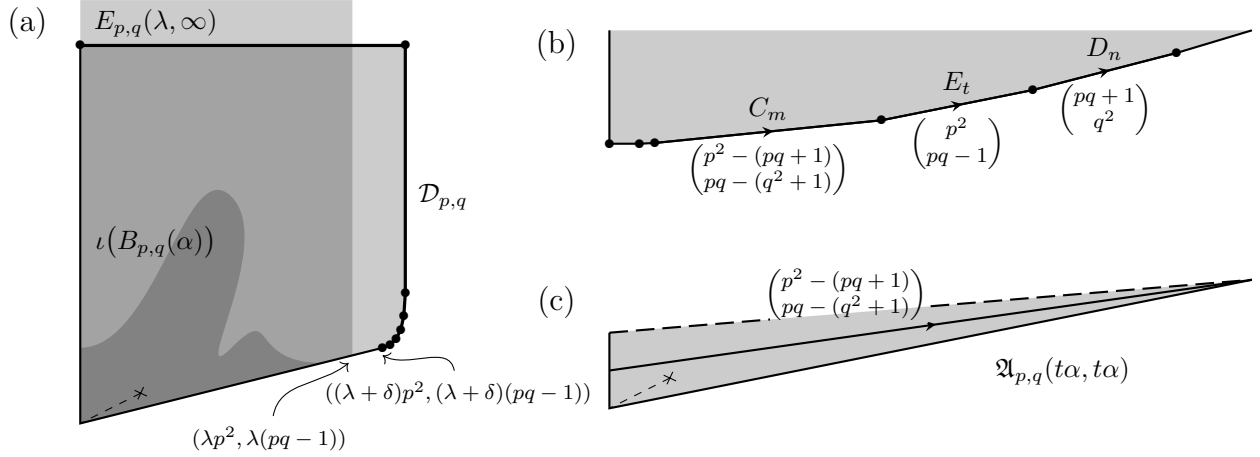
    Rationally blowing up the balls $B_{p,q}(t\alpha)$ for $t\in (0,1]$ along the embedding $\iota$ yields a family of symplectic manifolds $\big(\widetilde{X}_{p,q}(t),\widetilde{\omega}(t)\big)$.
    We choose the rational blow-up data such that the last sphere in the Wahl chain $C_m$ has maximal symplectic area.
    For the sake of clarity of the proof, we assume that all the other spheres have vanishing symplectic area. 
    It is not hard to see that this setup results in the following:
    \begin{enumerate}[label={(\roman*)}]
        \item\label{area}
            The symplectic area of $C_m$ is
            $$\int_{C_m}\widetilde{\omega}(t)= \frac{t\alpha p^2}{p^2-(pq+1)},$$
            as shown in \cref{fig:nonsqueezingpictorially} (c).\footnote{A computation in toric geometry relying on \eqref{eq:HJ_facts_matrices} shows that the direction associated to $C_m$ in the minimal resolution is given by $(p^2-(pq+1),pq-(q^2+1))$. Then the length of the longest segment that can fit into $\mathfrak{A}_{p,q}(t\alpha,t\alpha)$ is given by the term in the equation. See also \cite[Appendix B]{AdHa25} for the $(p,q)=(n,1)$ case.}
        \item 
            For $t\in (0,1]$ such that $t\alpha \leq \varepsilon$, the rational blow-up is visible and the piece of the toric boundary connecting the end of the Wahl chain $\mathcal{C}_{p,q}$ to the end of the compactifying divisor $\mathcal{D}_{p,q}$ is a symplectically embedded $(-1)$-sphere $E_t$. This sphere
            $E_t$ intersects $\mathcal{C}_{p,q}$ in exactly one point on $C_m$ and similarly for $\mathcal{D}_{p,q}$.
            See \cref{fig:nonsqueezingpictorially} (b).
    \end{enumerate}
    As in \cref{sec: Uniq up to symp}, we write 
    $$
    E_t=\sum a_i D_i + \sum c_j C_j \in H_2\big(\widetilde{X}_{p,q}(t),\mathbb{Q}\big).
    $$
    It is well known that $(-1)$-spheres are stable under deformations of the symplectic form as explained in \cite[Chapter 5]{Wen18:Jhollow}.\footnote{In the proof of \cite[Theorem 1.5.2]{ABEHS25} this was justified by invoking spherical Gromov invariants.}
    This means that we have for all $t \in (0,1]$
    $$
    0<\int_{E_t} \widetilde{\omega}(t) = \sum a_i\int_{D_i} \widetilde{\omega}(t) + \sum c_j \int_{C_j}\widetilde{\omega}(t).
    $$
    Since the compactifying divisor is disjoint from the embedding, the sum involving the terms connected to $\mathcal{D}_{p,q}$ will be constant.
    Considering the limit $t\to 0$ we see that the sum is, in fact, equal to $\lambda+\delta$.
    Since we assumed that the symplectic area of all but the last sphere in the Wahl chain vanishes, the inequality reads
    $$
    0<(\lambda +\delta) + c_m \int_{C_m}\widetilde{\omega}(t).
    $$
    We can calculate the coefficient $c_m$ via the equation $\boldsymbol{c}=M_{\mathcal{C}}^{-1}\boldsymbol{e}_m$, where $\bm{e}_m$ is the intersection profile of $E_t$ with $\cc_{p,q}$ mentioned in (ii), i.e.\ the $m$th canonical basis vector.
    From \eqref{eq:extremal_acomp} we deduce that 
    $$c_m=-\frac{1}{p^2}[pq-1]^{-1}=-\frac{1}{p^2}(p^2-(pq+1)),$$
    which follows from \cref{lma:M_inverse} because $(M_{\mathcal{C}}^{-1})_{mm}=-(e_mf_m)/p^2=-e_m/p^2$. 
    Using (i) the inequality therefore reads
    $$
    t\alpha\leq \lambda + \delta,
    $$
    where $\delta > 0$ can be chosen arbitrarily small.
    This means $\alpha\leq \lambda$, which is what we wanted to show.
\end{proof}

\begin{remark}
    The curve that is used to obstruct the embedding in this proof is exactly the one that was used in the proof of the non-squeezing theorem in the $(p,q)=(n,1)$-case by the first and third authors in \cite[Theorem D]{AdHa25}.
    In the proof above the geometric situation becomes much more transparent, because the uniqueness theorem for Lagrangian pinwheels is available.
\end{remark}

\begin{theorem}\label{th:nonsqueezing_2}
    If there exists a symplectic embedding $B_{p,q}(\alpha) \xhookrightarrow{s} E_{p,q}(\infty,\lambda)$, then $\alpha \leq \lambda$.
\end{theorem}

\begin{proof}
    The proof is analogous to the proof of \cref{th:nonsqueezing_1} and is therefore omitted.
\end{proof}

\begin{remark}\label{re:symetry_cylinders}
    Note that the proof of \cref{th:nonsqueezing_1} also shows that for $p \geq 3$, $E_{p,q}(\alpha,1)$ embeds into $E_{p,q}(1,\infty)$ if and only if $\alpha \leq 1$. Since $E_{p,q}(\alpha,1)$ embeds trivially into $E_{p,q}(\infty,1)$ for all $\alpha >1$, it follows that the two pin-cylinders cannot be symplectomorphic.
\end{remark}

Having established the non-squeezing theorem for pin-balls and recalling that in \cite[Theorem 1.5.2]{ABEHS25} it was proven that the set
$$
    \mathcal{A}_{p,q}:=\left\{(\alpha,\beta) \in \mathbb{R}^2_{>0} \hspace{0.1cm}\middle|\hspace{0.1cm} E_{p,q}(\alpha,\beta) \xhookrightarrow{s} \mathbb{C}P^2(1)\right\}
$$
is either empty or carries a symplectic staircase structure, a natural question emerges.

\begin{question}
    Do the symplectic problems of embedding pin-ellipsoids into pin-balls carry the structure of symplectic staircases?
\end{question}

In the forthcoming paper \cite{Ha25} this question will be answered in the positive for $(p,q)=(2,1)$ and the corresponding ellipsoid domains in $B_{2;1,1}\cong T^*S^2$, by showing that the set
$$
\mathcal{E}_{2,1}:=\left\{(\alpha,\beta) \in \mathbb{R}^2_{>0} \hspace{0.1cm}\middle|\hspace{0.1cm} E_{2,1}(\alpha,\beta) \xhookrightarrow{s} B_{2,1}(1)\right\}
$$
coincides with $\mathcal{A}_{2,1}$ and similarly for the $(2;1,1)$-case. 
These two sets are particularly interesting to compute since $E_{2,1}(\alpha,\beta)$ and $E_{2;1,1}(\alpha,\beta)$ are domains in $T^*\mathbb{R}P^2$ and $T^*S^2$, i.e.\ they carry inherent physical meaning and, moreover, they are easy to describe in terms of Finsler structures on the respective cotangent bundles.

\subsection{Eliashberg--Polterovich's theorem}

We illustrate the methods of this paper in a slightly different situation by reproving an old theorem of Eliashberg and Polterovich: 
\begin{theorem}[\normalfont{\cite[Theorem 1.1.A]{EliPol96}}]
\label{th:Eliashberg_Polterovich}
    Suppose that $\Delta \subseteq (\R^4,\omega_\std)$ is an embedded
    Lagrangian plane that coincides with a flat Lagrangian plane $\Delta_0$ at infinity.
    Then $\Delta$ and $\Delta_0$ are isotopic via a compactly supported Hamiltonian isotopy.
\end{theorem}

\begin{remark}
    Within the context of this paper, this theorem should be understood to be the $(p,q)=(1,1)$ case of \cref{th: nearby lagrangian for pinwheels}.
    The strategy of proof is very similar to that of \cref{th: nearby lagrangian for pinwheels}.
    The main difference from the rest of this paper is that a "$(1,1)$-pinwheel" is an embedded Lagrangian disc and therefore we have to carefully control the boundary behaviour of the embedded disc.
\end{remark}

\begin{remark}
    \cref{th:Eliashberg_Polterovich} is a foundational result in symplectic topology, as it shows that there are no "local" Lagrangian knots, i.e.\ there does not exist a local operation that knots a Lagrangian submanifold in dimension four.
    Eliashberg and Polterovich's proof uses the technique of filling by holomorphic discs.
    Since one cannot fill a Lagrangian plane by discs, they introduce an auxiliary object called a "plug",
    see also the survey \cite{PolSch24}.
    In our proof the auxiliary object will be a compactifying divisor, which allows us to use rational closed $J$-curves, i.e.\ regulations, in order to unknot the Lagrangian plane.
\end{remark}

The model of a flat Lagrangian plane that will make our proof most transparent is the Lagrangian plane that is given by the fixed locus of the anti-symplectic involution $\tau$ on $\mathbb{C}^2$ defined by $\tau(z_1,z_2) =(\overline{z}_2,\overline{z}_1)$.
Denote this fixed point set by $\Delta_\tau$.
Then 
$$\Delta_\tau=\{(z,\overline{z})\in \mathbb{C}^2\mid z\in \mathbb{C}\}$$
and in polar coordinates $z_j=r_j e^{i\theta_j}$ on $\mathbb{C}^2$ we have 
$$
\Delta_\tau=\{(z_1,z_2)\mid r_1=r_2 \text{ and } \theta_1=-\theta_2\},
$$
which means that $\Delta_\tau$ is the visible Lagrangian that projects to the diagonal ray under the standard toric moment map on $\mathbb{C}^2$.

\begin{proof}[Proof of \cref{th:Eliashberg_Polterovich}]
Without loss of generality we assume that $\Delta$ coincides with $\Delta_\tau$ outside $B^4(1-\varepsilon)$.
The proof of \cref{th:Eliashberg_Polterovich} proceeds in four steps.

\subsubsection{Compactifying Lagrangian planes}
\label{subsec:compactifying_planes}

The first step is to transfer the problem to a compact situation.
Performing a symplectic cut along the standard $S^3 \subseteq \mathbb{C}^2$, with respect to the standard diagonal Hamiltonian action, we obtain $(\mathbb{C}P^2,\omega_{\textsc{FS}})$.
Moreover, the Lagrangian planes $\Delta$ and $\Delta_\tau$ correspond to two embedded Lagrangian $\mathbb{R}P^2$s, which we denote by $L$ and $L_\tau$ in the following, that coincide in a neighbourhood of the line at infinity $\mathbb{C}P^1_{\infty}\subseteq \mathbb{C}P^2$.\footnote{Since the involution that defines $\Delta_\tau$ descends to $\mathbb{C}P^2$, it follows that the Lagrangians $L$ and $L_\tau$ are obtained by gluing Lagrangian discs to a Lagrangian Möbius strip along their boundaries.
This is also clear from the toric moment map, as illustrated in \cref{fig:Lagrangian planes}.}

\begin{figure}[htb]
  \centering
      \begin{tikzpicture}
        \begin{scope}[shift={(-4,0)}]
        \node at (-0.5,3.3) {(a)};
            \fill[opacity=0.2] (0,3.8) -- (0,0) -- (3.8,0) -- (3.8,3.8);
            \draw[thick] (0,3.8) -- (0,0) -- (3.8,0);
            \draw (0,3) -- (3,0);
            \node[text=red] at (2.8,3.3) {\footnotesize $\Delta_\tau$};
            \node at (0.5,3) {\footnotesize $S^3$};
            \draw[thick, blue]
                  plot[smooth,tension=0.8] coordinates 
                    {(1.4,1.4) (0.8,0.8) (1,0.5) (0.5,0.5) (0.3,2.4)};
            \draw[thick, red]
                (0,0) -- (3.8,3.8);
            \node[text=blue] at (0.65,2) {\footnotesize $\Delta$};
        \end{scope}

        \draw[->] (0.3,2.4) to[out=200,in=-40] (-1,2.8);

        \begin{scope}[shift={(0.5,2)}, scale=0.6]
            \fill[opacity=0.2] (0,2) -- (0,0) -- (2,0) -- (2,2);
            \draw[thick, red] (2,0) -- (0,2);
            \draw (0,2) -- (0,0) -- (2,0) -- (2,2) -- (0,2);
            \draw[->] (0,2) -- (1,2);
            \draw[->] (0,0) -- (1,0);
            \draw[->] (0,0) -- (0,1);
            \draw[->] (2,0) -- (2,1);
        \end{scope}

        \begin{scope}[shift={(3.5,0)}]
            \node at (-0.5,3.3) {(b)};
            \fill[opacity=0.2] (0,3) -- (0,0) -- (3,0);
            \fill[opacity=0.2] (0,3) -- (0,2.8) -- (2.8,0) -- (3,0);
            \fill[opacity=0.2] (0.3,3.8) -- (0.3,3.3) -- (3.3,0.3) -- (3.8,0.3) -- (3.8,3.8);
            \draw[thick] (0.3,3.8) -- (0.3,3.3) -- (3.3,0.3) -- (3.8,0.3);
            \draw[thick, red] (1.8,1.8) node[circlecross]{} -- (3.8,3.8);
            \draw[thick, blue]
                  plot[smooth,tension=0.8] coordinates 
                    {(1.5,1.5) (1.4,1.4) (0.8,0.8) (1,0.5) (0.5,0.5) (0.3,2.4)};
            \draw[thick, red] (1.5,1.5) -- (0,0);
            \node[text=red] at (0.54,0.18) {\footnotesize $L_\tau$};
            \draw[thick] (0,3) -- (0,0) -- (3,0) -- (0,3);
            \node[text=blue] at (0.63,1.8) {\footnotesize $L$};
            \draw[thick, red] (1.8,1.8) node[circlecross]{} -- (3.8,3.8);
        \end{scope}

        \begin{scope}[shift={(8.8,0)}]
            \node at (-0.5,3.3) {(c)};
            \fill[opacity=0.2] (0,3) -- (0,0) -- (3,0);
            \fill[opacity=0.2] (1,2) -- (2,1) -- (1.5,0.5) -- (0.5,1.5);
            \draw[thick] (0,3) -- (3,0);
            \draw[thick, blue]
                (0,0.1) to[out=45,in=225] (0.5,0.5) -- (1.5,1.5) node[circlecross]{};
            \node[text=blue] at (0.25,0.65) {\footnotesize $L$};
            \fill[opacity=0.2] (0.5,3.8) -- (3.8,0.5) -- (3.8,2) -- (2,3.8);
            \node[text=red] at (3.05,3.05) {\footnotesize $\gamma$};
            \node at (2.6,3.6) {\footnotesize $C_+$};
            \node at (3.65,2.55) {\footnotesize $C_-$};
            \draw
                (0.5,3.8) -- (3.8,0.5)
                (2,3.8) -- (3.8,2);
            \begin{scope}[rotate=-45]
                \fill[opacity=0.2]
                (0.7071,3.0406)
                arc (90:-90:-0.2 and -0.5303)
                -- (-0.7071,3.0406+2*0.5303)
                arc (-90:90:-0.2 and -0.5303)
                -- cycle;
            \end{scope}
            \begin{scope}[rotate=-45]
                \fill[opacity=0.2]
                (0.7071,3.0406)
                arc (90:-90:0.2 and -0.5303)
                -- (-0.7071,3.0406+2*0.5303)
                arc (-90:90:0.2 and -0.5303)
                -- cycle;
            \end{scope}
            \draw[rotate around={-45:(2.15,2.15)}, red]
                (2.15,2.15) arc (90:-90:0.2 and -2.1213/4);
            \draw[dashed, rotate around={-45:(2.15,2.15)}, red]
                (2.15,2.15) arc (90:-90:-0.2 and -2.1213/4);
            \draw[rotate around={-45:(2.65,1.65)}]
                (2.65,1.65) arc (90:-90:0.2 and -2.1213/4);
            \draw[rotate around={-45:(1.65,2.65)}] 
                (1.65,2.65) arc (90:-90:0.2 and -2.1213/4);
            \draw[dashed, rotate around={-45:(2.65,1.65)}]
                (2.65,1.65) arc (90:-90:-0.2 and -2.1213/4);
            \draw[dashed, rotate around={-45:(1.65,2.65)}]
                (1.65,2.65) arc (90:-90:-0.2 and -2.1213/4);
            \draw[->] (2,2) -- (1.7,1.7);
            \draw[thick, red]
                (0,0) -- (1.5,1.5) node[circlecross]{};
            \node[text=red] at (0.65,0.25) {\footnotesize $L_\tau$};
        \end{scope}
    \end{tikzpicture}
    \caption{(a) Schematic illustration of the initial situation of \cref{th:Eliashberg_Polterovich}. The plane $\Delta$ coincides with $\Delta_\tau$ outside $B^4(1-\varepsilon)$. (b) The situation after the symplectic cut as described in \cref{subsec:compactifying_planes}. The Lagrangian $\mathbb{R}P^2$s $L$ and $L_\tau$ coincide in a small neighbourhood of the line at infinity $\mathbb{C}P^1_{\infty}\subseteq \mathbb{C}P^2$, which is shaded darker in the figure. On the complementary piece obtained by the symplectic cut we obtain two Lagrangian Möbius strips that coincide. (c) The explicit Weinstein neighbourhood of the embedded, shared, Möbius strip shaded in darker gray and the configuration in $\mathbb{C}P^1_\infty$, which projects to the toric boundary.}
  \label{fig:Lagrangian planes}
\end{figure}
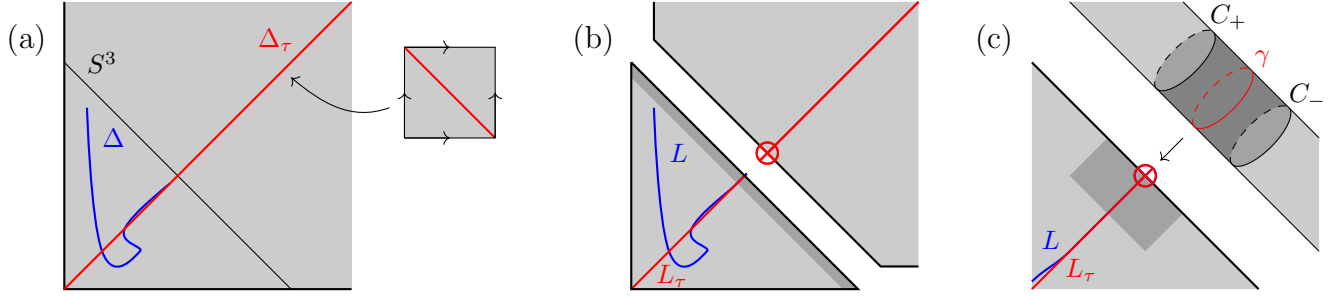

\subsubsection{Constructing a symplectomorphism of the complements of the Lagrangian $\mathbb{R}P^2$s in $\mathbb{C}P^2$}
\label{subsec:Compactly_supported_symp}

Recall that the rational blow-up of a Lagrangian $\mathbb{R}P^2$ in a symplectic manifold $(X,\omega)$ can be viewed as an operation induced by a symplectic cut.
Pick a Weinstein neighbourhood of the Lagrangian $\mathbb{R}P^2$, which, in particular, defines a symplectic embedding of the unit cotangent disc bundle $D^*_\varepsilon\mathbb{R}P^2$ for some small $\varepsilon >0$, where $D^*_\varepsilon\mathbb{R}P^2$ is defined with respect to the round metric on $\mathbb{R}P^2$.
Then a symplectic cut can be performed along the periodic geodesic flow on $\partial D^*_\varepsilon\mathbb{R}P^2:=S^*_\varepsilon\mathbb{R}P^2$, since it is induced by a Hamiltonian function.
See \cite[Section 2.2.1]{BoLiWu14} for details on this construction and \cite{Aud07, Ler95} for the original references.

In our situation we can investigate the effect of rationally blowing up $L$ and $L_\tau$ on $\mathbb{C}P^2$ explicitly.
Note that since $L$ and $L_\tau$ coincide in a neighbourhood of $\mathbb{C}P^1_\infty$ the effect of the rational blow-up on $\mathbb{C}P^1_\infty$ will be the same if the Weinstein neighbourhoods of $L$ and $L_\tau$ coincide near $\mathbb{C}P^1_\infty$.
Hence, we will consider $L$ for the moment.

\begin{lemma}
\label{lma:realRP2_blow_up}
    $L$ can be rationally blown up in such a way that the rational blow-up $(X,\omega)$ contains a chain of embedded symplectic spheres $\ct=(F^+,\Sigma,F_-)$ of self-intersection $[F^\pm]^2=0$ and $[\Sigma]^2=-4$, where $F^+ \cup F^-$ is the "proper transform" of the line $\mathbb{C}P^1_\infty$.
\end{lemma}

\begin{proof}
    Consider the sphere $S^2\subseteq \mathbb{R}^3$, equipped with the round metric $g$, and parametrize the equator by $\gamma(\theta)=(\cos(\theta),\sin(\theta),0)$.
    Let $s$ be the signed geodesic distance from the equator.
    Then
    $$\gamma_s(\theta)=(\cos{s}\cos{\theta},\cos{s}\sin{\theta},\sin{s})$$
    is the displacement of $\gamma(\theta)$ by distance $s$ along the normal geodesic to the equator.
    The round metric in these coordinates is given by $g=ds^2+\cos^2(s)d\theta^2$. 
    Taking the quotient $S^2/{\sim}=\mathbb{R}P^2$ by the antipodal map we find a tubular neighbourhood of the equator:
    $$
        M_\rho:=(S^1 \times (-\rho,\rho) )/{\sim}\quad\text{where }(\theta,s) \sim (\theta+\pi,-s) \text{ and } \rho<\frac{\pi}{2},
    $$
    which is a Möbius strip.
    The metric $g$ descends to the quotient because it is invariant under the equivalence relation. 
    Consider the cotangent bundle of the Möbius strip $T^*M_\rho$ with coordinates $(s,\theta,p_s,p_\theta)$.
    Then taking the dual metric on $T^*M_\rho$, the codisc bundle of radius $\varepsilon>0$ and its boundary are given by
    $$
        D^*_{\varepsilon}M_\rho=\left\{p_s^2+\frac{p_\theta^2}{\cos^2(s)}\leq \varepsilon^2\right\} \quad\text{and}\quad S^*_{\varepsilon}M_\rho=\left\{p_s^2+\frac{p_\theta^2}{\cos^2(s)}= \varepsilon^2\right\}.
    $$
    Near the equator $\gamma$ of $\mathbb{C}P_\infty^1$, which is the core circle of $L$, we have coordinates of the form $T^*S^1_{(\tau,p_\tau)} \times \mathbb{C}_{z=u+iv}$, in which the symplectic form $\omega_\text{FS}$ reads $dp_\tau\wedge d\tau \oplus du\wedge dv$, such that $\gamma=\{p_\tau=z=0\}$ and the part of $\mathbb{C}P_\infty^1$ contained in this chart is given by $\{z=0\}$.
    These coordinates come directly from the toric geometry as shown in \cref{fig:Lagrangian planes} (c).
    In these coordinates define the embedding $\iota:D^*_\varepsilon M_\rho \hookrightarrow \mathbb{C}P^2$ via
    \begin{equation}
        \label{eq:embedding_unitcodisc_moebius}
        \iota\left([\theta,s,p_\theta,p_s]\right):=\left(2\theta,\frac{1}{4}(s^2+p_s^2) + \frac{1}{2}p_\theta,e^{-i\theta}(s-ip_s)\right),
    \end{equation}
    which is easily seen to be well-defined and symplectic. 
    Under the usual moment map on $T^*S^1 \times \mathbb{C}$, given by $\mu(\tau,p_\tau,z):=(p_\tau,\lvert z \rvert^2/2)$, the codisc bundle $D_\varepsilon^*M_\rho$ projects to a thin strip around the ray $\{(\lambda,2\lambda)\mid \lambda \in \mathbb{R}_{\geq 0}\}\subseteq \R \times \mathbb{R}_{\geq 0}$.\footnote{This discussion is completely analogous to the immersion of Lagrangian cylinders in \cite[Section 5.3]{Ev23:Book}.}
    The intersection of $\mathbb{C}P^1_\infty$ with $\iota(S^*_\varepsilon M)$ in this chart is given by
    \begin{equation}
    \label{eq:int_Linfinity_WeinsteinM}
    \iota(S^*_\varepsilon M)\cap \mathbb{C}P^1_\infty=\{z=0,\quad p_\tau=\pm \varepsilon/2\}
    \end{equation}
    meaning that the intersection $\iota(S^*_\varepsilon M_\rho)\cap \mathbb{C}P^1_\infty$ consists of the images of the two circles
    $$
    C_{\pm}:=\{s=0,p_s=0,p_\theta=\pm \varepsilon\} \subseteq S^*_\varepsilon M_{\rho},
    $$
    see \cref{fig:Lagrangian planes} (c).
    The geodesic flow on $S^*_\varepsilon M_\rho$ is induced by the Hamiltonian flow of the standard quadratic Hamiltonian $H=\frac{1}{2}\lVert \cdot \rVert^2$ and the geodesic equations are easily computed from the Hamiltonian equations:
    \begin{equation}
    \label{eq:geodesic_eq}
    \Dot{\theta}=\frac{p_\theta}{\cos^2(s)},\qquad \Dot{s}=p_s,\qquad \Dot{p}_\theta=0,\qquad \Dot{p}_s=-p_\theta^2\frac{\sin(s)}{\cos^3(s)}. 
    \end{equation}
    This implies that $C_{\pm}$ are genuine geodesics, because the geodesic equations \eqref{eq:geodesic_eq} on $C_\pm$ are $\Dot{\theta}=\pm \varepsilon, \Dot{s}=0, \Dot{p}_\theta=0, \Dot{p}_s=0$, i.e.\ the geodesic flow just rotates $C_{\pm}$.
    
    Now, consider a symplectic embedding $\psi:D^*_\varepsilon \mathbb{R}P^2 \hookrightarrow \mathbb{C}P^2$ obtained by restricting a Weinstein neighbourhood of $L \subseteq \mathbb{C}P^2$.
    Choosing $\varepsilon >0$ and $\rho >0$ small enough we can assume that $\psi|_{D^*_\varepsilon M_\rho}$ coincides with $\iota$, defined in \eqref{eq:embedding_unitcodisc_moebius}, and that $\psi|_{D^*_\varepsilon \mathbb{R}P^2\setminus D^*_\varepsilon M_\rho}$ is disjoint from $\mathbb{C}P_\infty^1$.
    Denote the two symplectic discs in $\mathbb{C}P^1_\infty$ that form the complement of the annulus $\mathbb{C}P^1_\infty \cap \psi(D^*_\varepsilon M_\rho)$ by $D^\pm$, as shown in \cref{fig:Lagrangian planes}.
    Perform the rational blow-up of $L$ by symplectically cutting along the embedded copy of $S^*_\varepsilon \mathbb{R}P^2$ and denote the resulting symplectic manifold by $(X,\omega)$. 
    Since the boundaries of $D^\pm$ are geodesics, these two discs are two embedded symplectic spheres in $(X,\omega)$, which we denote by $F^\pm$.
    
    The goal is now to compute the self-intersection number of these spheres.
    The computation is completely symmetric for the two spheres, and hence we carry it out for $F^+$.
    In order to do so we compute the Euler number of its normal bundle.
    The normal bundle $N_{\mathbb{C}P^1_\infty/ \mathbb{C}P^2}$ is trivialized over a neighbourhood of $C^+$ by $\iota$, and we can define the local section $\partial_u=\partial_{\mathfrak{R}(z)}$ in this trivialization.
    Moreover, because the normal bundle $N_{\mathbb{C}P^1_\infty/ \mathbb{C}P^2}$ is trivial over the disc $D_+$ we can extend this section to a section $\nu_+$ over $D^+$.
    Since the geodesic flow leaves the normal coordinate invariant, the section $\nu_+$ descends to give a nowhere vanishing section of $N_{F^+/ X}$, which implies:
    \begin{equation*}
        F^+\cdot F^+ = e\big(N_{F^+/ X}\big) =0.
    \qedhere
    \end{equation*}
\end{proof}

\begin{remark}
    \cref{lma:realRP2_blow_up} can also be proven in another way, by working in complex coordinates.
    Oakley and Usher \cite[Section 3]{OaUs16} construct an explicit Weinstein neighbourhood of $L_\tau$ and a Hamiltonian torus action on $\C P^2\backslash L_\tau$.\footnote{This construction was also described by Wu \cite[Section 3]{Wu15}.}
    This construction then allows one to perform the rational blow-up along $L_\tau$ in a visible manner.
    Since $L$ and $L_\tau$ coincide near the line at infinity, this will yield an alternative proof of \cref{lma:realRP2_blow_up}.
    
    We view $T^*\R P^2$ as the quotient of $\{(p,q)\in \R^3\oplus \R^3 \mid |q|=1,\langle p,q\rangle=0\}$ by the usual $\Z_2$-action and denote by $D_1^* \R P^2$ the unit codisc bundle.
    We view $\C P^2$ as the symplectic reduction of the standard sphere with radius $\sqrt{2}$, i.e.\ $S^5(\sqrt{2})\subset \C^3$, and thus we understand elements of $\C P^2$ as tuples $[z_1:z_2:z_3]$ satisfying $|z_1|^2+|z_2|^2+|z_3|^2=2$, up to multiplication by $e^{i\theta}$.

    Oakley and Usher construct an explicit Weinstein neighbourhood of $L_\tau$ via a map $\varphi: D_1^*\R P^2\rightarrow \C P^2$, which is a symplectomorphism onto its image $\C P^2\backslash Q$, where $Q$ is the Fermat quadric $\{z_1^2+z_2^2+z_3^2=0\}\subset \C P^2$.
    Moreover, there is an explicit Hamiltonian torus action on $\C P^2\backslash L_\tau$ given by
    \begin{equation}
    \label{eq:Hamiltonian_action}
        G([z_1,z_2,z_3]):=\text{Im}(\bar{z}_1z_2) \quad\text{and}\quad H([z_1,z_2,z_3]):=\frac{1}{4}\sqrt{4-\left|\sum z_i^2\right|^2}.
    \end{equation}
    The moment image is shown in \cref{fig:delzant-transform}.
    \begin{figure}[ht]
        \centering
        \begin{tikzpicture}[scale=1.2]
        \begin{scope}
            \coordinate (A) at (-2,1);
            \coordinate (B) at (2,1);
            \coordinate (C) at (0,0);
            \fill[gray!15] (A) -- (B) -- (C) -- cycle;
            \draw (A) node{\tiny \(\bullet\)} -- (B) node{\tiny \(\bullet\)} -- (C) -- cycle;
            \draw[opacity=0.5, ->]
                (-2.5,0) -- (2.5,0);
            \draw[opacity=0.5, ->]
                (0,-0.5) -- (0,1.5);
            \draw[fill=white, line width=1pt] (C) circle (2pt);
            \node[left]  at (A) {\small $(-1,1/2)$};
            \node[right] at (B) {\small $(1,1/2)$};
            \node[below left]       at (C) {\small $(0,0)$};
            \node at (0.3,1.5) {$H$};
            \node at (2.8,0) {$G$};
        \end{scope}
        
        \begin{scope}[xshift=8cm]
            \coordinate (D) at (-2,1);
            \coordinate (E) at (-2,0);
            \coordinate (F) at (2,1);
            \fill[gray!15] (D) -- (E) -- (F) -- cycle;
            \draw (D) node{\tiny \(\bullet\)} -- (E) -- (F) node{\tiny \(\bullet\)} -- cycle;
            \draw[fill=white, line width=1pt] (E) circle (2pt);
        \end{scope}

        \draw[->]
        (3.0,0.5) .. controls (3.8,0.75) and (4.6,0.75) .. (5.4,0.5);
        
        \end{tikzpicture}
        
        \caption{On the left, the moment image of $(G,H)$, with the origin missing. Note that it is clear from \eqref{eq:Hamiltonian_action} that $L_\tau$ sits over the origin. On the right, the moment image after an affine $SL_{2}(\Z)$ transformation.}
        \label{fig:delzant-transform}
    \end{figure}
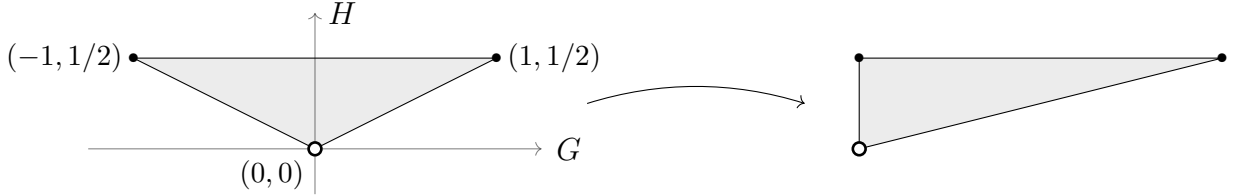
    The line at infinity $\C P^1_{\infty}$ is given by the equation $\{z_3=0\}$ and, since it has real coefficients, intersects $L_\tau$ in a circle (an algebraic line of $\R P^2$).
    Thus $\C P^1_\infty$ splits into two discs in $\C P^2\backslash L_\tau$, related by the involution $\tau$.
    It is straightforward to check that each of these discs projects to one of the slanted boundary segments of the moment image.

    The rational blow-up is then given by performing a symplectic cut along the visible hypersurface $H^{-1}(\alpha)$, for some $0<\alpha<1/2$, and the resulting moment image is the usual moment image of the Hirzebruch surface $\F_4$.
\end{remark}

To recapitulate: performing the rational blow-up along the embedding $\psi$, introduced in the proof of \cref{lma:realRP2_blow_up}, transforms $\mathbb{C}P^1_\infty$ into two embedded symplectic $(0)$-spheres $F^{\pm}$, which intersect the $(-4)$-sphere $\Sigma$, introduced by the rational blow-up, each in a single point, i.e.\ $X$ contains the linear chain of symplectic spheres $\mathcal{T}:=(F^+,\Sigma,F^-)$.
Since the rational blow-up preserves rationality, see for example \cite[Lemma 2.3]{BoLiWu14} or \cite[Corollary 3.10]{AdHa25}, we know that $(X,\omega)$ is diffeomorphic to $S^2 \times S^2$ and that $[F^\pm]\in H_2(X;\mathbb{Z})$ is a fibre.

Since the argument that $(X,\omega)$ is rational and ruled can be carried out quite easily, we give it for the sake of completeness.
Define $F:=[F^\pm] \in H_2(X;\mathbb{Z})$ and for $J\in \mathcal{J}(X,\omega)$ consider the moduli space $\mathcal{M}_{0,1}(F,J)$ of rational $J$-holomorphic curves representing $F$ with one marked point.\footnote{Note that a priori we only know that $[\Sigma]$ and $[F]$ form an integral basis of the free part of $H_2(X;\mathbb{Z})$ because of an argument involving the intersection matrix of $[\Sigma]$ and $[F]$ and using the fact that $b_2(X;\mathbb{Q})=2$. However, since $F^\pm$ are symplectic spheres, they cannot represent a torsion element in $H_2(X;\mathbb{Z})$, so the integer homology classes they represent coincide.}
Moreover, define the following subspace of the space of $\omega$-tame almost complex structures
$$
\mathcal{J}(\mathcal{T}):=\{J\in \mathcal{J}(X,\omega)\mid \text{each component of }\mathcal{T}\text{ is }J\text{-holomorphic}\}.
$$
The space $\mathcal{J}(\mathcal{T})$ is non-empty since $\mathcal{T}$ is a symplectic normal crossing divisor.
\begin{lemma}
    For every $J \in \mathcal{J}(\mathcal{T})$, the moduli space $\mathcal{M}_{0,1}(F,J)$ is a compact manifold of dimension $\dim(\mathcal{M}_{0,1}(F,J))=4$.
    Moreover, $(X,\omega)$ is regulated by the rational curves in $\mathcal{M}_{0,0}(F,J)$.
    In particular, the regulation has no broken fibre and $(X,\omega)$ is rational and ruled.
\end{lemma}

\begin{proof}
    Note that $F$ is a primitive homology class, because $\Sigma\cdot F=1$, which in particular implies that all curves in $\mathcal{M}_{0,1}(F,J)$ are simple.
    Therefore, it follows by automatic transversality, see for example \cite[Chapter 2]{Wen18:Jhollow} or \cite{HoLiSi97}, that $\mathcal{M}_{0,1}(F,J)$ is a smooth manifold of dimension $\dim(\mathcal{M}_{0,1}(F,J))=4$, because $c_1(F)=F\cdot F + 2 = 2$ and hence the index of a curve $u \in \mathcal{M}_{0,1}(F,J)$ is
    $$
    \text{ind}(u)=-2+2c_1([u])=2>-2.
    $$
    This implies that the evaluation map $\text{ev}$ has degree one.
    The statement about the compactness of $\mathcal{M}_{0,1}(F,J)$ follows directly from the fact that $F$ is primitive and from positivity of intersection.
    Moreover, $\Sigma$ defines a section of the regulation and so $(X,\omega)$ is a rational and ruled surface, diffeomorphic to $S^2\times S^2$.
\end{proof}

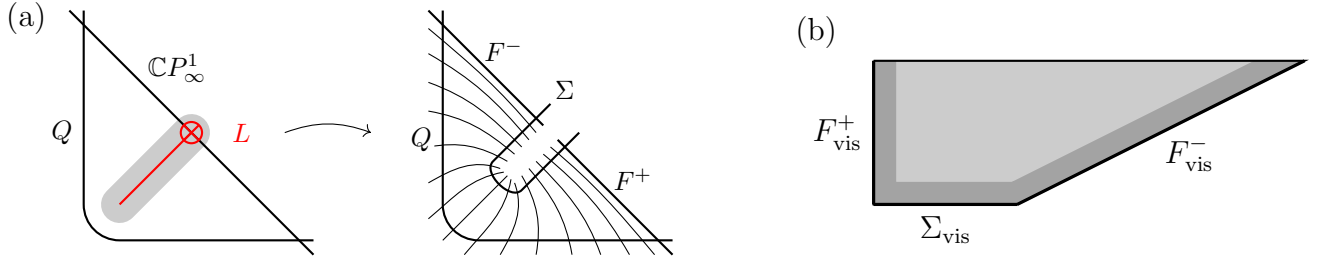
\begin{figure}[htb]
\begin{center}
    \begin{tikzpicture}[scale=0.95]
        \begin{scope}[shift={(-8,0)}]
            \node at (-0.8,3.1) {(a)};
            \draw[thick] (3.2,-0.2) to[out=135,in=-45] (-0.2,3.2);
            \node (L) at (1.3,2.4) {\footnotesize $\mathbb{C}P^1_{\infty}$};
            \draw[thick] (0,3.2) -- (0,0.5) to[out=-90,in=180] (0.5,0) --  (3.2,0);
            \node (Q) at (-0.3,1.5) {\footnotesize $Q$};
            \draw[opacity=0.2, line width=14pt, line cap=round]
                (1.5,1.5) -- (0.5,0.5);
            \draw[thick, red]
                (1.5,1.5) node[circlecross]{} -- (0.5,0.5);
            \node[text=red] at (2.2,1.5) {\footnotesize $L$};
        \end{scope}
    
        \begin{scope}[shift={(-3,0)}]
            \draw[thick] (3.2,-0.2) -- (1.6,1.4);
            \draw[thick] (-0.2,3.2) -- (1.4,1.6);
            \draw[thick] (1.3+0.2,1.7+0.2) -- (1.3-0.6,1.7-0.6) to[out=225,in=225] (1.7-0.6,1.3-0.6) --  (1.7+0.2,1.3+0.2);
    
            \draw (-0.2,2.95) to[out=-40,in=130] (1.3,1.5);
            \draw (-0.2,2.6) to[out=-30,in=130] (1.2,1.4);
            \draw (-0.2,2.2) to[out=-20,in=135] (1.1,1.3);
            \draw (-0.2,1.8) to[out=-10,in=140] (1,1.2);
            \draw (-0.2,1.3) to[out=10,in=150] (0.9,1.1);
            \draw (-0.2,0.8) to[out=30,in=160] (0.8,1);
            \draw (-0.2,0.3) to[out=40,in=190] (0.85,0.95);
    
            \draw (0,0) -- (0.9,0.9);
            
            \draw (2.95,-0.2) to[out=130,in=320] (1.5,1.3);
            \draw (2.6,-0.2) to[out=120,in=320] (1.4,1.2);
            \draw (2.2,-0.2) to[out=110,in=315] (1.3,1.1);
            \draw (1.8,-0.2) to[out=100,in=310] (1.2,1);
            \draw (1.3,-0.2) to[out=80,in=300] (1.1,0.9);
            \draw (0.8,-0.2) to[out=60,in=290] (1,0.8);
            \draw (0.3,-0.2) to[out=50,in=260] (0.95,0.85);
    
            \node at (2.65,0.85) {\footnotesize $F^+$};
            \node at (0.85,2.65) {\footnotesize $F^-$};
            \draw[thick] (0,3.2) -- (0,0.5) to[out=-90,in=180] (0.5,0) --  (3.2,0);
            \node[fill=white,rounded corners=2pt,inner sep=1pt] at (-0.3,1.4) {\footnotesize $Q$};
            \node at (1.3+0.4,1.7+0.4) {\footnotesize $\Sigma$};
        \end{scope}

        \draw[->] (-5.2,1.5) to[out=20,in=160] (-4,1.5);

        \begin{scope}[shift={(3,0.5)}]
            \node at (-0.8,2.4) {(b)};
            \fill[opacity=0.2] 
                (0,2) -- (0,0) -- (2,0) -- (6,2) -- (0,2);
            \begin{scope}
                \clip (0,2) -- (0,0) -- (2,0) -- (6,2) -- (0,2);
                \draw[opacity=0.2, line width=17pt, line cap=round]
                    (0,2) -- (0,0) -- (2,0) -- (6,2);
            \end{scope}
            \draw[thick] 
                (0,2) -- (0,0) -- (2,0) -- (6,2) -- (0,2);
            \draw[very thick]
                (0,0) -- node[anchor=north] {$\Sigma_\vis$} (2,0);
            \draw[very thick]
                (2,0) -- node[anchor=north west, xshift=-3pt, yshift=3pt] {$F^-_\vis$} (6,2);
            \draw[very thick]
                (0,0) -- node[anchor=east] {$F^+_\vis$} (0,2);
        \end{scope}
    \end{tikzpicture}
    \caption{(a) A schematic illustration of the effect of the rational blow-up of~$L$. Shaded in gray is the Weinstein neighbourhood discussed in \cref{subsec:Compactly_supported_symp}. The "proper transform" of the line $\mathbb{C}P^1_\infty$ is given by the two fibres $F^+$ and $F^-$. (b) The reference model, namely the standard toric fibration on $\F_4$, with the configuration $\ct_\vis=(F^+_\vis,\Sigma_\vis,F^-_\vis)$ and the neighbourhood $\nu_\vis$ shaded in darker gray.}
    \label{fig:compactification_birational}
\end{center}
\end{figure}

\begin{remark}
    In \cref{fig:compactification_birational} we illustrate how the regulation of $(X,\omega)$ is related to the initial situation in $\mathbb{C}P^2$.
    While it is general knowledge that in the complement of an embedded Lagrangian $\mathbb{R}P^2$ in $\mathbb{C}P^2$ there is always a smooth quadric, as can be shown, for example, by neck-stretching, in our case this follows a posteriori: in a normal neighbourhood of the configuration $\mathcal{T}$ we find a positive section, of self-intersection $+4$, of the regulation in the class $F$, which therefore also exists in $\mathbb{C}P^2\setminus L$.
\end{remark}

\subsubsection{Constructing a compactly supported symplectomorphism of $\mathbb{C}^2$ that identifies the Lagrangian planes.}

By \cref{subsec:Compactly_supported_symp}, we find symplectic embeddings $\psi, \psi_\tau:D^*_\varepsilon \mathbb{R}P^2 \hookrightarrow \mathbb{C}P^2$, obtained by restricting Weinstein neighbourhoods of $L,L_\tau \subseteq \mathbb{C}P^2$, that coincide on $D_\varepsilon^*M_\rho$.
Performing the rational blow-up along the embedded copies of $S^*_\epsilon \mathbb{R}P^2$ yields two symplectic manifolds $(X,\omega)$ and $(X_\tau,\omega_\tau)$.

Moreover, these rational blow-ups contain linear chains of symplectic spheres $\mathcal{T}:=(F^+,\Sigma,F^-)$ and $\mathcal{T}_\tau:=(F_\tau^+,\Sigma_\tau,F_\tau^-)$.
By the symplectic neighbourhood theorem there are symplectic embeddings $\phi:\nu \hookrightarrow (\F_4,\omega_4)$ and $\phi_\tau:\nu_\tau \hookrightarrow (\F_4,\omega_4)$ of normal neighbourhoods of these chains, such that $\phi_\tau(\nu_\tau)=\nu_\vis$ and $\phi_\tau(\ct_\tau)=\ct_\vis$ and similarly for $\phi$, as shown in \cref{fig:compactification_birational}.
Here $(\F_4,\omega_4)$ denotes the $4$th Hirzebruch surface equipped with a symplectic form $\omega_4$ that realises the symplectic data of $\ct$ and $\ct_\tau$.
Since $X$ is rational and ruled, the complement $X\setminus \ct$ is a disc bundle over an annulus and hence minimal.
Therefore, \cref{th:McDSalalternative} yields a symplectomorphism
$$
\Phi:((X,\omega),\ct) \to ((\F_4,\omega_4),\ct_\vis),
$$
as in the proof of \cref{th:symp}, and similarly $\Phi_\tau$.
Using these symplectomorphisms we obtain a symplectomorphism $\Psi:((\mathbb{C}P^2,\omega_\text{FS}),L) \to ((\mathbb{C}P^2,\omega_\text{FS}),L_\tau)$ that fixes a neighbourhood of the line at infinity $\mathbb{C}P^1_\infty$ and hence lifts to a compactly supported symplectomorphism of $(\mathbb{C}^2,\omega_\text{st})$ that takes $\Delta$ to $\Delta_\tau$.

\subsubsection{Upgrading the symplectomorphism to a compactly supported Hamiltonian diffeomorphism}

To conclude the argument, it is now enough to observe that Gromov proved that $\text{Symp}_c(B^4,\omega_{\text{st}})$ is contractible.
In particular, this implies that the symplectomorphism $\Psi\in \text{Symp}_c(\mathbb{C}^2,\omega_\text{st})$ that takes $\Delta$ to $\Delta_\tau$ and is compactly supported in $B^4$, is actually a Hamiltonian diffeomorphism $\Psi^H \in \text{Ham}_c(\mathbb{C}^2,\omega_\text{st})$, with compact support in $B^4$.
It follows that $\Delta$ and $\Delta_\tau$ are isotopic via an ambient compactly supported Hamiltonian isotopy.
\end{proof}

\begin{remark}
    This compactification is not the only choice that makes this argument work.
    One could equally well have chosen a monotone $S^2 \times S^2$ as compactification.
    The Lagrangian planes then transform to embedded Lagrangian spheres in $S^2 \times S^2$ and the theorem follows along similar lines.
\end{remark}

\subsection{Lagrangian pinwheels in $B_{p,q}$}
The goal of this subsection is to show that apart from $(p,q)$- and $(p,p-q)$-pinwheels, no other Lagrangian pinwheel embeds into $B_{p,q}$, i.e.\ we prove the following theorem.
\begin{theorem}\label{th:other_pinwheels_in_Bpq}
    If $L_{m,n}\subseteq B_{p,q}$ with $0<n<m$ coprime is a Lagrangian pinwheel, then 
    $$(m,n) \in \{(p,q),(p,p-q)\}.$$
\end{theorem}

\begin{proof}
    By the neighbourhood theorem for Lagrangian pinwheels, proven in \cite[Section 3]{Kho13} or \cref{thm:pinwheels_locally_visible}, we find a symplectic embedding $\iota:B_{m,n}(\varepsilon) \hookrightarrow B_{p,q}$.
    In particular, we have that $\iota^*c_1(B_{p,q})=c_1(B_{m,n})$.
    Moreover, the Chern classes are primitive by \cite[Lemma 2.13]{EvSm18} and therefore we have $m \mid p$.\footnote{Recall that $H^2(B_{p,q},\Z) \cong \mathbb{Z}_p$.}
    We will now run the argument carried out in \cref{sec:trafo_regulation}.
    For the sake of brevity we will skip some steps in order to arrive at the important part of the argument quickly.
    The arguments are completely analogous to those in \cref{sec:trafo_regulation} and are therefore mainly a matter of bookkeeping.
    
    A few points are crucial to observe at this point: the compactification $X_{p,q}$ of $B_{p,q}$ is as before, i.e.\ the intersection matrix of the compactifying divisor $\mathcal{D}_{p,q}$, denoted by $M_\mathcal{D}$, has the same numerology as before, whereas the numerology of the intersection matrix $M_\mathcal{C}$, associated with the Wahl chain $\mathcal{C}_{m,n}$, is governed by the numerology derived from $(m,n)$.
    The fibre class $F$ also has non-trivial intersection number with $L_{m,n}$.\footnote{Assume that $[L_{m,n}]\cdot F=0$. Then, as explained in \cite[Appendix C]{AdHa25}, the class $F$ can be represented by an embedded submanifold disjoint from $L_{m,n}$.
    However, this implies that $\iota^*c_1(B_{p,q})=0$, which is a contradiction.}
    
    As before, we stretch the neck along the contact boundary of $\iota(B_{m,n}(\varepsilon))$ and consider the limit of curves in the class $F$.
    We again consider the component $C$ of the holomorphic limit building that intersects the divisor $D_0$ and rationally blow up $X_{p,q}$ along $\iota$ to obtain a symplectic manifold $\widetilde{X}$. 
    We denote the curve that corresponds to the component of the limit building that intersects the compactifying divisor in $D_0$ under this procedure by $\widetilde{C}$. 
    Expressing the canonical class $K$ of $\widetilde{X}$ and $\widetilde{C}$ in the basis defined by $\mathcal{D}_{p,q}$ and $\mathcal{C}_{m,n}$ we obtain, in analogy to \eqref{eq:KandC}:
    $$
    K=-F-\sum D_i+\sum k_jC_j,\qquad \widetilde{C}=\sum a_i D_i +\sum c_j C_j.
    $$
    Repeating verbatim the proof of \cref{lma:intersection_profile_CD} we see that $\widetilde{C}$ is embedded and a sphere of non-positive self-intersection.
    We can assume that $\widetilde{C}$ is a $(0)$-sphere.
    We want to investigate the adjunction formula for this curve in order to determine $(m,n)$ in terms of $(p,q)$.
    The adjunction formula for $\widetilde{C}$ reads, in analogy to \eqref{eq:adjunction_TilC}:
    $$
    1= \frac{p^2+1}{2p^2}- \frac{1}{2}\left(\bm{k}^T\bm{\chi} + \bm{\chi}^T M_\mathcal{C}^{-1} \bm{\chi}\right).
    $$
    As in the proof of \cref{lma:intersection_profil_CC} we find that $\widetilde{C}$ intersects the Wahl chain $\mathcal{C}_{m,n}$ precisely once, which means $\chi_j=1$ for exactly one $j$ in the notation from before. 
    Then the adjunction formula takes the form
    $$
    1=\frac{p^2+1}{2p^2} - \frac{1}{2}\left(-1+\frac{e_j+f_j}{m^2}-\frac{e_jf_j}{m^2}\right)
    $$
    which is, by writing $p=lm$ for a positive integer $l$, equivalent to
    $$
    0=\frac{p^2+1}{2p^2} +\frac{1}{2m^2}(m^2+e_j f_j-(e_j + f_j))-1=\frac{1}{2p^2}(1+l^2 e_j f_j-l^2(e_j + f_j))
    $$
    meaning
    $$
    0=\left(\frac{1}{l^2}-1+(e_j-1)(f_j-1)\right).
    $$
    This implies that $l^2=1$, i.e.\ $p=m$ because the term $(e_j-1)(f_j-1)$ is an integer, and therefore that $e_j=1$ or $f_j=1$, which means that $\widetilde{C}$ intersects $\mathcal{C}_{m,n}$ at one of its ends. 
    Because $\widetilde{C}$ is a $(0)$-curve we have
    $$
    0=a_0+\bm{\chi}^T M_\mathcal{C}^{-1} \bm{\chi}=\frac{pq-1}{p^2}-\frac{1}{p^2}e_jf_j,
    $$
    which implies $pq-1=e_jf_j$.
    Since $\widetilde{C}$ intersects the Wahl chain at one of its ends we have $e_jf_j\in \{mn-1,[mn-1]^{-1}\}=\{pn-1,[pn-1]^{-1}\}$ and because $[pn-1]^{-1}=p^2-(pn+1)=p(p-n)-1$ this means $n=q$ or $n=p-q$.
\end{proof}

\begin{corollary}
    If there is a symplectic embedding of a star-shaped subset of $B_{m,n}$ into $B_{p,q}$, then $B_{m,n}$ is symplectomorphic to $B_{p,q}$.
\end{corollary}

\begin{proof}
    The star-shaped subset contains an $(m,n)$-pinwheel; therefore, by Theorem \ref{th:other_pinwheels_in_Bpq}, $(m,n)\in\{(p,q),(p,p-q)\}$, and \cite[Remark 2.8]{ABEHS25} shows that $B_{p,q}$ and $B_{p,p-q}$ are symplectomorphic since their base diagrams are related by an integral affine transformation. 
\end{proof}

\begin{remark}
    Note that a weaker version of this result could be obtained without evoking pseudoholomorphic curve techniques, just by considering the soft obstructions of Lagrangian pinwheels, namely that $c_1(L_{m,n})=n \mod m$ and $L_{m,n}^2=-1 \mod m$ (see \cite[Remark 2.8]{AdHa25}).
    However, these considerations do not exclude the case of a symplectic embedding $B_{4,3}\hookrightarrow B_{8,1}$ for example. 
\end{remark}

\section{Appendix}

\subsection*{Pinwheels and their neighbourhoods}
In this appendix, we elaborate on some aspects regarding pinwheels and their neighbourhoods.
We introduce a weaker (perhaps more natural) definition of a Lagrangian pinwheel, show in \cref{lemma:minimal=good} that it is equivalent to the one given in the main text, and demonstrate how these admit a ``Weinstein'' neighbourhood symplectomorphic to $B_{p,q}$ after completion.
The neighbourhood theorem is originally due to Khodorovskiy \cite[Section 3]{Kho13}, and we improve upon it by showing that the symplectomorphism into $B_{p,q}$ takes the pinwheel to the visible one with respect to the standard ATF on $B_{p,q}$.

\begin{definition}
    Define a \textit{topological $p$-pinwheel $P_p$} to be the topological space $D^2/{\sim}$, where the equivalence relation is given by $z \sim z'$ if and only if $z,z'\in \partial D^2$ and $z=e^{2\pi i k/p}z'$ for some $k\in \mathbb{N}$.
\end{definition}

Recall that we are always assuming, as in the rest of this paper, that $0<q<p$ are two coprime integers.
We now define a Lagrangian $(p,q)$-pinwheel differently from \cref{def:pinwheel_prelim}.

\begin{definition}\label{def:pinwheel_equivariance}
    A \textit{Lagrangian $(p,q)$-pinwheel} in a symplectic 4-manifold $(X,\omega)$ is an immersion $f:D^2 \subseteq \mathbb{C} \looparrowright X$ such that
    \begin{enumerate}[label={\roman*)}]
        \item 
            The restriction of $f$ onto the interior of the disc $f|_{D^2\setminus \partial D^2}$ is a Lagrangian embedding and there exists an embedding $\gamma:\partial D^2 \hookrightarrow X$ such that $C:=f(\partial D^2)=\gamma(\partial D^2)$, which we will call the \textit{core circle};
        \item 
            The map $f_\partial:=f|_{ \partial D^2}$ factors through a continuous embedding of $P_p$;
        \item 
        Furthermore, let $\Lambda \to C$ be the $S^1$-bundle whose fibre over $x\in C$ consists of Lagrangian $2$-planes in $T_xX$ that contain $T_xC$.
    Choose a trivialization $\Phi:\Lambda \to C \times S^1$.
    Pulling this trivialization back by $f_\partial$ gives a trivialization of the pullback bundle $(f|_\partial)^*\Lambda \to \partial D^2$.
    The immersion defines a section $s:=f_*(TD^2)$ in $(f|_\partial)^*\Lambda$, which in the trivialization defines a map $\vartheta:S^1 \to S^1$. 
    We then require that this section satisfies
    \begin{equation}\label{eq:def_pinwheel_equivariance}
        \vartheta(t+\delta)-\vartheta(t)=q\delta,
    \end{equation}
    where $\delta:=\frac{2\pi}{p}$.
    \end{enumerate}
\end{definition}
\begin{remark}
    Before we move on to investigate this definition in more detail, let us elaborate on the definition itself and its connections to earlier definitions of Lagrangian pinwheels. 
    Both Khodorovskiy \cite[Definition 3.1]{Kho13} and Evans--Smith \cite[Definition 2.3]{EvSm18} give a definition of a Lagrangian $(p,q)$-pinwheel that exchanges the last part of our definition with the weaker condition that the relative winding number, which is only defined modulo $p$, of the flanges around the core is equal to $q$.
    \begin{enumerate}[label=\roman*)]
        \item 
            For the purposes of Khodorovskiy and Evans--Smith the weaker definition suffices because in an arbitrarily small $C^0$-neighbourhood of a pinwheel satisfying this weaker definition there is a $(p,q)$-pinwheel that admits a standard neighbourhood. 
            See \cite[Definition 3.3]{Kho13} and \cite[Definition 2.10]{EvSm18}.
        \item 
            For us, however, the weaker definition is insufficient: to show Hamiltonian uniqueness there is an obstruction coming from the configuration space of $p$-lines in the plane.
            Indeed, observe that, by definition, at a point on the core circle the immersion $f$ defines $p$ lines in the symplectic normal bundle.
            These lines cannot, in general, be taken to another configuration of $p$ lines in the symplectic normal bundle by a linear symplectic (or even just smooth) map.
            Hence, it is absolutely crucial to fix the geometric tangent data along the core circle, which is exactly the last part of our \cref{def:pinwheel_equivariance} (see also the proof of \cref{lemma:minimal=good}).
    \end{enumerate}
    Theorem \ref{th: nearby lagrangian for pinwheels} shows uniqueness up to Lagrangian isotopy in $B_{p,q}$ of these weakly defined pinwheels.
    This follows from our theorem together with the fact that weakly defined Lagrangian $(p,q)$-pinwheels are isotopic through weak Lagrangian $(p,q)$-pinwheels to a Lagrangian $(p,q)$-pinwheel, as just defined, and this isotopy can be localised around the core circle; cf. \cite[Lemma 3.3]{Kho13}.
\end{remark}
Let $f:D^2 \looparrowright X$ define a Lagrangian $(p,q)$-pinwheel in $(X,\omega)$ and $\gamma:\partial D^2 \hookrightarrow X$ the parametrization of the core circle given in the definition. 
Since $\gamma$ is isotropic we find coordinates around this core circle of the form $T^*S^1 \times \mathbb{C}$, with coordinates $(\tau,x)$ on $T^*S^1$, where $\tau$ is the base coordinate, and $z=u+iv$ on $\mathbb{C}$, such that the core circle is given by $S^1 \times \{0\}$.
Moreover, the symplectic form is identified with $d\tau \wedge dx + du \wedge dv$.
This means that, if we write $\Sigma:S^1 \times [0,\epsilon)\to T^*S^1 \times \mathbb{C}$ for $f$ in these coordinates, we can assume that
\begin{equation}\label{eq:local_expression_pinwheel}
    \Sigma(t,s)=(pt,x(t,s),z(t,s)).
\end{equation}
Note that there is a natural choice of Lagrangian immersion of this collar that satisfies the conditions on the collar part in \cref{def:pinwheel_equivariance}:
\begin{equation}\label{eq:pinwheel_good_parametrization}
    \Sigma^g(t,s)=\left(pt,\frac{q}{2p}s^2,s e^{iqt}\right)
\end{equation}
This immersion is exactly the immersion that naturally appears in toric geometry.
See for example \cite[Section 5.3]{Ev23:Book}.
This parametrization warrants a definition.

\begin{definition}[\normalfont{\cite[Lemma 3.3]{Kho13}}]
\label{def:good_pinwheel}
    A \textit{good Lagrangian $(p,q)$-pinwheel} is a Lagrangian immersion $f:D^2\looparrowright (X,\omega)$ such that we find coordinates around the core circle in which $f$ is expressed as $\Sigma^g$, defined in \eqref{eq:pinwheel_good_parametrization}.
\end{definition}

In the main text, we have defined Lagrangian $(p,q)$-pinwheels as good Lagrangian $(p,q)$-pinwheels; this is allowed by the following lemma.

\begin{lemma}\label{lemma:minimal=good}
    \cref{def:pinwheel_equivariance} and the \say{good} \cref{def:good_pinwheel} of Lagrangian $(p,q)$-pinwheels agree.
\end{lemma}

\begin{proof}
    Our first claim is that coordinates around the core of any Lagrangian $(p,q)$-pinwheel can be chosen so that it is a good Lagrangian $(p,q)$-pinwheel to first order.

    Consider the local expression $\Sigma(t,s)=(pt,x(t,s),z(t,s))$ obtained in \eqref{eq:local_expression_pinwheel}.
    Along the core circle we have $\Sigma(t,0)=(pt,0,0)$ and hence $\partial_t\Sigma(t,0)=(p,0,0)$.
    Moreover, 
    $$
    \partial_s\Sigma(t,0)=(0,\partial_sx(t,0),V(t)),
    \quad\text{where }
    V(t):=\partial_s z(t,0)\in\mathbb C.
    $$
    We first show that $\partial_sx(t,0)=0$.
    Indeed, since $\Sigma$ is Lagrangian,
    $$
        0=\omega(\partial_t\Sigma,\partial_s\Sigma)=p\partial_sx+du\wedge dv(\partial_tz,\partial_sz).
    $$
    Using
    $$
        du\wedge dv(X,Y) = \operatorname{Im}(\overline{dz(X)}dz(Y)),
    $$
    this gives
    $$
        p\partial_sx = -\operatorname{Im}(\overline{\partial_t z}\,\partial_s z).
    $$
    Restricting to $s=0$, the right hand side vanishes because $z(t,0)=0$, and hence $\partial_tz(t,0)=0$. Therefore,
    $$
        \partial_sx(t,0)=0.
    $$
    Thus $\partial_s\Sigma(t,0)=(0,0,V(t))$, where $V(t) \neq 0$. 
    Using the polar decomposition: $V(t)=\lambda(t) e^{i\psi(t)}$ for $\lambda(t)=\lvert V(t) \rvert>0$.
    If we apply the diffeomorphism of the domain that is given by scaling the radial directions by $1/ \lambda(t)$ and work in this new coordinate system we obtain $\partial_s \Sigma(t,0)=(0,0,e^{i\psi(t)})$.
    
    By the definition of a Lagrangian $(p,q)$-pinwheel, in particular \eqref{eq:def_pinwheel_equivariance}, we must have $\psi(t+\delta)=\psi(t)+q\delta$, where $\delta:=2\pi / p$.
    Define $\eta(t):=\psi(t)-qt$ and observe:
    $$
    \eta(t+\delta)=\psi(t+\delta)-q(t+\delta)=\psi(t)+q\delta-q(t+\delta)=\eta(t).
    $$
    This precisely means that $\eta$ descends to a smooth function under the projection, i.e.\ we can define $\eta_p(pt):=\eta(t)$.
    Now consider the symplectomorphism defined by
    \begin{equation}\label{eq:symplecto_straightening}
        \Phi_{\eta_p}(\tau,x,z)=\left(\tau,x-\frac{1}{2}\eta_p'(\tau)\lvert z \rvert^2,e^{-i\eta_p(\tau)}z\right).   
    \end{equation}
    That $\Phi_{\eta_p}$ indeed defines a symplectic map follows from
    $$
    d\tau\wedge d\left(x-\frac{1}{2}\eta_p'(\tau)\lvert z \rvert^2\right)=d\tau \wedge dx-\frac{1}{2}\eta_p'(\tau)d\tau \wedge d\lvert z \rvert^2
    $$
    and, defining $w:=e^{-i\eta_p(\tau)}z$,
    $$
    \frac{i}{2}dw \wedge d\overline{w}=\frac{i}{2}dz \wedge d\overline{z} + \frac{1}{2}\eta_p'(\tau) d\tau \wedge d \lvert z\rvert^2.
    $$
    The symplectomorphism $\Phi_{\eta_p}$ fixes the core pointwise and its differential rotates the complex normal direction by exactly $e^{-i\eta_p(\tau)}$.
    In particular, this means that if we apply the coordinate change $\Phi_{\eta_p}$ on the ambient coordinates and express $\Sigma$ in these new coordinates we obtain
    $$
    \partial_t\Sigma(t,0)=(p,0,0) \quad\text{and}\quad \partial_s\Sigma(t,0)=(0,0,e^{iqt}).
    $$
    So the tangent planes of $\Sigma^g$ and $\Sigma$ along the core circle coincide.
    To sum up: we are now working in a coordinate system in which $\Sigma$ agrees to first order with the model $\Sigma^g$ along the core circle.
    
    Our second claim is that now \textit{$\Sigma$ can be Hamiltonian isotoped to $\Sigma^g$.}
    Consider the $p$-fold cover of $S^1$ explicitly given by $\pi:S^1 \to S^1$, where $\pi(t)=pt$ and the $p$-fold covering this induces on $T^*S^1 \times \mathbb{C}$.
    Lifting the Lagrangian immersions $\Sigma$ and $\Sigma^g$ gives Lagrangian embeddings
    $$
    \widetilde{\Sigma}(t,s)=\left(t,x(t,s),z(t,s)\right) \quad\text{and}\quad \widetilde{\Sigma}^g(t,s)=\left(t,\frac{q}{2p}s^2,s e^{iqt}\right).
    $$
    Since $\widetilde{\Sigma}$ and $\widetilde{\Sigma}^g$ agree along the core to first order, the Weinstein neighbourhood theorem implies that, after possibly shrinking the collar, $\widetilde{\Sigma}$ can be written as the graph of a closed one-form $\alpha \in \Omega^1_{dR}(\widetilde{\Sigma}^g)$ on $\widetilde{\Sigma}^g$ that vanishes on the boundary $S^1 \times \{0\}$.
    Because $H^1_{dR}(S^1 \times [0,\epsilon),S^1\times \{0\})=0$ it follows that $\alpha$ is exact, meaning that there exists a smooth function $H$, vanishing on the boundary, such that $\alpha=dH$.
    
    Viewing $H$ as a function on the Weinstein neighbourhood of $\widetilde{\Sigma}^g$, the Hamiltonian flow of $H$ maps $\widetilde{\Sigma}^g$ to $\widetilde{\Sigma}$. 
    The crucial observation is now that the lifted core circle and the first order data along the core circle are invariant under the deck transformation.
    Consequently, the one-form $\alpha$ and its primitive $H$ can be chosen to be deck-invariant, which ensures that the Hamiltonian flow generated by $H$ is equivariant with respect to the deck transformation.
    Therefore, the Hamiltonian flow descends to the quotient, showing that in a neighbourhood of the core circle the immersion $\Sigma$ is Hamiltonian isotopic to $\Sigma^g$.
\end{proof}

\begin{figure}[htb]
  \centering
      \begin{tikzpicture}
        \begin{scope}[shift={(-8,0)}]
            \node at (-0.5,1.3) {(a)};
            \clip (-0.7,-1.5) rectangle (3,1.5);
            \fill[opacity=0.2] (0,1.5) -- (0,-1.5) -- (3,-1.5) -- (3,1.5);
            \draw[thick] (0,1.5) -- (0,-1.5);
            \draw[->]
                (1,-0.2) -- node[above, text=black]{\tiny $\pmat{p \\ q}$}
                (1.8,0.2);
            \draw[opacity=0.2, line width=14pt, line cap=round]
                (4,1) -- (2.4,0.2);
            \node[text=red] (L) at (-0.4,-0.52) {\footnotesize $L$};
            \draw[thick, red]
                (3,0.5) -- (0,-1) node[circlecross]{};
        \end{scope}
    
        \draw[->] (-4.7,-0.6) to[out=170,in=290] (-5.4,0.2);

        \begin{scope}[shift={(-4.5,-1.2)}, scale=0.6]
            \fill[opacity=0.2] (0,2) -- (0,0) -- (2,0) -- (2,2);
            \begin{scope}
                \clip (0,0) rectangle (2,2);
                \draw[opacity=0.2,line width=8pt]
                    (2.5,0) -- (1,3)
                    (2,-1) -- (0,3)
                    (1,-1) -- (-0.5,2);
            \end{scope}
            \draw[thick, red] (2,1) -- (1.5,2) (1.5,0) -- (0.5,2) (0.5,0) -- (0,1);
            \draw (0,2) -- (0,0) -- (2,0) -- (2,2) -- (0,2);
            \draw[->] (0,2) -- (1,2);
            \draw[->] (0,0) -- (1,0);
            \draw[->] (0,0) -- (0,1);
            \draw[->] (2,0) -- (2,1);
        \end{scope}

        \begin{scope}[shift={(-2,-1.5)}]
        \node at (-0.5,2.8) {(b)};
            \fill[opacity=0.2] (0,3) -- (0,0) -- (3,0);
            \draw[thick] (0,3) -- (0,0) -- (3,0);
            \draw[thick,dash pattern=on 7pt off 3pt] (0,3) -- (3,0);
            \node[text=red] at (1.6,2) {\footnotesize $D_{\text{st}}$};
            \begin{scope}
                \clip (0,3) -- (0,0) -- (3,0);
                \draw[opacity=0.2,line width=14pt]
                    (-1,-1) -- (4,4);
                \draw[opacity=0.2, line width=14pt, line cap=round]
                    (-1,-1) -- (0.1,0.1);
            \end{scope}
            \draw[thick, red]
                (0,0) -- (1.5,1.5);
        \end{scope}

        \draw[->] (0.6,-0.15) to[out=240,in=-30] (-0.6,-0.3);

        \begin{scope}[shift={(0.5,0)}, scale=0.6]
            \fill[opacity=0.2] (0,2) -- (0,0) -- (2,0) -- (2,2);
            \begin{scope}
                \clip (0,0) rectangle (2,2);
                \draw[opacity=0.2,line width=8pt]
                    (-1,3) -- (3,-1);
            \end{scope}
            \draw[thick, red] (2,0) -- (0,2);
            \draw (0,2) -- (0,0) -- (2,0) -- (2,2) -- (0,2);
            \draw[->] (0,2) -- (1,2);
            \draw[->] (0,0) -- (1,0);
            \draw[->] (0,0) -- (0,1);
            \draw[->] (2,0) -- (2,1);
        \end{scope}

        \begin{scope}[shift={(3,-1.5)}]
            \node at (-0.5,2.8) {(c)};
            \fill[opacity=0.2] (0,3) -- (0,0) -- (3,0);
            \draw[thick] (0,3) -- (0,0) -- (3,0);
            \draw[thick,dash pattern=on 7pt off 3pt] (0,3) -- (3,0);
            \node[text=red] at (1.6,2) {\footnotesize $D_{\text{Aur}}$};
            \begin{scope}
                \clip (0,3) -- (0,0) -- (3,0);
                \draw[opacity=0.2,line width=14pt]
                    (-1,-1) -- (4,4);
                \draw[opacity=0.2, line width=14pt, line cap=round]
                    (-1,-1) -- (0.1,0.1);
            \end{scope}
            \draw[thick, red]
                (0.15,0.15) -- (1.5,1.5);
            \draw[dashed] (0.15,0.15) node[cross]{} -- (0,0);
        \end{scope}
    \end{tikzpicture}
    \caption{(a) A neighbourhood $U_\Sigma$ of the core of a Lagrangian $(p,q)$-pinwheel $L$ admits a Lagrangian torus fibration as shown. Shaded in darker gray is a Weinstein neighbourhood of the annulus $\Gamma$. (b) The standard toric fibration of the ball and the visible Lagrangian disc $D_{\text{st}}$. A Weinstein neighbourhood of the annulus is shown in darker gray. Note that a neighbourhood of the centre of the disc contains a saturated neighbourhood of the origin, i.e.\ the fibre tori close to the origin are completely contained in the neighbourhood, while away from the origin the Weinstein neighbourhood is fibred by Lagrangian annuli. (c) The neighbourhood after a nodal trade, which yields a Weinstein neighbourhood of the disc $D_{\text{Aur}}$.}
  \label{fig:ATF_neigh_1}
\end{figure}
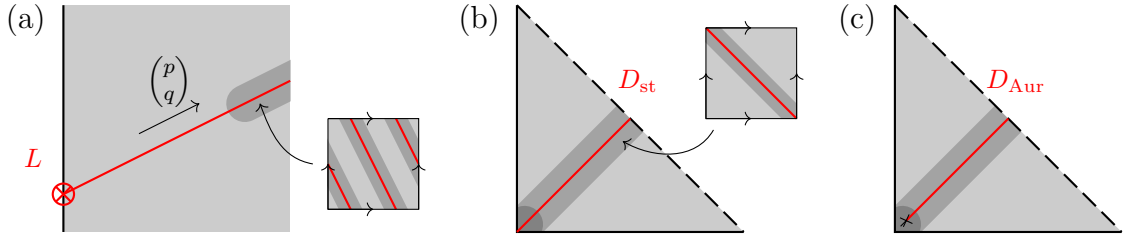

\begin{theorem}\label{thm:pinwheels_locally_visible}
        Suppose that $L_{p,q} \subseteq (X,\omega)$ is a Lagrangian pinwheel in a symplectic manifold. Then for some $\epsilon >0$ there exists a symplectic embedding $\psi: B_{p,q}(\epsilon) \hookrightarrow (X,\omega)$ such that $\psi(L_{p,q}^\text{vis})=L_{p,q}$, i.e.\ there exists a neighbourhood $U_L$ of $L_{p,q}$ that admits an ATF such that $L_{p,q}$ is a visible pinwheel.
\end{theorem}

\begin{proof}
The idea of the proof is to produce ATFs in the neighbourhoods of the core circle and the capping disc and glue them appropriately. 
By \cref{lemma:minimal=good} there is a neighbourhood $U_\Sigma$ of the core circle of $L_{p,q}$ that is symplectomorphic to $S^1 \times D^3$ such that $L_{p,q}$ is visible in this neighbourhood as shown in \cref{fig:ATF_neigh_1} (a). Consider the annulus $\Gamma$ defined in the figure (to be the gluing region); the ATF defines an embedding of a small cotangent disc bundle of an annulus such that the zero section is identified with $\Gamma$.
\begin{figure}[ht]
  \centering
      \begin{tikzpicture}
        \begin{scope}[shift={(-6,0)}]
            \node at (-0.5,1.7) {(a)};
            \begin{scope}
                \clip (0,-0.5) rectangle (4,2);
                \draw[opacity=0.2,line width=26pt, line cap=round]
                    (-0.4,-0.2) -- (3,1.5);
                \clip (0,-0.5) rectangle (1,2);
                \draw[opacity=0.2,line width=26pt, line cap=round]
                    (-0.4,-0.2) -- (1.3,0.65);
            \end{scope}
            \begin{scope}
                \clip (2,2) -- (4,-1) -- (4,2) -- (2,2);
                \draw[opacity=0.2,line width=26pt, line cap=round]
                    (1.3,0.65) -- (3,1.5);
            \end{scope}
            \draw[thick] (0,0.51) -- (0,-0.505);
            \draw[thick, red]
                (3,1.5) -- (0,0) node[circlecross]{};
            \node[text=red] (L) at (-0.5,0.45) {\footnotesize $L_{p,q}$};
            \draw[dashed] (3,1.5) node[cross]{} -- (1.11*3,1.11*1.5);
        \end{scope}
    
        \begin{scope}[shift={(-1,0)}]
            \node at (-0.5,1.7) {(b)};
            \begin{scope}
                \clip (0,-0.5) rectangle (4,2);
                \draw[opacity=0.2,line width=26pt, line cap=round]
                    (-0.4,-0.2) -- (3,1.5);
                \clip (0,-0.5) rectangle (1,2);
                \draw[opacity=0.2,line width=26pt, line cap=round]
                    (-0.4,-0.2) -- (1.3,0.65);
            \end{scope}
            \begin{scope}
                \clip (2,2) -- (4,-1) -- (4,2) -- (2,2);
                \draw[opacity=0.2,line width=26pt, line cap=round]
                    (1.3,0.65) -- (3,1.5);
            \end{scope}
            \draw[thick] (0,0.51) -- (0,-0.505);
            \draw[thick, red]
                (0.8,0.4) -- (0,0) node[circlecross]{};
            \node[text=red] (L) at (-0.5,0.45) {\footnotesize $L_{p,q}$};
            \draw[dashed] (0.8,0.4) node[cross]{} -- (1.11*3,1.11*1.5);
        \end{scope}
        
        \begin{scope}[shift={(4,0)}]
            \node at (-0.5,1.7) {(c)};
            \filldraw[opacity=0.2] (0,1) -- (0,0) -- (4,1) -- cycle;
            \draw[thick] (0,1) -- (0,0) -- (4,1);
            \draw[dashed] (0.8,0.4) node[cross] {} -- (0,0);
            \draw[thick, red]
                (0.8,0.4) -- (0,0) node[circlecross]{};
            \node[text=red] (L) at (-0.5,0.45) {\footnotesize $L_{p,q}$};
            \draw[thick,dash pattern=on 7pt off 3pt] (0,1) -- (4,1);
        \end{scope}
    \end{tikzpicture}
    \caption{(a) The base diagram of the Lagrangian fibration constructed on a neighbourhood of a Lagrangian pinwheel $L_{p,q}\subseteq (X,\omega)$. (b) The Lagrangian fibration on the same subset after a nodal slide and (c) a restriction of the neighbourhood given by taking the preimage of $\ATF_{p,q}(\epsilon)$, which is contained in (b). To see that (b) contains a base of the form $\ATF_{p,q}(\epsilon)$ change the branch cut in (b).}
    \label{fig:ATF_neigh_2}
\end{figure}
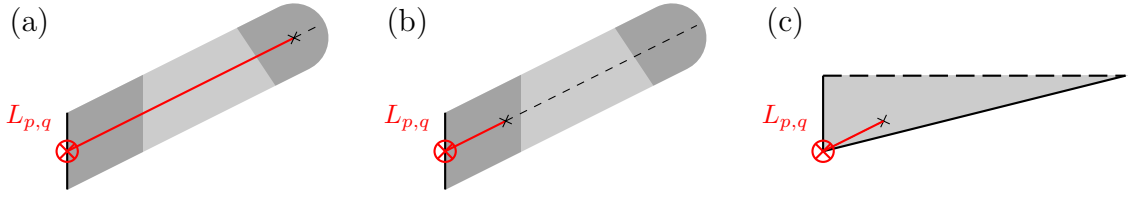
This embedding extends to an embedding of a small cotangent disc bundle of the disc such that the zero section is mapped to the embedded capping disc $\Delta$.
We denote the image of this embedding by $U_\Delta$.
The union of the neighbourhoods $U:=U_\Sigma \cup U_\Delta$ defines a neighbourhood of $L_{p,q}$.
Moreover, this construction carries a Lagrangian fibration, which we now explain.
On $U_\Delta$ we have a fibration as shown in \cref{fig:ATF_neigh_1} (b) and after a nodal trade we see that $U_\Delta$ admits a fibration as shown in \cref{fig:ATF_neigh_2} (c).\footnote{Note that this is a non-trivial operation. A priori it is not clear that $D_\text{Aur}$ can be assumed to coincide with $D_\text{st}$.
However, this is guaranteed to be the case after a Hamiltonian isotopy by \cref{th:Eliashberg_Polterovich}.}
On $U_\Sigma$ we have a Lagrangian fibration too, as explained in the beginning of the proof.
Now, after shrinking $U$, we obtain a fibration as shown in \cref{fig:ATF_neigh_2} (a) by gluing the two fibrations on the overlapping part. 

We are not done: the fibres over the lighter-gray shaded regions are annuli and not tori.
The goal is now to perform a nodal slide in order to pull the node over the region over which the Lagrangian fibres are only annuli and not full tori: this corresponds to the change of base diagram from \cref{fig:ATF_neigh_2} (a) to (b).
As this is a non-standard situation we have to use the interpretation of a nodal slide that Groman and Varolgunes give in \cite[Section 7.3/7.4]{GroVar23}.
The idea is to consider the preimage of the eigenline in \cref{fig:ATF_neigh_2}.
This preimage carries a Hamiltonian $S^1$-action with stabilizer given by the focus-focus singularity. 
If we now quotient the preimage by this Hamiltonian $S^1$-action, we obtain a situation as shown in \cref{fig:ATF_neigh_slide} (a).
Then changing the foliation as illustrated in the change from \cref{fig:ATF_neigh_slide} (a) to (b) changes the fibration on $U$ from \cref{fig:ATF_neigh_2} (a) to (b).
Then after shrinking $U$ again we obtain a neighbourhood $U_L$ of $L_{p,q}$ as shown in \cref{fig:ATF_neigh_2} (c).
\end{proof}

\begin{figure}[ht]
  \centering
  \begin{tikzpicture}
    \begin{scope}[shift={(-7.5,0)}, scale=0.9]
        \node at (-0.7,1.75) {(a)};
        \draw
            (0,2) -- (8,2)
            (0,0) -- (1,0)
            (7,0) -- (8,0)
            (2,0.8) -- (6,0.8)
            (1,0) to[out=0, in=180] (2,0.8)
            (6,0.8) to[out=0, in=180] (7,0);
        \draw[dashed]
            (2,1.2) -- (6,1.2)
            (0.7,0) to[out=0, in=180] (2,1.2)
            (6,1.2) to[out=0, in=180] (7.3,0);
        \draw
            (0.5,0) arc (90:-90:0.3 and -1)
            (7.5,0) arc (90:-90:0.3 and -1)
            (8,0) arc (90:-90:0.3 and -1)
            (0,0) arc (90:-90:0.3 and -1)
            (0,0) arc (90:-90:-0.3 and -1);
        \draw[dashed]
            (0.5,0) arc (90:-90:-0.3 and -1)
            (8,0) arc (90:-90:-0.3 and -1)
            (7.5,0) arc (90:-90:-0.3 and -1);
        \begin{scope}
            \clip (3,0.8) rectangle (5,2);
            \draw (4,0) arc (90:-90:0.3 and -1);
        \end{scope}
        \begin{scope}
            \clip (3,1.2) rectangle (5,2);
            \draw[dashed] (4,0) arc (90:-90:-0.3 and -1);
        \end{scope}
        \draw[red] (0.738,1.6) -- (8.238,1.6);
        \draw (0.738,1.6) node[cross]{};
    \end{scope}

    \begin{scope}[shift={(1.5,0)}, scale=0.9]
    \node at (-0.7,1.75) {(b)};
        \draw
            (0,2) -- (8,2)
            (0,0) -- (1,0)
            (7,0) -- (8,0)
            (2,0.8) -- (6,0.8)
            (1,0) to[out=0, in=180] (2,0.8)
            (6,0.8) to[out=0, in=180] (7,0);
        \draw[dashed]
            (2,1.2) -- (6,1.2)
            (0.7,0) to[out=0, in=180] (2,1.2)
            (6,1.2) to[out=0, in=180] (7.3,0);
        \draw
            (8,0) arc (90:-90:0.3 and -1)
            (0,0) arc (90:-90:0.3 and -1)
            (0,0) arc (90:-90:-0.3 and -1);
        \draw[dashed]
            (0.5,0) arc (90:-90:-0.3 and -1)
            (8,0) arc (90:-90:-0.3 and -1)
            (7.5,0) arc (90:-90:-0.3 and -1);
        \draw 
            (4,2) to[out=0, in=110, looseness=0.5] 
            (4.02,1.98) to[out=290, in=90, looseness=0.3]
            (0.6,1.6) to[out=270, in=110, looseness=0.3] 
            (4.27,1.2) to[out=290, in=90, looseness=0.5]
            (4.3,0.8);
        \draw 
            (7.5,2) to[out=0, in=110, looseness=0.5] 
            (7.6,1.9) to[out=290, in=0, looseness=0.1]
            (2,1.7) to[out=180, in=90, looseness=0.2]
            (0.738,1.6) to[out=270, in=180, looseness=0.2]
            (2,1.5) to[out=0,in=110, looseness=0.1]
            (7.7,1.3) to[out=290, in=0, looseness=0.5]
            (7.5,0);
        \draw 
            (0.5,2) to[out=0, in=90, looseness=0.9]
            (0.65,1.85) to[out=270, in=90, looseness=0.9]
            (0.45,1.6) to[out=270, in=110, looseness=0.9]
            (0.75,1.3) to[out=290, in=0, looseness=0.5]
            (0.5,0);
        \begin{scope}
            \clip (3,1.2) rectangle (5,2);
            \draw[dashed] (4,0) arc (90:-90:-0.3 and -1);
        \end{scope}
        \draw[red] (0.738,1.6) -- (8.238,1.6);
        \draw (0.738,1.6) node[cross]{};
    \end{scope}

  \end{tikzpicture}
  \caption{(a) The reduced space. (b) The reduced space with the changed foliation that induces the nodal slide.}
  \label{fig:ATF_neigh_slide}
\end{figure}
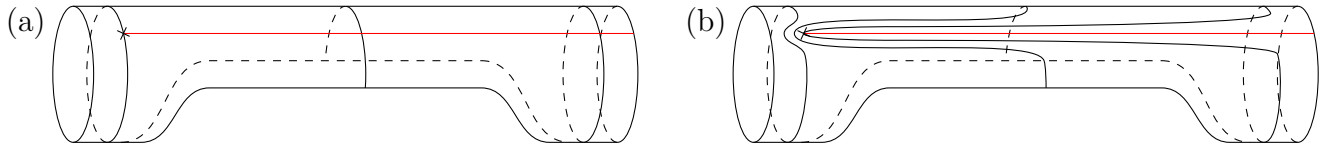

\begin{remark}
    Note that this is a refinement of the Weinstein-type theorem that Khodorovskiy proves in \cite[Section 3]{Kho13}.
    We emphasize that this theorem allows us to pass freely from embeddings of Lagrangian pinwheels to embeddings of small rational homology balls.
\end{remark}

\emergencystretch=1em
\printbibliography

@misc{AdHa25,
    title={Pinwheels in symplectic rational and ruled surfaces and non-squeezing of rational homology balls}, 
    author={N. Adaloglou and J. Hauber},
    year={2025},
    eprint={2503.16250},
    archivePrefix={arXiv},
    primaryClass={math.SG},
    url={https://arxiv.org/abs/2503.16250}, 
}

@misc{ABEHS25,
    title={Markov staircases}, 
    author={N. Adaloglou and J. Brendel and J. Evans and J. Hauber and F. Schlenk},
    year={2025},
    eprint={2509.03224},
    archivePrefix={arXiv},
    primaryClass={math.SG},
    url={https://arxiv.org/abs/2509.03224}, 
}

@article{Abreu98,
    author = {Abreu, M.},
    title = {Topology of symplectomorphism groups of {{\(S^2\times S^2\)}}},
    fjournal = {Inventiones Mathematicae},
    journal = {Invent. Math.},
    issn = {0020-9910},
    volume = {131},
    number = {1},
    pages = {1--23},
    year = {1998},
    language = {English},
    doi = {10.1007/s002220050196},
    zbMATH = {1146048},
    Zbl = {0902.53025}
}

@article{AbMcD00,
    author = {Abreu, M. and McDuff, D.},
    title = {Topology of symplectomorphism groups of rational ruled surfaces},
    fjournal = {Journal of the American Mathematical Society},
    journal = {J. Am. Math. Soc.},
    issn = {0894-0347},
    volume = {13},
    number = {4},
    pages = {971--1009},
    year = {2000},
    language = {English},
    doi = {10.1090/S0894-0347-00-00344-1},
    zbMATH = {1498870},
    Zbl = {0965.57031}
}

@article{Hi04,
    author = {Hind, R.},
    title = {Lagrangian spheres in {$S^2\times S^2$}},
    fjournal = {Geometric and Functional Analysis. GAFA},
    journal = {Geom. Funct. Anal.},
    issn = {1016-443X},
    volume = {14},
    number = {2},
    pages = {303--318},
    year = {2004},
    language = {English},
    doi = {10.1007/s00039-004-0459-6},
    zbMATH = {2121744},
    Zbl = {1066.53129},
}

@article{AK18,
    author = {Abouzaid, M. and Kragh, T.},
    title = {Simple homotopy equivalence of nearby {Lagrangians}},
    fjournal = {Acta Mathematica},
    journal = {Acta Math.},
    issn = {0001-5962},
    volume = {220},
    number = {2},
    pages = {207--237},
    year = {2018},
    language = {English},
    doi = {10.4310/ACTA.2018.v220.n2.a1},
    zbMATH = {6925264},
    Zbl = {1396.53104},
}

@book{Ev23:Book, 
    place={Cambridge}, 
    series={London Mathematical Society Student Texts}, 
    title={Lectures on Lagrangian Torus Fibrations}, 
    DOI={10.1017/9781009372671},
    publisher={Cambridge University Press}, 
    author={Evans, J.}, 
    year={2023}, 
    collection={London Mathematical Society Student Texts},
}

@misc{Ev24:KIAS,
 author = {Evans, J.},
 title = {{KIAS} {Lectures} on {Symplectic} {Aspects} of {Degenerations}},
 year = {2024},
 howpublished = {Preprint, {arXiv}:2403.03519 [math.{SG}]},
 url = {https://arxiv.org/abs/2403.03519},
 arXiv = {arXiv:2403.03519}
}

@article{HPW,
    author = "R. Hind and M. Pinsonnault and W. Wu",
    title = "{Symplectomophism groups of non-compact manifolds, orbifold balls, and a space of Lagrangians}",
    journal = "J. Symplectic Geom.",
    volume = "14(1)",
    pages = "203--226 ",
    year = "2016",
}

@misc{Bu25,
    title={Lagrangian Spheres and Cyclic Quotient T-singularities}, 
    author={M. R. Buck},
    year={2025},
    eprint={2509.18976},
    archivePrefix={arXiv},
    primaryClass={math.SG},
    url={https://arxiv.org/abs/2509.18976}, 
}

@article{Wu14,
    author = {Wu, W.},
    title = {Exact {Lagrangians} in $A_n$-surface singularities},
    fjournal = {Mathematische Annalen},
    journal = {Math. Ann.},
    issn = {0025-5831},
    volume = {359},
    number = {1-2},
    pages = {153--168},
    year = {2014},
    language = {English},
    doi = {10.1007/s00208-013-0993-3},
    zbMATH = {6312288},
    Zbl = {1315.53099},
}

@article{LiWu12,
    author = {Li, T.-J. and Wu, W.},
    title = {Lagrangian spheres, symplectic surfaces and the symplectic mapping class group},
    fjournal = {Geometry \& Topology},
    journal = {Geom. Topol.},
    issn = {1465-3060},
    volume = {16},
    number = {2},
    pages = {1121--1169},
    year = {2012},
    language = {English},
    doi = {10.2140/gt.2012.16.1121},
    zbMATH = {6068622},
    Zbl = {1253.53073},
}

@misc{Kho13,
    title={Symplectic Rational Blow-up}, 
    author={T. Khodorovskiy},
    year={2013},
    eprint={1303.2581},
    archivePrefix={arXiv},
    primaryClass={math.SG},
    url={https://arxiv.org/abs/1303.2581}, 
}

@article{Ad24,
    author = {Adaloglou, N.},
    title = {Embeddings and disjunction of Lagrangian pinwheels via rational blow-ups},
    fjournal = {The Journal of Symplectic Geometry},
    journal = {J. Symplectic Geom.},
    issn = {1527-5256},
    volume = {24},
    number = {1},
    pages = {105--129},
    year = {2026},
    language = {English},
    doi = {10.4310/JSG.260423234950},
   
}

@article{AbCoGuKr25,
    author = {Abouzaid, M. and Courte, S. and Guillermou, S. and Kragh, T.},
    title = {Twisted generating functions and the nearby {Lagrangian} conjecture},
    fjournal = {Duke Mathematical Journal},
    journal = {Duke Math. J.},
    issn = {0012-7094},
    volume = {174},
    number = {5},
    pages = {949--1011},
    year = {2025},
    language = {English},
    doi = {10.1215/00127094-2024-0052},
    url = {projecteuclid.org/journals/duke-mathematical-journal/volume-174/issue-5/Twisted-generating-functions-and-the-nearby-Lagrangian-conjecture/10.1215/00127094-2024-0052.full},
    zbMATH = {8050362}
}

@misc{AbAGCoKr,
    title={Normal invariant of nearby Lagrangians via twisted derivative}, 
    author={Abouzaid, M. and Álvarez-Gavela, D. and Courte, S. and Kragh, T.},
    year={2025},
    eprint={2505.12515},
    archivePrefix={arXiv},
    primaryClass={math.SG},
    url={https://arxiv.org/abs/2505.12515}, 
}

@article{Ba78,
 author = {Banyaga, A.},
 title = {Sur la structure du groupe des diff{\'e}omorphismes qui preservent une forme symplectique},
 fjournal = {Commentarii Mathematici Helvetici},
 journal = {Comment. Math. Helv.},
 issn = {0010-2571},
 volume = {53},
 pages = {174--227},
 year = {1978},
 language = {French},
 doi = {10.1007/BF02566074},
 url = {https://eudml.org/doc/139732},
 zbMATH = {3610489},
 Zbl = {0393.58007}
}

@article{Sei99,
    author = {Seidel, P.},
    title = {Lagrangian two-spheres can be symplectically knotted},
    fjournal = {Journal of Differential Geometry},
    journal = {J. Differ. Geom.},
    issn = {0022-040X},
    volume = {52},
    number = {1},
    pages = {145--171},
    year = {1999},
    language = {English},
    doi = {10.4310/jdg/1214425219},
    zbMATH = {1465075},
    Zbl = {1032.53068}
}

@article{ShSmi20,
    author = {Shevchishin, V. and Smirnov, G.},
    title = {Symplectic triangle inequality},
    fjournal = {Proceedings of the American Mathematical Society},
    journal = {Proc. Am. Math. Soc.},
    issn = {0002-9939},
    volume = {148},
    number = {4},
    pages = {1389--1397},
    year = {2020},
    language = {English},
    doi = {10.1090/proc/14842},
    zbMATH = {7176128},
    Zbl = {1433.14031}
}

@article{EvSm20,
    author = {Evans, J. and Smith, I.},
    title = {Bounds on {Wahl} singularities from symplectic topology},
    fjournal = {Algebraic Geometry},
    journal = {Algebr. Geom.},
    issn = {2313-1691},
    volume = {7},
    number = {1},
    pages = {59--85},
    year = {2020},
    language = {English},
    doi = {10.14231/AG-2020-003},
    zbMATH = {7262961},
    Zbl = {1458.14054}
}

@Article{Ow20:nonsymp,
    Author = {Owens, B.},
    Title = {Smooth, nonsymplectic embeddings of rational balls in the complex projective plane},
    FJournal = {The Quarterly Journal of Mathematics},
    Journal = {Q. J. Math.},
    ISSN = {0033-5606},
    Volume = {71},
    Number = {3},
    Pages = {997--1007},
    Year = {2020},
    Language = {English},
    DOI = {10.1093/qmathj/haaa013},
    zbMATH = {7277648},
    Zbl = {1470.53012}
}

@Article{Ow2018:equiemb,
    Author = {Owens, B.},
    Title = {Equivariant embeddings of rational homology balls},
    FJournal = {The Quarterly Journal of Mathematics},
    Journal = {Q. J. Math.},
    ISSN = {0033-5606},
    Volume = {69},
    Number = {3},
    Pages = {1101--1121},
    Year = {2018},
    Language = {English},
    DOI = {10.1093/qmath/hay016},
    URL = {eprints.gla.ac.uk/158935/19/158935.pdf},
    zbMATH = {6945159},
    Zbl = {1405.57007}
}

@InCollection{LiPa24,
    Author = {Lisca, P. and Parma, A.},
    Title = {On almost complex embeddings of rational homology balls},
    BookTitle = {Frontiers in geometry and topology. Summer school and research conference, The Abdus Salam International Centre for Theoretical Physics, Trieste, Italy, August 1--12, 2022},
    ISBN = {978-1-4704-7087-6},
    Pages = {183--193},
    Year = {2024},
    Publisher = {Providence, RI: American Mathematical Society (AMS)},
    Language = {English},
    DOI = {10.1090/pspum/109/01995},
    zbMATH = {7936424}
}

@misc{EtHyPi2023,
    title={Small symplectic caps and embeddings of homology balls in the complex projective plane}, 
    author={Etnyre, J. B. and Min, H. and Piccirillo, L. and Roy, A.},
    year={2023},
    eprint={2305.16207},
    archivePrefix={arXiv},
    primaryClass={math.GT},
    url={https://arxiv.org/abs/2305.16207}, 
}

@Article{LiPa22,
    Author = {Lisca, P. and Parma, A.},
    Title = {On {Stein} rational balls smoothly but not symplectically embedded in {{\(\mathbb{CP}^2\)}}},
    FJournal = {Bulletin of the London Mathematical Society},
    Journal = {Bull. Lond. Math. Soc.},
    ISSN = {0024-6093},
    Volume = {54},
    Number = {3},
    Pages = {949--960},
    Year = {2022},
    Language = {English},
    DOI = {10.1112/blms.12607},
    zbMATH = {7729878},
    Zbl = {1522.57056}
}

@misc{GoOw25,
 author = {Golla, M. and Owens, B.},
 title = {The {Farey} tree and embeddings of lens spaces and rational balls in {$\mathbb{CP}^2$}},
 year = {2025},
 howpublished = {Preprint, {arXiv}:2512.09183 [math.{GT}]},
 url = {https://arxiv.org/abs/2512.09183},
 arXiv = {arXiv:2512.09183}
}

@book{McDSa:Intro,
    author = {McDuff, D. and Salamon, D.},
    title = {Introduction to symplectic topology},
    edition = {3rd edition},
    fseries = {Oxford Graduate Texts in Mathematics},
    series = {Oxf. Grad. Texts Math.},
    volume = {27},
    isbn = {978-0-19-879489-9},
    year = {2016},
    publisher = {Oxford: Oxford University Press},
    language = {English},
    zbMATH = {6638013},
    Zbl = {1380.53003},
}

@article{Se20,
    author = {Seidel, P.},
    title = {Graded {Lagrangian} submanifolds},
    fjournal = {Bulletin de la Soci{\'e}t{\'e} Math{\'e}matique de France},
    journal = {Bull. Soc. Math. Fr.},
    issn = {0037-9484},
    volume = {128},
    number = {1},
    pages = {103--149},
    year = {2000},
    language = {English},
    doi = {10.24033/bsmf.2365},
    url = {https://eudml.org/doc/87820},
    zbMATH = {1463890},
    Zbl = {0992.53059},
}

@article{Ev14,
    author = {Evans, J.},
    title = {Symplectic mapping class groups of some {Stein} and rational surfaces},
    fjournal = {The Journal of Symplectic Geometry},
    journal = {J. Symplectic Geom.},
    issn = {1527-5256},
    volume = {9},
    number = {1},
    pages = {45--82},
    year = {2011},
    language = {English},
    doi = {10.4310/JSG.2011.v9.n1.a4},
    zbMATH = {5957820},
    Zbl = {1242.58004}
}

@misc{KeSmWe26,
    title={Splitting symplectic monodromy}, 
    author={A. Keating and I. Smith and M. Wemyss},
    year={2026},
    eprint={2601.20438},
    archivePrefix={arXiv},
    primaryClass={math.SG},
    url={https://arxiv.org/abs/2601.20438}, 
}

@misc{Kho13Bounds,
    title={Bounds on Embeddings of Rational Homology Balls in Symplectic 4-manifolds}, 
    author={T. Khodorovskiy},
    year={2013},
    eprint={1307.4321},
    archivePrefix={arXiv},
    primaryClass={math.SG},
    url={https://arxiv.org/abs/1307.4321}, 
}

@article{BrSch24,
    author = {Brendel, J. and Schlenk, F.},
    title = {Pinwheels as {Lagrangian} barriers},
    fjournal = {Communications in Contemporary Mathematics},
    journal = {Commun. Contemp. Math.},
    issn = {0219-1997},
    volume = {26},
    number = {5},
    pages = {21},
    note = {Id/No 2350020},
    year = {2024},
    language = {English},
    doi = {10.1142/S0219199723500207},
    zbMATH = {7838060},
    Zbl = {1547.53093}
}

@article{FuSeSm08,
    author = {Fukaya, K. and Seidel, P. and Smith, I.},
    title = {Exact {Lagrangian} submanifolds in simply-connected cotangent bundles},
    fjournal = {Inventiones Mathematicae},
    journal = {Invent. Math.},
    issn = {0020-9910},
    volume = {172},
    number = {1},
    pages = {1--27},
    year = {2008},
    language = {English},
    doi = {10.1007/s00222-007-0092-8},
    zbMATH = {5261615},
    Zbl = {1140.53036}
}

@misc{Kar18,
      title={Microlocal Sheaves on Pinwheels}, 
      author={D. Karabas},
      year={2018},
      eprint={1810.09021},
      archivePrefix={arXiv},
      primaryClass={math.SG},
      url={https://arxiv.org/abs/1810.09021}, 
}

@article{Kra13,
    author = {Kragh, T.},
    title = {Parametrized ring-spectra and the nearby {Lagrangian} conjecture},
    fjournal = {Geometry \& Topology},
    journal = {Geom. Topol.},
    issn = {1465-3060},
    volume = {17},
    number = {2},
    pages = {639--731},
    year = {2013},
    language = {English},
    doi = {10.2140/gt.2013.17.639},
    zbMATH = {6172908},
    Zbl = {1267.53081}
}

@article{EvSm18,
    author = {Evans, J. and Smith, I.},
    title = {Markov numbers and {Lagrangian} cell complexes in the complex projective plane},
    fjournal = {Geometry \& Topology},
    journal = {Geom. Topol.},
    issn = {1465-3060},
    volume = {22},
    number = {2},
    pages = {1143--1180},
    year = {2018},
    language = {English},
    doi = {10.2140/gt.2018.22.1143},
    url = {eprints.lancs.ac.uk/id/eprint/132430/1/1606.08656.pdf},
    zbMATH = {6828606},
    Zbl = {1381.53159},
}

@article{LeMa14,
    author = {Lekili, Y. and Maydanskiy, M.},
    title = {The symplectic topology of some rational homology balls},
    fjournal = {Commentarii Mathematici Helvetici},
    journal = {Comment. Math. Helv.},
    issn = {0010-2571},
    volume = {89},
    number = {3},
    pages = {571--596},
    year = {2014},
    language = {English},
    doi = {10.4171/CMH/327},
    zbMATH = {6361413},
    Zbl = {1315.53098},
}

@article{Sym98:RatBlo,
    author = {Symington, M.},
    title = {Symplectic rational blowdowns},
    fjournal = {Journal of Differential Geometry},
    journal = {J. Differ. Geom.},
    issn = {0022-040X},
    volume = {50},
    number = {3},
    pages = {505--518},
    year = {1998},
    language = {English},
    doi = {10.4310/jdg/1214424968},
    zbMATH = {1368011},
    Zbl = {0935.57035},
}

@article{Sym01:GenRatBlo,
    author = "M. Symington",
    title = "{Generalized symplectic rational blowdowns}",
    journal = "Algebr. Geom. Topol.",
    volume = "1",
    pages = "503--518",
    year = "2001",
}

@incollection{Sym02:Fourtwo,
    AUTHOR = {Symington, M.},
    TITLE = {Four dimensions from two in symplectic topology},
    BOOKTITLE = {Topology and geometry of manifolds ({A}thens, {GA}, 2001)},
    SERIES = {Proc. Sympos. Pure Math.},
    VOLUME = {71},
     PAGES = {153--208},
    PUBLISHER = {Amer. Math. Soc., Providence, RI},
    YEAR = {2003},
    ISBN = {0-8218-3507-6},
    MRCLASS = {53D35 (53D20 55R55 57R17)},
    MRNUMBER = {2024634},
    MRREVIEWER = {Vicente Mu\~{n}oz},
    DOI = {10.1090/pspum/071/2024634},
    URL = {https://doi.org/10.1090/pspum/071/2024634},
}

@article{Dim17,
  author  = {Dimitroglou Rizell, G.},
  title   = {The classification of Lagrangians nearby the Whitney immersion},
  journal = {Geometry \& Topology},
  volume  = {23},
  year    = {2019},
  pages   = {3367--3458},
  doi     = {10.2140/gt.2019.23.3367},
  url     = {https://msp.org/gt/2019/23-7/p04.xhtml}
}

@article{DRGI16,
    author = {Dimitroglou Rizell, G. and Goodman, E. and Ivrii, A.},
    title = {Lagrangian isotopy of tori in {$S^2 \times S^2$} and {$\mathbb{C} P^2$}},
    fjournal = {Geometric and Functional Analysis. GAFA},
    journal = {Geom. Funct. Anal.},
    issn = {1016-443X},
    volume = {26},
    number = {5},
    pages = {1297--1358},
    year = {2016},
    language = {English},
    doi = {10.1007/s00039-016-0388-1},
    zbMATH = {6668657},
    Zbl = {1358.53089},
}

@article{Ad25,
    author = {Adaloglou, N.},
    title = {Uniqueness of {Lagrangians} in {$T^*\mathbb{R}P^2$}},
    fjournal = {Annales Math{\'e}matiques du Qu{\'e}bec},
    journal = {Ann. Math. Qu{\'e}.},
    issn = {2195-4755},
    volume = {49},
    number = {1},
    pages = {215--222},
    year = {2025},
    language = {English},
    doi = {10.1007/s40316-024-00238-3},
    zbMATH = {8030618},
    Zbl = {1564.53074},
}

@misc{UrZu25:BirMarkov,
    title={The birational geometry of Markov numbers}, 
    author={G. Urz\'ua and J. P. Z\'u{\~n}iga},
    year={2025},
    eprint={2310.17957},
    archivePrefix={arXiv},
    primaryClass={math.AG},
    url={https://arxiv.org/abs/2310.17957}, 
}

@misc{Ur25:Lecture,
    title={Negative continued Fractions in Birational Geometry: A guide to Degenerations of Surfaces with Wahl Singularities}, 
    author={G. Urz\'ua},
    year={2025},
    url={https://www.mat.uc.cl/~urzua/libro.pdf}, 
}

@article{haktelurz,
    author = {Hacking, P. and Tevelev, J. and Urz{\'u}a, G.},
    title = {Flipping surfaces},
    fjournal = {Journal of Algebraic Geometry},
    journal = {J. Algebr. Geom.},
    issn = {1056-3911},
    volume = {26},
    number = {2},
    pages = {279--345},
    year = {2017},
    language = {English},
    doi = {10.1090/jag/682},
    zbMATH = {6678659},
    Zbl = {1360.14055}
}

@article{Au97,
    author = {Auroux, D.},
    title = {Asymptotically holomorphic families of symplectic submanifolds},
    fjournal = {Geometric and Functional Analysis. GAFA},
    journal = {Geom. Funct. Anal.},
    issn = {1016-443X},
    volume = {7},
    number = {6},
    pages = {971--995},
    year = {1997},
    language = {English},
    doi = {10.1007/s000390050033},
    zbMATH = {1123717},
    Zbl = {0912.53020},
}

@article{BhuOno12,
    author = {Bhupal, M. and Ono, K.},
    title = {Symplectic fillings of links of quotient surface singularities},
    fjournal = {Nagoya Mathematical Journal},
    journal = {Nagoya Math. J.},
    issn = {0027-7630},
    volume = {207},
    pages = {1--45},
    year = {2012},
    language = {English},
    zbMATH = {6081392},
    Zbl = {1258.53088},
}

@article{PaPaShUr18,
    author = {Park, H. and Park, J. and Shin, D. and Urz{\'u}a, G.},
    title = {Milnor fibers and symplectic fillings of quotient surface singularities},
    fjournal = {Advances in Mathematics},
    journal = {Adv. Math.},
    issn = {0001-8708},
    volume = {329},
    pages = {1156--1230},
    year = {2018},
    language = {English},
    doi = {10.1016/j.aim.2018.03.002},
    zbMATH = {6863466},
    Zbl = {1390.14018},
}

@article{KoShBa88,
    author = {Koll{\'a}r, J. and Shepherd-Barron, N. I.},
    title = {Threefolds and deformations of surface singularities},
    fjournal = {Inventiones Mathematicae},
    journal = {Invent. Math.},
    issn = {0020-9910},
    volume = {91},
    number = {2},
    pages = {299--338},
    year = {1988},
    language = {English},
    doi = {10.1007/BF01389370},
    url = {https://eudml.org/doc/143542},
    zbMATH = {4045919},
    Zbl = {0642.14008}
}

@article{OaUs16,
    author = {Oakley, J. and Usher, M.},
    title = {On certain {Lagrangian} submanifolds of {{\(S^2\times S^2\)}} and {{\(\mathbb{C}\operatorname{P}^n\)}}},
    fjournal = {Algebraic \& Geometric Topology},
    journal = {Algebr. Geom. Topol.},
    issn = {1472-2747},
    volume = {16},
    number = {1},
    pages = {149--209},
    year = {2016},
    language = {English},
    doi = {10.2140/agt.2016.16.149},
    zbMATH = {6553357},
    Zbl = {1335.53105}
}

@article{Wu15,
    author = {Wu, W.},
    title = {On an exotic {Lagrangian} torus in {{\(\mathbb{C}P^2\)}}},
    fjournal = {Compositio Mathematica},
    journal = {Compos. Math.},
    issn = {0010-437X},
    volume = {151},
    number = {7},
    pages = {1372--1394},
    year = {2015},
    language = {English},
    doi = {10.1112/S0010437X14007945},
    zbMATH = {6486865},
    Zbl = {1325.53105}
}

@book{McDSal02:Jcurves,
    AUTHOR = {McDuff, D. and Salamon, D.},
    TITLE = {$J$-holomorphic curves and symplectic topology},
    SERIES = {American Mathematical Society Colloquium Publications},
    VOLUME = {52},
    PUBLISHER = {American Mathematical Society, Providence, RI},
    YEAR = {2004},
    PAGES = {xii+669},
    ISBN = {0-8218-3485-1},
    MRCLASS = {53D45 (32Q65 53D40 57R17)},
    MRNUMBER = {2045629},
    MRREVIEWER = {Ignasi Mundet-Riera},
    DOI = {10.1090/coll/052},
    URL = {https://doi.org/10.1090/coll/052},
}

@misc{CrGaMaMcD25,
    author = {Cristofaro-Gardiner, D. and Magill, N. and McDuff, D.},
    title = {Curvy points, the perimeter, and the complexity of convex toric domains},
    year = {2025},
    howpublished = {Preprint, {arXiv}:2506.23498 [math.{SG}]},
    url = {https://arxiv.org/abs/2506.23498},
    arXiv = {arXiv:2506.23498},
}

@article{Gr85,
    author = "M. Gromov",
    title = "Pseudo holomorphic curves in symplectic manifolds",
    journal = "Inventiones mathematicae",
    volume = "82",
    pages = "307--347",
    year = "1985",
}

@book{Wen18:Jhollow,
    author = {Wendl, C.},
    title = {Holomorphic curves in low dimensions. From symplectic ruled surfaces to planar contact manifolds},
    fseries = {Lecture Notes in Mathematics},
    series = {Lect. Notes Math.},
    issn = {0075-8434},
    volume = {2216},
    isbn = {978-3-319-91369-8},
    year = {2018},
    publisher = {Berlin: Springer},
    language = {English},
    doi = {10.1007/978-3-319-91371-1},
    zbMATH = {6875550},
    Zbl = {1432.57055},
}

@article{McDSie25:singalg,
    author = {McDuff, D. and Siegel, K.},
    title = {Singular algebraic curves and infinite symplectic staircases},
    fjournal = {Inventiones Mathematicae},
    journal = {Invent. Math.},
    issn = {0020-9910},
    volume = {242},
    number = {2},
    pages = {387--459},
    year = {2025},
    language = {English},
    doi = {10.1007/s00222-025-01359-4},
    zbMATH = {8103859}
}

@article{McDSie25:superpot,
    author = {McDuff, D. and Siegel, K.},
    title = {Ellipsoidal superpotentials and singular curve counts},
    fjournal = {IMRN. International Mathematics Research Notices},
    journal = {Int. Math. Res. Not.},
    issn = {1073-7928},
    volume = {2025},
    number = {21},
    pages = {37},
    note = {Id/No rnaf285},
    year = {2025},
    language = {English},
    doi = {10.1093/imrn/rnaf285},
    zbMATH = {8115445}
}

@unpublished{Ha25,
    author = {Adaloglou, N. and Hauber, J.},
    title  = {Symplectic staircases for domains in cotangent bundles},
    note   = {Forthcoming}
}

@article{Ab12,
    author = {Abouzaid, M.},
    title = {Nearby {Lagrangians} with vanishing {Maslov} class are homotopy equivalent},
    fjournal = {Inventiones Mathematicae},
    journal = {Invent. Math.},
    issn = {0020-9910},
    volume = {189},
    number = {2},
    pages = {251--313},
    year = {2012},
    language = {English},
    doi = {10.1007/s00222-011-0365-0},
    zbMATH = {6110397},
    Zbl = {1261.53077}
}

@article{GroVar23,
    author = {Groman, Y. and Varolgunes, U.},
    title = {Locality of relative symplectic cohomology for complete embeddings},
    fjournal = {Compositio Mathematica},
    journal = {Compos. Math.},
    issn = {0010-437X},
    volume = {159},
    number = {12},
    pages = {2551--2637},
    year = {2023},
    language = {English},
    doi = {10.1112/S0010437X23007492},
    zbMATH = {7755910},
    Zbl = {1527.53077}
}

@article{EliPol96,
     author = {Eliashberg, Y. and Polterovich, L.},
     title = {Local {Lagrangian} 2-knots are trivial},
     fjournal = {Annals of Mathematics. Second Series},
     journal = {Ann. Math. (2)},
     issn = {0003-486X},
     volume = {144},
     number = {1},
     pages = {61--76},
     year = {1996},
     language = {English},
     doi = {10.2307/2118583},
     zbMATH = {935301},
     Zbl = {0872.57030}
}

@misc{PolSch24,
     author = {Polterovich, L. and Schlenk, F.},
     title = {Lagrangian knots and unknots -- an essay.},
     year = {2024},
    eprinttype  = {arxiv},
    eprint      = {2406.15967},
    eprintclass = {math.SG},
     note = {With an appendix by G. Dimitroglou Rizell},
     url = {https://arxiv.org/abs/2406.15967},
     arXiv = {arXiv:2406.15967}
}

@article{Aud07,
     author = {Audin, M.},
     title = {Lagrangian skeletons, periodic geodesic flows and symplectic cuttings},
     fjournal = {Manuscripta Mathematica},
     journal = {Manuscr. Math.},
     issn = {0025-2611},
     volume = {124},
     number = {4},
     pages = {533--550},
     year = {2007},
     language = {English},
     doi = {10.1007/s00229-007-0134-y},
     zbMATH = {5248331},
     Zbl = {1132.53042}
}

@article{BoLiWu14,
     author = {Borman, M. S. and Li, T.-J. and Wu, W.},
     title = {Spherical {Lagrangians} via ball packings and symplectic cutting},
     fjournal = {Selecta Mathematica. New Series},
     journal = {Sel. Math., New Ser.},
     issn = {1022-1824},
     volume = {20},
     number = {1},
     pages = {261--283},
     year = {2014},
     language = {English},
     doi = {10.1007/s00029-013-0120-z},
     zbMATH = {6261571},
     Zbl = {1285.53074}
}

@article{HoLiSi97,
     author = {Hofer, H. and Lizan, V. and Sikorav, J.-C.},
     title = {On genericity for holomorphic curves in four-dimensional almost-complex manifolds},
     fjournal = {The Journal of Geometric Analysis},
     journal = {J. Geom. Anal.},
     issn = {1050-6926},
     volume = {7},
     number = {1},
     pages = {149--159},
     year = {1997},
     language = {English},
     doi = {10.1007/BF02921708},
     zbMATH = {1192507},
     Zbl = {0911.53014}
}

@article{BEHWZ03,
     author = {Bourgeois, F. and Eliashberg, Y. and Hofer, H. and Wysocki, K. and Zehnder, E.},
     title = {Compactness results in symplectic field theory},
     fjournal = {Geometry \& Topology},
     journal = {Geom. Topol.},
     issn = {1465-3060},
     volume = {7},
     pages = {799--888},
     year = {2003},
     language = {English},
     doi = {10.2140/gt.2003.7.799},
     url = {https://eudml.org/doc/128488},
     zbMATH = {2062477},
     Zbl = {1131.53312}
}

@article{CieMoh05,
     author = {Cieliebak, K. and Mohnke, K.},
     title = {Compactness for punctured holomorphic curves},
     fjournal = {The Journal of Symplectic Geometry},
     journal = {J. Symplectic Geom.},
     issn = {1527-5256},
     volume = {3},
     number = {4},
     pages = {589--654},
     year = {2005},
     language = {English},
     doi = {10.4310/JSG.2005.v3.n4.a5},
     zbMATH = {5121973},
     Zbl = {1113.53053}
}

@misc{AtMaWu25,
     author = {M. Atallah and C. Y. Mak and W. Wu},
     title = {{$C^0$}-rigidity of the {Hamiltonian} diffeomorphism group of symplectic rational surfaces},
     year = {2025},
     howpublished = {Preprint, {arXiv}:2508.20285 [math.{SG}]},
     url = {https://arxiv.org/abs/2508.20285},
     arXiv = {arXiv:2508.20285}
}

@incollection{ElGiHo00,
    author = {Eliashberg, Y. and Givental, A. and Hofer, H.},
    title = {Introduction to {Symplectic} {Field} {Theory}},
    booktitle = {GAFA 2000. Visions in mathematics---Towards 2000. Proceedings of a meeting, Tel Aviv, Israel, August 25--September 3, 1999. Part II.},
    pages = {560--673},
    year = {2000},
    publisher = {Basel: Birkh{\"a}user},
    language = {English},
    zbMATH = {1643843},
    Zbl = {0989.81114}
}

@article{Vi16,
    author = {Ferreira de Velloso Vianna, R.},
    title = {Infinitely many exotic monotone {Lagrangian} tori in {{\(\mathbb{C}\mathbb{P}^{2}\)}}},
    fjournal = {Journal of Topology},
    journal = {J. Topol.},
    issn = {1753-8416},
    volume = {9},
    number = {2},
    pages = {535--551},
    year = {2016},
    language = {English},
    doi = {10.1112/jtopol/jtw002},
    zbMATH = {6598132},
    Zbl = {1350.53102}
}

@article{Ler95,
    author = {Lerman, E.},
    title = {Symplectic cuts},
    fjournal = {Mathematical Research Letters},
    journal = {Math. Res. Lett.},
    issn = {1073-2780},
    volume = {2},
    number = {3},
    pages = {247--258},
    year = {1995},
    language = {English},
    doi = {10.4310/MRL.1995.v2.n3.a2},
    zbMATH = {817751},
    Zbl = {0835.53034}
}

\end{document}